\documentclass[11pt,reqno,twoside]{article}

\usepackage{dsfont}

\usepackage{braket}

\usepackage{amsmath}
\usepackage{amsthm} % http://tex.stackexchange.com/questions/130641/how-to-put-the-qed-symbol-of-a-proof-at-the-right-place-inside-align
\usepackage{amssymb}
\usepackage{amstext}
\usepackage{amsfonts}

\usepackage{array}

\usepackage[caption=false]{subfig}
\usepackage{pgfplots}

\usepackage{algorithmic}

\usepackage{graphicx,epstopdf}

\usepackage{amsopn}

\usepackage{xspace}
\usepackage{bold-extra}
\usepackage[most]{tcolorbox}

\newcounter{example}
\colorlet{texcscolor}{blue!50!black}
\colorlet{texemcolor}{red!70!black}
\colorlet{texpreamble}{red!70!black}
\colorlet{codebackground}{black!25!white!25}

\usepackage{caption}
\usepackage{graphicx}
\usepackage{epstopdf}

\usepackage{layout}
\usepackage{multirow}

\usepackage{relsize}

\usepackage{amssymb} % NOT COMATIBLE WITH svjour3

\newcommand{\myNum}[1]{(\emph{#1})}

\newcommand{\sfrac}[2]{%
    {\mathsmaller{\frac{\raisebox{0.05em}{\footnotesize $#1$}}{\raisebox{-0.15em}{\footnotesize $#2$}}}}
}

 \usepackage{xcolor}
\usepackage{color}
\usepackage{graphicx}
\usepackage{verbatim}

\usepackage{enumitem}
\setlist{nolistsep}

\newcommand{\R}{\mathbb{R}} % Reals
\newcommand{\cN}{{\cal N}}

\newcommand{\cX}{{\cal X}}

\usepackage[colorinlistoftodos,bordercolor=orange,backgroundcolor=orange!20,linecolor=orange,textsize=scriptsize]{todonotes}

\newcommand{\eqdef}{\overset{\text{def}}{=}} 
\DeclareMathOperator{\prox}{prox}       % proximal operator      

\DeclareMathOperator{\Arg}{Arg}         % Argument

\usepackage{enumitem}
\setlist{nolistsep}

\newcommand{\pa}[1]{(#1)}
\newcommand{\Pa}[1]{\big({#1}\big)}

\newcommand{\ie}{\textit{i.e.}~}
\newcommand{\eg}{\textit{e.g.}~}

\newcommand{\Id}{\mathrm{Id}}

\newcommand{\program}{\textrm{ProGrAMMe}}

\newcommand{\iprod}[2]{\langle #1,\,#2 \rangle}

\renewcommand{\cX}{\mathring{X}}

\newcommand{\KL}{Kurdyka-{\L}ojasiewicz{~}}

\usepackage{mathptmx}
\usepackage[T1]{fontenc}

\usepackage{fullpage}
\usepackage[top=2.35cm, bottom=2.35cm, left=2.35cm, right=2.35cm]{geometry}

\usepackage{xcolor}
\usepackage[colorlinks=true]{hyperref}
\hypersetup{
    colorlinks,
    linkcolor=red,          % color of internal links (change box color with linkbordercolor)
    citecolor=blue,        % color of links to bibliography
    filecolor=magenta,      % color of file links
    urlcolor=blue,           % color of external links
}

\newenvironment{@abssec}[1]{%
\vspace{.05in}\small
   \parindent .2in
     {\upshape\bfseries #1. }\ignorespaces
}
{\par\vspace{.025in}}
\renewenvironment{abstract}{\begin{@abssec}{\abstractname}}{\end{@abssec}}
\newenvironment{keywords}{\begin{@abssec}{\keywordsname}}{\end{@abssec}}
\newenvironment{AMS}{\begin{@abssec}{\AMSname}}{\end{@abssec}}

\renewcommand\abstractname{{\noindent}Abstract}
\newcommand\keywordsname{{\noindent}Key words}
\newcommand\AMSname{{\noindent}AMS subject classifications}

\usepackage[boxed,ruled,vlined]{algorithm2e}
\usepackage{algorithmic}

\makeatletter
\def\th@plain{%
  \thm@notefont{}% same as heading font
  \itshape % body font
}
\def\th@definition{%
  \thm@notefont{}% same as heading font
  \normalfont % body font
}
\makeatother
\makeatletter
\renewenvironment{proof}[1][\proofname]{%
  \par\pushQED{\qed}\normalfont%
  \topsep6\p@\@plus6\p@\relax
  \trivlist\item[\hskip\labelsep\bfseries#1\@addpunct{.}]%
  \ignorespaces
}{%
  \popQED\endtrivlist\@endpefalse
}
\makeatother
\theoremstyle{plain}% default
\newtheorem{theorem}{Theorem}[section]
\newtheorem{lemma}[theorem]{Lemma}
\newtheorem{proposition}[theorem]{Proposition}

\theoremstyle{definition}
\newtheorem{definition}{Definition}[section]
\newtheorem{remark}[theorem]{Remark}

\newtheorem*{definition*}{Definition}

\usepackage{titlesec}
\titlespacing*{\section}{0pt}{0.9\baselineskip}{0.375\baselineskip}
\titlespacing*{\subsection}{0pt}{0.65\baselineskip}{0.15\baselineskip}
\titlespacing*{\subsubsection}{0pt}{0.5\baselineskip}{0.125\baselineskip}
\titlespacing*{\paragraph}{0pt}{0.25\baselineskip}{0.25\baselineskip}

\begin{document}

\setlength{\abovedisplayskip}{5.5pt}
\setlength{\belowdisplayskip}{5pt}

\title{An Adaptive Rank Continuation Algorithm for General Weighted Low-rank Recovery}
\author{
		Aritra Dutta\footnote{Equal contributions.} \thanks{Department of Mathematics and Computer Science (IMADA), University of Southern Denmark, DK (Email: {ard@sdu.dk}).}  
\and Jingwei Liang\footnotemark[2] \thanks{School of Mathematical Sciences and Institute of Natural Sciences, Shanghai Jiao Tong University, China (Email: {jingwei.liang@sjtu.edu.cn}).}
\and Xin Li\thanks{Department of Mathematics, University of Central Florida, USA (Email: {xin.li@ucf.edu}).}
		}
\date{}
\maketitle

%% ------------------------------------------------------------------
%% ABSTRACT
%% ------------------------------------------------------------------

\begin{abstract}
This paper is devoted to proposing a general weighted low-rank recovery model and designing a fast SVD-free computational scheme to solve it. 
First, our generic weighted low-rank recovery model unifies several existing approaches in the literature.~Moreover, our model readily extends to the non-convex setting.~Algorithm-wise, most first-order proximal algorithms in the literature for low-rank recoveries require computing singular value decomposition (SVD).~As SVD does not scale appropriately with the dimension of the matrices, these algorithms become slower when the problem size becomes larger. By incorporating the variational formulation of the nuclear norm into the sub-problem of proximal gradient descent, we avoid computing SVD, which results in significant speed-up.~Moreover, our algorithm preserves the {\em rank identification property} of nuclear norm \cite{liang2017activity} which further allows us to design a rank continuation scheme that asymptotically achieves the minimal iteration complexity. Numerical experiments on both toy examples and real-world problems, including structure from motion~(SfM) and photometric stereo, background estimation, and matrix completion, demonstrate the superiority of our proposed algorithm.
\end{abstract}



%%%%%%%%%%%%%%%%%%%%%%%%%%%%%%%%%%%%%%%%%%%%%%%%%%%%%%%%%%%%%%%%%%%%

%\include{tex_main}

\section{Introduction}

Low-rank matrix recovery is an important problem to study as it covers many interesting problems arising from diverse fields including machine learning, data science, signal/image processing, and computer vision, to name a few. The goal of low-rank recovery is to recover or approximate the targeted matrix $\cX \in \mathbb{R}^{m\times n}$ whose rank is much smaller than its dimension.~For example, matrix completion \cite{rmc_taoyuan,candes_MC}, structure from motion \cite{RPCA-BL}, video segmentation \cite{godec,grasta,APG}, image processing and signal retrieval \cite{reprocs, wen2019nonconvex} exploit the inherent low-rank structure of the data.

%\jingwei{Probably an example here, the one i'm thinking is phaselift by Candes et al or the Netflix prize...}

For many problems of interests, instead of accessing the data $\cX$ directly, often we can only observe it through some agent (\eg a linear operator) $\Psi$.~A general observation model takes the following form
\begin{equation}
F = \Psi (\cX) + \varepsilon ,
\end{equation}
where $\Psi : \mathbb{R}^{m\times n} \to \mathbb{R}^{d\times \ell}$ is the (observation) operator which is assumed to be bounded linear. For example, in the compressed sensing scenario, $\Psi$ returns a linear measurement of $\cX$ which is a $d$-dimensional vector \cite{reprocs}; for matrix completion $\Psi$ is a binary mask \cite{candes_MC}.~In the above model, variable $\varepsilon \in \mathbb{R}^{d \times \ell}$ denotes additive noise (\eg white Gaussian) and $F \in \mathbb{R}^{d \times \ell}$ is the obtained noise contaminated observation.

Over the years, numerous low-rank recovery models are proposed in the literature, for example \cite{ma2011fixed,RPCAgd,wen2019nonconvex, svdfree_xiao,APG,regularizedlow_das}, to mention a few. 
%\jingwei{Some work on rank constrained recovery, nuclear norm, or just start from the next line...} 
When the rank of $\cX$ is available, one can consider the following rank constrained weighted least square %recovery model
\begin{equation}\label{prblm:wlr_constraint}
\min_{X\in\mathbb{R}^{m\times n}} ~ \sfrac{1}{2}\|\Pa{\Psi (X) - F }\odot W\|^2 \quad \mathrm{such~that}\quad \mathrm{rank}(X) \leq r  , 
\end{equation}
where $r = \mathrm{rank}(\cX)$ is the rank of $\cX$, $W\in\mathbb{R}^{d\times l}$ is a non-negative weight matrix, and $\odot$ is the Hadamard product. 
%\begin{itemize}[leftmargin=1cm]
%\item $\|\cdot\|_{*}$ denotes the nuclear norm (or trace norm). 
%\item $W : \mathbb{R}^{d}$ is a non-negative weight matrix and $\odot$ is the Hadamard product;
%\end{itemize}
The motivation of considering a weight $W$ is such that \eqref{prblm:wlr_constraint} can handle more general noise model $\varepsilon$, rather than mere Gaussian noise \cite{dutta_thesis, duttali_bg,duttalirichtarik_modeling}.~A clear limitation of \eqref{prblm:wlr_constraint} is that, for many problems it is in general impossible to know $\mathrm{rank}(\cX)$ {\em a priori}. As a result, instead of using rank as constraint, one can penalize it to the objective which results in rank regularized recovery model
\begin{equation}\label{prblm:wlr_rank}
\min_{X\in\mathbb{R}^{m\times n}} ~ \sfrac{1}{2}\|\Pa{\Psi (X) - F }\odot W\|^2 + \tau \mathrm{rank}(X) ,
\end{equation}
where $\tau > 0$ is the regularization parameter. 
Though avoids the estimation of $r$, one needs to choose $\tau$ properly.~Moreover, due to {\tt rank} function, \eqref{prblm:wlr_constraint} and \eqref{prblm:wlr_rank} are non-convex, imposing challenges to both theoretical analysis and algorithmic design.

In literature, a popular approach to avoid  non-convexity is to replace the {\tt rank} function with its convex surrogate---the nuclear norm (a.k.a. trace norm) \cite{LinChenMa, caicandesshen}. Correspondingly, we obtain the following nuclear norm constrained form of \eqref{prblm:wlr_constraint}: 
\begin{equation}%\label{prblm:wlr_constraint}
\min_{X\in\mathbb{R}^{m\times n}} ~ \sfrac{1}{2}\|\Pa{\Psi (X) - F }\odot W\|^2 \quad \mathrm{such~that}\quad \|X\|_{*} \leq c , 
\end{equation}
where $c$ is a predefined constant, \eg $c = \|\cX\|_{*}$ if possible.~Consequently, for \eqref{prblm:wlr_rank}, we arrive at the following {\em unconstrained} nuclear norm regularized recovery model
\begin{equation}\label{prblm:wlr_nuclear}
\min_{X\in\mathbb{R}^{m\times n}} ~ \sfrac{1}{2}\|\Pa{\Psi (X) - F } \odot W\|^2 + \tau\|X\|_{*} .
\end{equation}
Note that, besides the weighted $\ell_2$ loss, one can also consider the general loss function $f(X,F,W)$ which gives the following recovery model (similar to \cite{Ji_APG,Shen_anaccelerated})
\begin{equation}\label{prblm:wlr_general_loss}
\min_{X\in\mathbb{R}^{m\times n}} ~ f(X,F,W) + \tau\|X\|_{*} .
\end{equation}
For the rest of the paper, we mainly focus on model \eqref{prblm:wlr_nuclear} and only present a short discussion of \eqref{prblm:wlr_general_loss} in Section~\ref{sec:generlization}.

\subsection{Related work}
Our recovery model \eqref{prblm:wlr_nuclear} is connected with several established work in the literature, and moreover covers some as special cases. Therefore in what follows, we present a short overview of literature study.  

\paragraph{Trace LASSO.}
When the entries of $W$ are all $1$'s, problem \eqref{prblm:wlr_nuclear} becomes the Trace LASSO considered in \cite{grave2011trace}, \ie nuclear norm regularized least square. For this case, $\varepsilon$ corresponds to additive white Gaussian noise. 
When $\Psi \in \mathbb{R}^{d\times m} ,\, F \in \mathbb{R}^{d\times n}$ and $W = \mathbf{1} \in \mathbb{R}^{d\times n}$, \eqref{prblm:wlr_nuclear} becomes 
\begin{equation}\label{prblm:wlr_cluster}
\min_{X\in\mathbb{R}^{m\times n}} ~ \sfrac{1}{2}\|\Psi X - F\|^2 + \tau\|X\|_{*} ,
\end{equation}
which is studied in \cite{Pong}. Examples of \eqref{prblm:wlr_cluster} include multivariate linear regression, multi-class classification and multi-task learning \cite{Pong}. 
%Instead of considering the least square as loss function, 
As proposed in \cite{jaggi_nn,Ji_APG}, one can also consider a general loss function $f(X, F)$ (as a special case of (\ref{prblm:wlr_general_loss})) to solve
\begin{equation}\label{prblm:wlr_loss}
\min_{X\in\mathbb{R}^{m\times n}} ~ f(X, F) + \tau\|X\|_{*} ,
\end{equation}
For different choice of $f(X, F)$ in \eqref{prblm:wlr_loss}, one can recover affine-rank minimization \cite{Toh2009AnAP, ma2011fixed}, regularized semi-definite linear least squares \cite{Shen_anaccelerated}, etc. % and singular value threhsolding (SVT) \cite{caicandesshen,ma2011fixed}, etc. 

\paragraph{Weighted low-rank recovery.}
For the case $W$ is a general non-negative weight, there is also a train of works in the literature. For most of them, $\Psi$ is an identity operator. We start with rank constraint case \begin{equation}\label{prblm:wlr}
\min_{X\in\mathbb{R}^{m\times n}} ~ \sfrac{1}{2}\|(F-X)\odot W\|^2 \quad \mathrm{such~that}\quad \mathrm{rank}(X) \leq r , 
\end{equation}
which is well-studied in the literature under different settings  \cite{lupeiwang,shpak,srebro,markovosky1,markovosky3}. In \cite{manton}, instead of considering a generic weight $W$, the authors proposed a general matrix induced weighted norm
\begin{equation}\label{manton}
 \min_{X\in\mathbb{R}^{m\times n}} \|F-X\|_{Q}^2  \quad \mathrm{such~that}\quad \mathrm{rank}(X) \leq r , 
\end{equation}
where ${Q\in{\R}^{mn\times mn}}$ is symmetric positive definite and $${\|F-X\|_Q^2\eqdef {\rm vec}(F-X)^\top Q{\rm vec}(F-X)}$$ with ${\rm vec}(\cdot)$ being an operator which maps the entries of $\mathbb{R}^{m\times n}$ to vectors in $\mathbb{R}^{mn\times 1}$ by stacking the columns.~We refer to \cite{dutta_thesis,manton,duttali_bg,duttalirichtarik_modeling,razenshteyn2016weighted,duttali_acl} and the references therein for more discussions. If we lift the constraint to the objective function as for \eqref{prblm:wlr_rank}, we get the problem below: 
\begin{equation}\label{prblm:wlr_rankconstraint}
 \min_{X\in\mathbb{R}^{m\times n}} \sfrac{1}{2}\|(X - F)\odot W\|^2+\tau{\rm rank}(X) ,\end{equation}
which is studied in \cite{duttali,dutta_thesis}.~The latest addition to this class of problems is the weighted singular value thresholding studied in \cite{duttaligongshah} which takes the form:
\begin{equation}\label{eq:wsvt}
\min_{X\in\mathbb{R}^{m\times n}}\sfrac{1}{2}\|(X - F)W\|^2+\tau \|X\|_{*} ,
\end{equation}
where $W\in\R^{n\times n}$ is a weight matrix.~Let $U\Sigma V^\top$ be a SVD of $W$ with $\Sigma = {\rm diag}(\sigma_1\;\sigma_2\cdots \sigma_{n})$. Applying the unitary invariance of the norms (and by the change of variable $X\to XU$), problem \eqref{eq:wsvt} becomes
\begin{equation}
\min_{X\in\mathbb{R}^{m\times n}}\sfrac{1}{2}\|(X-FU)\Sigma\|^2+\tau \|X\|_{*} ,
\end{equation}
which moreover can be equivalently written as
\begin{equation}
 \min_{X\in\mathbb{R}^{m\times n}}\sfrac{1}{2}\|(X-FU)\odot W_{\Sigma}\|^2+\tau \|X\|_{*},
\end{equation}
where ${W_{\Sigma}=(\sigma_1\mathbf{1};\;\sigma_2\mathbf{1} ;\cdots ;\sigma_n\mathbf{1})\in\R^{m\times n}}$ and $\mathbf{1}\in\R^{m\times 1}$ is the vector of all $1$'s.~In Table \ref{tab:svt_table} below, we summarize the above formulations studied in the literature to highlight the difference and connections between them.

\renewcommand{\arraystretch}{1.1}

\begin{table}[tbhp]
{
\begin{center}
  \begin{tabular}{lll} \hline
   \bf Name & \bf Formulation & \bf Reference \\ \hline
    SVD/PCA & $\min_{X: {\rm rank}(X)\leq r} \tfrac12\|X - F\|^2$ & \cite{svd,pca} \\
SVT & $\min_{X} \tau \|X\|_{*}+\|X - F\|^2$ & \cite{caicandesshen} \\
Weighted low-rank~(WLR)  & $\min_{X: {\rm rank}(X)\leq r} \tfrac12\|(X - F)\odot W\|^2$ & \cite{srebro,shpak,lupeiwang}\\
General WLR~(GWLR) & $\min_{X: {\rm rank}(X)\leq r} \tfrac12\|X - F\|^2_Q$ & \cite{manton,markovosky}\\
Weighted SVD & $\min_{X} \mathrm{rank}(X)+\tfrac12\|(X - F)W\|^2$ &\cite{duttali,dutta_thesis} \\
Weighted SVT~(WSVT) & $\min_{X} \tau \|X\|_{*}+\tfrac12\|(X - F)W\|^2$ &\cite{duttali,dutta_thesis} \\
Nuclear norm regularized GWLR & $\min_{X} \tau \|X\|_{*}+\tfrac12\|(\Psi X - F)\odot W\|^2$ & \textcolor{red}{This work} \\
Nuclear norm constrained & $\min_{X: \|X\|_{*}\leq t/2} f(X)$ & \cite{jaggi_nn,ma2011fixed,candes_MC}\\
Trace norm minimization & $\min_X f(X)+\tau\|X\|_{*}$ & \cite{Ji_APG,bl_factor,wen2012lmafit,svdfree_xiao}\\
 \hline
  \end{tabular}
\end{center}
}\vspace{-5mm}
\caption{SVT, weighted low-rank approximation and their variants.}\label{tab:svt_table}
\end{table}

\renewcommand{\arraystretch}{1.0}

\subsection{Contributions}

In this paper, we propose a general model \eqref{prblm:wlr_nuclear} for low-rank recovery. Based on the variational formulation of nuclear norm, we propose an efficient algorithm which avoids computing SVD.~More precisely, our contributions include the following aspects. 

\paragraph{\myNum{i} A generic low-rank recovery model.}
We propose general low-rank recovery models which covers several existing works as special cases.~We provide a detailed comparison of our problem with the existing ones, both analytically and empirically. {We believe problem \eqref{prblm:wlr_nuclear} with our dedicated {\em structure dependent} analysis should be studied as a standalone problem to close the existing knowledge gap.} %\jingwei{This is a bit vague, for example now we dont have analytical analysis of the model. need to revise...}

\paragraph{\myNum{ii} An efficient adaptive rank continuation algorithm.} 
In the literature, numerous numerical schemes can be applied to solve \eqref{prblm:wlr_nuclear}, since it is the sum of a smooth function and a non-smooth one. However, most of these algorithms require computing SVD, which does not scale properly with the dimension of the problem \cite{duttaligongshah,Ji_APG,Pong,RPCA-BL}. 
To efficiently solve \eqref{prblm:wlr_nuclear}, we propose an SVD-free method (see Algorithm \ref{alg:alg}). By combing proximal gradient descent \cite{lions1979splitting} and the variational characteristic of nuclear norm, we design a ``proximal gradient \& alternating minimization method'' which we coin as \program. 
Our algorithm can also be applied to solve the general model \eqref{prblm:wlr_general_loss} if the loss function $f(X,F,W)$ is smoothly differentiable with gradient being Lipschitz continuous. 
Moreover, our algorithm can be easily extended to the non-convex loss function case. 

%\paragraph{A rank continuation strategy}
Based on the result of \cite{liang2017activity}, we show that the sequence generated by Algorithm \ref{alg:alg} can find the rank of the minimizer (to which the generated sequence converges) in {\it finite number of iterations}, which we call {\it rank identification} property. In turn, we design a {\it rank continuation} technique which leads to Algorithm \ref{alg:alg_rc}. Compare to Algorithm \ref{alg:alg}, rank continuation is less sensitive to initial parameter, % $r$ which is the controlling parameter in the variational formulation of nuclear norm (Definition \ref{prop:variationalcharcater}). 
and asymptotically achieves the minimal per iteration complexity. 

\paragraph{\myNum{iii} Numerical comparisons.}
We evaluate our algorithms against 15 state-of-the-art weighted and unweighted low-rank approximation methods on various tasks, including structure  from  motion~(SfM) and photometric stereo, background estimation from fully and partially observed data, and matrix completion. In these problems, different weights are used as deem fit---from binary weights to random large weights. We observed in all the tasks our weighted low-rank algorithm performs either better or is as good as the other algorithms. This indicates that our algorithm is robust and scalable to both binary and general weights on a diverse set of tasks. 

\subsection{Notions and definitions} \label{sec:basic_term}

Throughout the paper, $\mathbb{R}^{n}$ is a finite dimensional Euclidean space equipped with scalar product $\iprod{\cdot}{\cdot}$ and induced norm $\|{\cdot}\|$. We abuse the notation $\|{\cdot}\|$ for the Frobenius norm when $\cdot$ is a matrix. $\Id_{n}$ denotes the identity operator on $\mathbb{R}^{n}$. 
%
%$\mathbb{N}$ is the set of non-negative integers and $k \in \mathbb{N}$ is the index. %$\R^n$ is the Euclidean space of $n$ dimension, and $\Id$ denotes the identity operator on $\R^n$. 
Let $S \subset \mathbb{R}^{n}$ be a non-empty close compact set, then $\mathrm{ri}(S)$ denotes its relative interior, and $\mathrm{par}(S)$ is the subspace which is parallel to $\mathrm{span}(S)$. 
The sub-differential of a proper closed convex function $g: \mathbb{R}^{n} \to \R\cup\{+\infty\}$ is a set-valued mapping defined by
$\partial g :  \mathbb{R}^{n} \rightrightarrows \mathbb{R}^{n} ,~ x \mapsto \big\{ v\in\mathbb{R}^{n} | g(x') \geq g(x) + \iprod{v}{x'-x} ,\, \forall x' \in \mathbb{R}^{n} \big\}$.

%\begin{definition}[$L$-smooth] \label{dfn:L-smooth}
%A function $f$ is $L$ smooth if its gradient is $L$-Lipschitz continuous, that is, for all $x,y\in\R^n$, $\|\nabla f(x)-\nabla f(y)\|\le L\|x-y\|$.  This can be equivalently written as for all $x,y\in\R^n$, $${f(x)\le f(y)+\langle \nabla f(y),x-y\rangle+ (L/2)\|x-y\|^2.}$$  
%\end{definition}
%
%\begin{definition}[$\mu$-strongly convex)]
%A function $f$ is $\mu$-strongly convex if for all $x,y\in\R^n$, $${f(x)\ge f(y)+\langle \nabla f(y),x-y\rangle+ (\mu/2)\|x-y\|^2.}$$  
%\end{definition}

\begin{definition}\label{def:prox}
The proximal mapping (or proximal operator) of a proper closed convex function $g:\R^n\to\R\cup\{+\infty\}$ is defined as: let $\gamma>0$
\begin{equation}
{\rm prox}_{\gamma g}(y)=\arg\min_x \big\{ \gamma g(x)+\sfrac{1}{2}\|x-y\|_2^2 \big\} .
\end{equation}
% where $\lambda>0$ is a balancing parameter. 
\end{definition}

For nuclear norm, its proximal mapping is singular value thresholding (SVT) \cite{caicandesshen}, which is the lifting of vector soft-shrinkage thresholding to matrix \cite{shrinkage2}. 

%\begin{example}[Singular value thresholding (SVT) \cite{caicandesshen}]
%Let $g = \|X\|_{*}$ be the nuclear norm, given a point $Y$ whose SVD reads $Y = U \Sigma V^T$, then the proximal mapping of nuclear norm is {\em singular value thresholding}
%\begin{equation}\label{eq:svt}
%{\rm prox}_{\lambda g}(Y) = U \mathcal{T}_{\lambda}(\Sigma) V^T ,
%\end{equation}
%where $\mathcal{T}_{\lambda}(\cdot)$ denotes soft-shrinkage thresholding operator. 
%\end{example}

\begin{lemma}[Variational formulation of nuclear norm \cite{recht2007,Rennie}]\label{prop:variationalcharcater}
Let $X\in\mathbb{R}^{m\times n}$ and $U\in\mathbb{R}^{m\times r}, V\in\mathbb{R}^{r\times n}$ with $r \geq \mathrm{rank}(X)$. We can write 
\[
\|X\|_{*}
= \min_{\substack{X= UV\\U\in\mathbb{R}^{m\times r}, V\in\mathbb{R}^{r\times n}}} \|U\|\|V\| 
= \min_{\substack{X= UV\\U\in\mathbb{R}^{m\times r}, V\in\mathbb{R}^{r\times n}}}\sfrac{1}{2}(\|U\|^2+\|V\|^2).
\]
\end{lemma}

%\begin{lemma}[\jingwei{...}]\label{prop:convexity}
%The function $f(X) = \frac{1}{2}\|(F-X)\odot W\|^2$ is strongly convex and smooth with Lipschitz constant $L\eqdef \max_{i,j}W_{i,j}^2$ and strong convexity $\mu\eqdef \min_{i,j}W_{i,j}^2.$ As a consequence, $\Phi(X)$ is strongly convex and smooth with the same convexity and smoothness constants as of $f(X).$
%\end{lemma} 

\paragraph{Paper organization}
The rest of the paper is organized as following. In Section \ref{sec:alg}, we describe our proposed algorithm \program~(Algorithm~\ref{alg:alg}) and discuss its global convergence. In Section \ref{sec:rank_cont}, we show the rank identification property of Algorithm \ref{alg:alg} and then propose a rank continuation strategy (Algorithm \ref{alg:alg_rc}). Numerical experiments are provided in Section \ref{sec:results}, followed by the conclusion of this paper.

\section{An SVD-free algorithm}\label{sec:alg} 
Problem \eqref{prblm:wlr_nuclear} is the composition of a smooth term and a non-smooth term. In the literature, numerical methods for such a structured problem are well studied, such as proximal gradient descent \cite{lions1979splitting} (a.k.a. Forward--Backward splitting) and its various variants including the celebrated FISTA \cite{fista09,fista_cd}.~Indeed, our problem \eqref{prblm:wlr_nuclear} can be handled by proximal gradient descent. However, such a method requires repeated SVD computation, which can significantly slow down its performance in many practical scenarios where the data size is large. 
Therefore, in this section, by combining proximal gradient descent and the variational formulation of nuclear norm (c.f., Lemma~\ref{prop:variationalcharcater}), we propose a SVD-free method for solving~\eqref{prblm:wlr_nuclear}. 
% Moreover, acceleration is considered \cite{liang2017activity}.

\subsection{Proposed algorithm}\label{subsection:alg}
In this part, we provide a detailed derivation of our algorithm, which is a combination of proximal gradient descent, alternating minimization, and inertial acceleration. For convenience, denote $f(X)= \frac{1}{2}\|\pa{ \Psi (X) - F }\odot W\|^2$ and $g(X)=\|X\|_{*}$.

\paragraph{Step 1 - Inertial proximal gradient descent.}
The first step to derive our algorithm is applying an inertial proximal gradient descent \cite{lions1979splitting,liang2017activity} to solve problem \eqref{prblm:wlr_nuclear}. Since $\Psi$ is a bounded linear mapping, we have the following simple lemma. % $f(X)$ is smooth differentiable and $\nabla f$ is Lipschitz continuous. 

\begin{lemma}
Let $\widetilde{W} = W\odot W$. The loss $f(X)$ is smoothly differentiable with its gradient given by %reads% then it is easy to verify that % the gradient of $f(X)$ reads,
\[
\nabla f(X) = \nabla \Psi(X) \Pa{ \pa{\Psi (X) - F} \odot \widetilde{W} }   ,
\]
which is $L$-Lipschitz continuous with $\displaystyle{L = \|\nabla \Psi(X)\|^2 \max_{i,j} \widetilde{W}_{i,j}}$. 
\end{lemma}

Below we provide two examples of $\nabla \Psi(X)$:
\begin{itemize}[leftmargin=1cm]
\item In compressed sensing scenario, $\Psi \in \mathbb{R}^{d\times mn}$ is a linear measurement matrix,
\[
\Psi(X) = \Psi \mathrm{vec}(X) 
\quad\mathrm{and}\quad 
\|\nabla \Psi(X)\| = \|\Psi\| .
\]
%and $\|\nabla \Psi(X)\| = \|\Psi\|$. 
\item For matrix completion problem, $\Psi \in \mathbb{R}^{m\times n}$ is a binary mask, and 
\[
\Psi(X) = \Psi \odot X 
\quad\mathrm{and}\quad 
\|\nabla \Psi(X)\| = 1 .
\]
%and $\|\nabla \Psi(X)\| = 1$. 
\end{itemize}
%%%%
%Derivation
%\[
%\begin{aligned}
%&\lim_{t\to0} \sfrac{f(X+th)-f(X)}{t} \\
%&= \frac{\frac{1}{2}\|(A-\Psi (X+th))\odot W\|^2 - \frac{1}{2}\|(A-\Psi X)\odot W\|^2}{t}  \\
%&= \frac{\|(A-\Psi X)\odot W - t\Psi h\odot W\|^2 - \|(A-\Psi X)\odot W\|^2}{2t}  \\
%&= \frac{\|(A-\Psi X)\odot W\|^2 - 2\iprod{(A-\Psi X)\odot W}{t\Psi h\odot W} + t^2\|\Psi h\odot W\|^2 - \|(A-\Psi X)\odot W\|^2}{2t}  \\ 
%&= \frac{ - 2\iprod{(A-\Psi X)\odot W}{t\Psi h\odot W} + t^2\|\Psi h\odot W\|^2}{2t} \\
%&= - \iprod{(A-\Psi X)\odot W}{\Psi h\odot W}  \\
%&= - \iprod{\Psi^T (A-\Psi X)\odot W\odot W}{ h}
%\end{aligned}
%\]
%Lipschitz constant
%\[
%\begin{aligned}
%\|\nabla f(X) - \nabla f(Y)\|
%&= \|\Psi^T \Pa{ (\Psi X - F)\odot \widetilde{W} } - \Psi^T \Pa{ (\Psi Y - F)\odot \widetilde{W} }\| \\
%&= \|\Psi^T \pa{ \Psi X \odot \widetilde{W} } - \Psi^T \pa{ \Psi Y \odot \widetilde{W} }\| \\
%&= \|\Psi^T \pa{ \Psi (X-Y) \odot \widetilde{W} } \| \\
%&\leq \|\Psi\| \|\Psi (X-Y) \odot \widetilde{W}\| \\
%\end{aligned}
%\]
%%%% 
In the literature, a routine approach to solve \eqref{prblm:wlr_nuclear} is inertial proximal gradient descent. 
Let $X_0 \in \mathbb{R}^{m\times n}$ be an arbitrary starting point, we consider the following iteration
\begin{equation}\label{eq:pgd}
\begin{aligned}
Y_{k} &= X_k + a_k (X_{k} - X_{k-1}) , \\
{Z}_{k} &= Y_{k} - \gamma \nabla \Psi(Y_k) \Pa{ (\Psi (Y_k) - F)\odot \widetilde{W} } , \\
X_{k+1} &= {\arg\min}_{X\in\mathbb{R}^{m\times n}} \tau \|X\|_{*} + \sfrac{1}{2\gamma}\|X-{Z}_{k}\|^2 , 
\end{aligned}
\end{equation}
where $a_k \in [0,1]$ is the inertial parameter, $\gamma \in ]0, 2/L[$ is the step-size. % and the last line of \eqref{eq:pgd} is nothing but SVT of $Z_k$. 
%The second line of \eqref{eq:pgd} has closed form solution, which is simply SVT~\cite{caicandesshen}. 
Iteration \eqref{eq:pgd} is a special case of the general inertial scheme proposed in \cite{liang2017activity}, and we refer to \cite{liang2017activity} and the references therein for more discussion on inertial schemes. 
% The last line of \eqref{eq:pgd} outputs the proximal mapping of nuclear norm, which is SVT and involves SVD computation.  

\paragraph{Step 2 - Alternating minimization.}
%This is the key step of our algorithm. 
As computing the proximal mapping of nuclear norm requires SVD, the goal of second step is to avoid SVD in solving SVT of $Z_k$ by incorporating the variational formulation of nuclear norm. 
To this end, the subproblem of \eqref{eq:pgd} reads
\begin{equation}\label{prblm:proximal}
\min_{X\in\mathbb{R}^{m\times n}} \tau \|X\|_{*} + \sfrac{1}{2\gamma}\|X-{Z}_{k}\|^2 .
\end{equation}
By plugging in the variational formulation of nuclear norm, we arrive at the following constrained minimization problem: let $r > \mathrm{SVT}_{\tau\gamma}(Z_k)$
\begin{equation}\label{prblm:nnfit0}
\min_{X\in\mathbb{R}^{m\times n},U\in\mathbb{R}^{m\times r},V\in\mathbb{R}^{r\times n}}  \sfrac{1}{2}\|X- {Z}_{k}\|^2 + \sfrac{\tau\gamma}{2}\left(\|U\|^2+\|V\|^2\right)\;{\rm such\;that\;}X=UV.
\end{equation}
Instead of considering the augmented Lagrangian multiplier of the constraint \cite{bl_factor}, we directly substitute the constraint $X=UV$ in the objective, which leads to
\begin{equation}\label{prblm:nnfit}
\min_{U\in\mathbb{R}^{m\times r},V\in\mathbb{R}^{r\times n}} \sfrac{1}{2}\|UV-{Z}_{k}\|^2 + \sfrac{\tau\gamma}{2}\left(\|U\|^2+\|V\|^2\right),
\end{equation}
which is a smooth, bi-convex optimization problem in each component $U$ and $V$. 

Different from \eqref{prblm:proximal}, problem \eqref{prblm:nnfit} does not admits closed form solution. However, when either $U$ or $V$ is fixed, the problem becomes a simple least square. Hence we can solve \eqref{prblm:nnfit} via a simple alternating minimization, namely a two block Gauss-Seidel iteration \cite{attouch2013convergence}: given $U_0\in\mathbb{R}^{m\times r}, V_0\in\mathbb{R}^{r\times n}$
\begin{equation}\label{eq:alternating_min}
\begin{aligned}
U_{i+1} &= {{Z}_{k}}V_{i}^\top \big( V_{i}V_{i}^\top+ {\tau\gamma} \Id_r \big)^{-1} , \\
V_{i+1} &= \big( U_{i+1}^\top U_{i+1}+{\tau\gamma}\Id_r \big)^{-1} U_{i+1}^\top{{Z}_{k}} ,
\end{aligned}
\end{equation}
where $\Id_r$ denotes the identity operator on $\mathbb{R}^r$. Substituting \eqref{eq:alternating_min} into \eqref{eq:pgd} as an inner loop, we obtain the following iterative scheme: let $I \in \mathbb{N}_{+}$
\begin{equation}\label{eq:proximal_alternating}
\begin{aligned}
&{Z}_{k} = Y_{k} - \gamma \nabla \Psi(Y_k) \Pa{ (\Psi (Y_k) - F)\odot \widetilde{W} }  , \\
& \textrm{Initialize $U_0, V_0$. For $i=0,\ldots,I-1$:} \\
&\left\lfloor 
\begin{aligned}
	U_{i+1} &= {{Z}_k}V_{i}^\top\Pa{ V_{i}V_{i}^\top + {\tau\gamma}\Id_r }^{-1} , \\
	V_{i+1} &= \Pa{U_{i+1}^\top U_{i+1} + {\tau\gamma}\Id_r}^{-1}U_{i+1}^\top{{Z}_k} ,
\end{aligned}
\right. \\
&X_{k+1} = U_{I}V_{I}  .
\end{aligned}
\end{equation}
Assembling the above steps, we obtain our proposed algorithm, proximal gradient \& alternating minimization method, which we call \program~and is summarized below in Algorithm \ref{alg:alg}.

%\paragraph{Step 3 - Inertial acceleration}
%The last step is add inertial technique into \eqref{eq:proximal_alternating} for acceleration purpose. 
%The combination of inertial and \eqref{eq:proximal_alternating} is rather straightforward, only needs to change the update of $Z_{k}$. %, the only different is updating $Z_k$ with $Y_k$, an extrapolated point based on $X_k, X_{k-1}$, instead of $X_k$. 
%More precisely, let $a_k \in [0, 1]$ be the inertial parameter, then 
%\[
%\begin{aligned}
%&Y_{k} = X_k + a_k (X_{k} - X_{k-1}) , \\
%&{Z}_k = Y_k - \gamma \nabla \Psi(X) \Pa{ (\Psi (Y_k) - F)\odot \widetilde{W} }  . 
%%& \textrm{For $i=1,\ldots,m$:} \\
%%&\left\lfloor 
%%\begin{aligned}
%%	U_{k+1} &= {\bar{Z}_k}V_{k}^\top(V_{k}V_{k}^\top+\frac{\tau}{L}I_r)^{-1} , \\
%%	V_{k+1} &= (U_{k+1}^\top U_{k+1}+\frac{\tau}{L}I_r)^{-1}U_{k+1}^\top{\bar{Z}_k} ,
%%\end{aligned}
%%\right. \\
%%&X_{k+1} = U_{k+1}V_{k+1}
%\end{aligned}
%\]
%%The rest follows that of \eqref{eq:proximal_alternating}. 
%%
%Assembling the above steps, we obtain our proposed algorithm, proximal gradient and alternating minimization method, which we call \program~and is described below in Algorithm \ref{alg:alg}. 

\begin{center}
	\begin{minipage}{0.975\linewidth}
	
\begin{algorithm}[H]
\caption{A Proximal Gradient \& Alternating Minimization Method (ProGrAMMe)}\label{alg:alg}
\begin{algorithmic}[1]
%\STATE{Data matrix $F$, weight matrix $W $; Choose $\tau >0$, $r > 0$ and $I \in \mathbb{N}_{+}$\;}
\STATE{Compute $\widetilde{W}=W\odot W, ~L$ and et $\gamma \in ]0, 2/L[$; Choose $r > 0$ and $I \in \mathbb{N}_{+}$;}
\WHILE{not convergent}
\STATE\label{line3}{$Y_{k} = X_k + a_k (X_{k} - X_{k-1}) $,\hfill $//$\textcolor{blue}{\tt inertial step}}$//$~~~~~
\STATE{${Z}_{k} = Y_{k} - \gamma \nabla \Psi(Y_k) \Pa{ (\Psi (Y_k) - F)\odot \widetilde{W} }$,\hfill $//$\textcolor{blue}{\tt gradient descent}}$//$~~~~~
\STATE{Initialize $U_0\in\R^{m\times r}, V_0\in\R^{r\times n}$,}
%
% \FOR{$i=0,...,I-1$} 
\STATE{{\bf for} $i=1,...,I-1$ {\bf do}\hfill $//$\textcolor{blue}{\tt inner loop}}$//$~~~~~
\STATE{\qquad$U_{i+1}={{Z}_k}V_{i}^\top \pa{ V_{i}V_{i}^\top+ {\tau\gamma}\Id_r }^{-1}$,}
\STATE{\qquad$V_{i+1}=\pa{ U_{i+1}^\top U_{i+1}+{\tau\gamma}\Id_r }^{-1}U_{i+1}^\top{{Z}_k}$,}
\STATE{{\bf end for}}
% \ENDFOR
%
\STATE{$X_{k+1}=U_{I}V_{I}$.}
\ENDWHILE
\RETURN $X_{k+1}$
\end{algorithmic}
\end{algorithm}
	\end{minipage}
\end{center}

\begin{remark}$~$
\begin{itemize}
\item
One highlight of our algorithm is that, via variational formulation of nuclear norm, we relaxed the convex subproblem \eqref{prblm:proximal} to a non-convex \eqref{prblm:nnfit} problem. 
\item
Every step of Algorithm \ref{alg:alg} requires initializing $U_0, V_0$ for the inner step, and the simplest way is using the $U_I, V_I$ from the last step. 
\item
For the controlling parameter $r$, theoretically it does not make any difference as long as it is larger than the rank of the solution of \eqref{prblm:wlr_nuclear}. However, practically it is crucial to the performance of Algorithm \ref{alg:alg}. Detailed discussion is provided in Section \ref{sec:rank_cont}.
\end{itemize}
\end{remark}

\begin{remark}
In the literature, several SVD-free approaches were proposed. For example, variational formulation was also considered in \cite{bl_factor} and the resulted problem was solved by method of Lagrange multiplier while we directly plug the constraint into the objective. 
In \cite{svdfree_xiao}, a dual characterization of the nuclear norm was used and a SVD-free gradient descent was designed. 
The benefits of our approach, as we shall see later, are  simple convergence analysis (see Secsion \ref{subsec:convergence}) and extensions to more general settings (see Section \ref{sec:generlization}). 
% {\textcolor{red}{It would be nicer if we can say why ours is better. Do the other allow easy convergence analysis? Do they have slower performance?}}
%In contrast, Liu et al.\ \cite{Liu:2014:NNR} converted \eqref{prblm:JiYe_nn} to a Riemannian optimization problem over matrix (Grassmannian) manifold and used conjugate gradient method to improve its convergence than the standard gradient descent. 
%Pong et al.\ proposed several numerical procedures, such as, DCG, DGP, PAPG, ${\rm PAPG_{avg}}$ in \cite{Pong} to solve \eqref{prblm:JiYe_nn} in their special form. 
%However, \cite{Pong} noted that DCG is slow; DGP and PAPG work only when the largest and smallest singular values of $A=W^\top$ are not too large and not too small, respectively. 
%Recently, to solve WSVT, Dutta et al.\ argued that in practice the weight matrix $W$ may have very large and small singular values and the proposed methods in \cite{Pong,Ji_APG,Shen_anaccelerated} struggle to converge (see~\cite{duttaligongshah}).
\end{remark}

\begin{remark}[Per iteration complexity]
Comparing Algorithm \ref{alg:alg} to proximal gradient descent \eqref{eq:pgd}, the only difference is \texttt{Line 5-9}. For proximal gradient descent, since SVD is needed, the iteration complexity each step is $O(mn\min\{m, n\})$. For Algorithm \ref{alg:alg}, suppose $I = 1$, the complexity of \texttt{Line 7-8} is $O((m+n+r)r^2)$. 
It can be concluded that, the smaller the value of $r$ (still larger than the rank of the solutions), the lower the per iteration complexity of Algorithm \ref{alg:alg}. As a result, the choice of $r$ is crucial to the practical performance of Algorithm \ref{alg:alg}. Therefore, in Section \ref{sec:rank_cont}, a detailed discussion is provided on how to choose $r$. 
\end{remark}

%\textcolor{red}{We note that, unlike Ban et al.\ in \cite{ban2019regularized}, we do not solve \eqref{prblm:wlr_nuclear} directly by replacing $\|X\|_*$ by Proposition \ref{prop:variationalcharcater} and then by using a sketch based update. }

%This non-convex relaxation step is where we deviate from the algorithm in \cite{Ji_APG}: we do not want to use SVD to directly solve \eqref{prblm:proximal} and, instead, we take the SVD-free approach. 

\subsubsection{Relation with existing work} 
Our algorithm is closely related with proximal splitting method and its variants, as our first step to derive Algorithm \ref{alg:alg} is the inertial proximal gradient descent (for example, proximal gradient descent \cite{lions1979splitting} and its accelerated versions including FISTA \cite{fista09,fista_cd}, as in 
\cite{Shen_anaccelerated,Toh2009AnAP,Ji_APG} where FISTA was adopted to solving low-rank recovery problem). 
% Jaggi et al.\ used a semi-definite program~(SDP) solver of \cite{hazan} to solve problem \eqref{prblm:jaggi_nn}. 

%In this part, we provide a dedicated discussion on relation with some related work, including \cite{Buchanan,bl_factor} for formulation and \cite{attouch2010proximal,bolte2014proximal} for algorithm. 

Our model \eqref{prblm:wlr_nuclear} and Algorithm \ref{alg:alg} share similarities with those of \cite{Buchanan,regularizedlow_das,bl_factor}, but there are some fundamental differences. 
First of all, all these works consider only the case $\Psi = \Id$, \ie $\Psi$ is an identity mapping
\begin{equation}\label{prblm:wlr_id}
\min_{X} \sfrac{1}{2}\|(X - F)\odot W\|^2 + \tau\|X\|_{*}  .
\end{equation}
In \cite{Buchanan,bl_factor}, the authors consider directly applying matrix factorization to \eqref{prblm:wlr_id} which results in: let $r > 0$ %and $U \in \mathbb{R}^{m\times r}, V \in \mathbb{R}^{r\times n}$
\begin{equation}\label{prblm:wlr_id_mf}
\min_{U \in \mathbb{R}^{m\times r}, V \in \mathbb{R}^{r\times n}} \sfrac{1}{2}\|(UV - F)\odot W\|^2 + \sfrac{\tau}{2}\left(\|U\|^2+\|V\|^2\right) ,
\end{equation}
which is a smooth and bi-convex optimization problem. 
Our approach, on the other hand, only consider applying matrix factorization for the subproblem of proximal gradient descent \eqref{eq:pgd}.

It is worth noting that \eqref{prblm:wlr_id_mf} shares the same continuous property as \eqref{prblm:nnfit}, hence can by handled by alternating minimization algorithm. 
Both \eqref{prblm:wlr_id_mf} and \eqref{prblm:nnfit} are also special cases of the following non-convex problem
\begin{equation}\label{prblm:general}
\min_{U, V} f(U, V) + \tau_1 g_1(U) + \tau_2 g_2(V) 
\end{equation}
where $f(U, V)$ is differentiable with Lipschitz continuous gradient, $\tau_1, \tau_2 > 0$ are regularization parameters and $g_1(\cdot),\, g_2(\cdot)$ are (non-smooth) regularization terms for $U, V$, respectively. 
Problem \eqref{prblm:general} was well studied in \cite{attouch2010proximal,bolte2014proximal}; for instance, the following algorithm was proposed in \cite{attouch2010proximal}: 
\[
\begin{aligned}
U_{k+1} &= \arg\min_{U \in \mathbb{R}^{m\times r}}~\big\{ f(U, V_{k}) + \tau_1 g_1(U) + \sfrac{1}{2\alpha}\|U - U_{k}\|^2 \big\} , \\
V_{k+1} &= \arg\min_{V \in \mathbb{R}^{r\times n}}~\big\{ f(U_{k+1}, V) + \tau_2 g_2(V) + \sfrac{1}{2\gamma}\|V - V_{k}\|^2 \big\} ,
\end{aligned}
\]
where $\alpha, \gamma > 0$ are parameters. 
Specializing to the case of \eqref{prblm:wlr_id_mf}, we get
\[
\begin{aligned}
U_{k+1} &= \arg\min~\Big\{ \sfrac{1}{2}\|(UV_{k} - F)\odot W\|^2 + \sfrac{\tau}{2} \|U\|^2 + \sfrac{1}{2\alpha}\|U - U_{k}\|^2 \Big\} , \\
V_{k+1} &= \arg\min~\Big\{ \sfrac{1}{2}\|(U_{k+1} V - F)\odot W\|^2 + \sfrac{\tau}{2}\|V\|^2 + \sfrac{1}{2\gamma}\|V - V_{k}\|^2 \Big\} .
\end{aligned}
\]
It can be observed that, though sharing similarities, our proposed algorithm is different from the above schemes. Same for the algorithm proposed in \cite{bolte2014proximal}.

%We do not use the matrix factorization scheme to represent the original problem \eqref{prblm:wlr_nuclear} with general weight matrix $W$. We linearize \eqref{prblm:wlr_nuclear} and transform it by using the matrix product for low-rank and replace the nuclear norm by Proposition \ref{prop:variationalcharcater}. Finally, we use the alternating minimization technique on this transformed linearized problem \eqref{prblm:nnfit}. That is, \eqref{prblm:nnfit} serves as an intermediate problem to solve the main problem \eqref{prblm:wlr_nuclear} with a general weight matrix $W$. Now we will state a result that {\em unifies} \eqref{prblm:proximal} and \eqref{prblm:nnfit}.
%
%
%
%For the rest of the section, we provide theoretical guarantees of Algorithm \ref{alg:alg}, including the relation between the solution of the variational formulation \eqref{prblm:nnfit} and the original nuclear norm regularized form \eqref{prblm:wlr_nuclear}, the convergence property of Algorithm \ref{alg:alg}. 

%\vspace{-5pt}
\subsection{Global convergence of Algorithm \ref{alg:alg}}
\label{subsec:convergence}

In this part, we provide global convergence analysis of Algorithm \ref{alg:alg}. The key of our proof is rewriting Algorithm \ref{alg:alg} as an inexact version of inertial proximal gradient descent \eqref{eq:pgd} whose convergence property is well established in the literature. 
Such an equivalence is obtained based on the result below from~\cite[Theorem 1]{bl_factor}.

\begin{lemma}[{\cite[Theorem 1]{bl_factor}}]\label{lem:equivalence}
Let $\widehat{X}$ be the unique minimizer of \eqref{prblm:proximal}, \ie \[\min_{X\in\mathbb{R}^{m\times n}} \tau \|X\|_{*} + \sfrac{1}{2\gamma}\|X-{Z}_{k}\|^2\] with rank $\hat{r} = \mathrm{rank}\pa{\widehat{X}}$, and $(\widehat{U}, \widehat{V})$ a solution of \eqref{prblm:nnfit} \[\min_{U\in\mathbb{R}^{m\times r},V\in\mathbb{R}^{r\times n}} \sfrac{\tau\gamma}{2}\left(\|U\|^2+\|V\|^2\right) + \sfrac{1}{2}\|UV-{Z}_{k}\|^2 \] with $r \geq \hat{r}$. There holds $\widehat{X} = \widehat{U} \widehat{V}$.  
\end{lemma}

The above lemma implies that, although we relaxed the strongly convex problem \eqref{prblm:proximal} to a non-convex one \eqref{prblm:nnfit}, we can recover the unique minimizer of \eqref{prblm:proximal} via solving  \eqref{prblm:nnfit}. In turn, we can cast Algorithm \ref{alg:alg} back to the proximal gradient descent \eqref{eq:pgd}, possibly with approximation errors due to finite step inner loop, and then prove its convergence. To this end, we first propose the following inexact characterization of Algorithm \ref{alg:alg}. 
Given $S\in\mathbb{R}^{m\times n}$ and $\gamma>0$, denote $\mathcal{T}_{\gamma}(S)$ the soft-thresholding operator \cite{shrinkage2} 
\begin{equation}\label{eq:svt}
\mathcal{T}_{\gamma}(S_{i,j}) = \max\big\{|S_{i,j}|-\gamma,~ 0\big\} \times \mathrm{sign}(S_{i,j}).
\end{equation}

%We first discuss the global convergence of $X_k$ to a minimizer of \eqref{prblm:wlr_nuclear}, to this end, we cast Algorithm \ref{alg:alg} as an instance of {\em inexact proximal gradient descent} and then establish its convergence based on existing result in the literature. 
%
%
%Based on Theorem \ref{thm:unify}, when the inner loop of Algorithm \ref{alg:alg} is ran for infinitely many step, then Algorithm \ref{alg:alg} is simply an instance of inertial proximal gradient descent \cite{liang2017activity}. For finite choice of $I$, the iteration accounts for an inexact realization of the proximal mapping of nuclear norm. As a result, we have the following result. 

%{\color{blue} 
\begin{proposition}[Inexact inertial proximal gradient descent]\label{prop:equivalence}
For Algorithm \ref{alg:alg}, let $I \in \mathbb{N}_{+}$. Then there exits a sequence $\{e_k\}_{k}\subset\mathbb{R}^{m\times n}$ such that Algorithm \ref{alg:alg} is equivalent o the following inexact inertial proximal gradient descent 
\begin{equation}\label{eq:inexact_ifb}
\begin{aligned}
Y_{k} &= X_k + a_k (X_{k} - X_{k-1}) , \\
{Z}_k &= Y_k - \gamma \nabla\Psi \Pa{ (\Psi (Y_k) - F)\odot\widetilde{W} }  , \\
X_{k+1} &= \prox_{\tau\gamma \|\cdot\|_*} ({Z}_k) + e_k  .
%R_{k+1} &= \prox_{\tau\gamma \|\cdot\|_*} ({Z}_k) , \\
%X_{k+1} &= %P \mathcal{T}_{\tau\gamma}\big(S + e_k \big)  Q^T = \prox_{\tau\gamma \|\cdot\|_*} ({Z}_k + c_k) , \quad
%R_{k+1} + e_k  .
\end{aligned}
\end{equation}
%where $e_k \in \mathbb{R}^{m\times n}$ accounts for truncation error due to finite-valued $I$. 
% where $Z_k = PSQ^T$ stands for the SVD of $Z_k$ and $e_k \in \mathbb{R}^{m\times n}$ accounts for truncation error due to finite-valued $I$. 
\end{proposition}
\begin{remark}
$e_k \in \mathbb{R}^{m\times n}$ accounts for truncation error due to finite-valued $I$, and vanishes if the inner iteration of Algorithm \ref{alg:alg} is solved exactly. 
\end{remark}
\begin{proof}
Recall the iteration in \eqref{eq:pgd}
\[
\begin{aligned}
Y_{k} &= X_k + a_k (X_{k} - X_{k-1}) , \\
{Z}_{k} &= Y_{k} - \gamma \nabla \Psi(Y_k) \Pa{ (\Psi (Y_k) - F)\odot \widetilde{W} } , \\
X^{\infty}_{k+1} &= {\arg\min}_{X\in\mathbb{R}^{m\times n}} \tau \|X\|_{*} + \sfrac{1}{2\gamma}\|X-{Z}_{k}\|^2 , 
\end{aligned}
\]
we use $X^{\infty}_{k+1}$ to denote the output of \eqref{eq:pgd},  
we have
\begin{equation}\label{eq:svd-x}
X^{\infty}_{k+1}
= \prox_{\tau\gamma \|\cdot\|_*} ({Z}_k)   .
\end{equation}
For the inner loop of Algorithm \ref{alg:alg}, after $I$ steps of iteration we have
\[
X^{I}_{k+1} 
= U_{I} V_{I} ,
% = P_{I} S_{I} Q^T_{I} 
\]
and that 
\begin{equation}\label{eq:ItoInf}
\lim_{I\to\infty} X^{I}_{k+1} 
= X^{\infty}_{k+1}     
\end{equation}
which owes to Lemma \ref{lem:equivalence}. 
Let $e_k = X^{I}_{k+1} - X^{\infty}_{k+1}$, we conclude the proof.     
\end{proof}
%}

The inexact formulation of Algorithm \ref{alg:alg} allows us to prove its convergence in a rather simple fashion, as \eqref{eq:inexact_ifb} is nothing but a special case of the inertial proximal gradient descent discussed in \cite{liang2017activity}. 
As a consequence, we have the following result regarding the global convergence of Algorithm \ref{alg:alg}.  
Recall that $f(X)= \frac{1}{2}\|\pa{ \Psi (X) - F }\odot W\|^2$ and $g(X)=\|X\|_{*}$.

\begin{proposition}[Convergence of Algorithm \ref{alg:alg}]\label{prop:convergence}
For Algorithm \ref{alg:alg}, let the inertial sequences $\{a_k\}_{k\in\mathbb{N}}$ be such that \[
\limsup_{k} a_k < 1
\quad\mathrm{and}\quad
\sum\nolimits_{k\in\mathbb{N}} a_{k}\|X_{k}-X_{k-1}\|^2 < +\infty .
\]
If, moreover, the error $e_k$ is such that
\[
\sum\nolimits_{k\in\mathbb{N}} k \|{e_k}\|  < +\infty. 
\]
Then, there exists $X^\star \in \arg\min(f+\tau g)$ to which the sequences $\{R_k, X_k\}_{k\in\mathbb{N}}$ generated by Algorithm \ref{alg:alg} converge.
\end{proposition}

Algorithm \ref{alg:alg} is a special case of the algorithm considered in \cite{liang2017activity}, and hence the convergence can be guaranteed by \cite[Theorem 3]{liang2017activity}. We decide to omit the proof here and refer the reader to \cite{liang2017activity} for detailed discussions.

\begin{remark}
The condition on error $e_k$ implies that the inner problem needs to be solved with an increasing accuracy, \ie the value of $I$ should be increasing along iteration. A more practice approach would be increasing the value of $I$ every $O(1)$ steps. Moreover, we observe that fixed value of $I$ works quite well in practice; see our numerical examples. 
The summability of $a_{k}\|X_{k}-X_{k-1}\|^2$ can be guaranteed by certain choices of $a_k$; See \cite[Theorem 4]{liang2017activity}. One can also use an online approach to determine $a_k$ such that the summability condition holds. For instance, let $a \in [0, 1]$ and $c > 0, \delta > 0$, then $a_k$ can be chosen as $a_k = \min \{a, \frac{c}{k^{1+\delta}\|X_{k}-X_{k-1}\|^2} \} $. 
\end{remark}

\begin{remark}[FISTA-like inertial parameter]
If the error term $e_k$ can be carefully taken care of, such as gradually increase the value of $I$, then according to \cite{aujol2015stability} FISTA rule for updating $a_k$ can be applied, \eg $a_k = \frac{k-1}{k+d}$ for $d>2$. %One can then use the lazy-start strategy of \cite{liang2018improving} for further speed-up. 
\end{remark}

%\begin{remark}
%In \cite{liang2017activity}, the authors also studied the local linear convergence of proximal gradient descent type methods, under the notion of ``partial smoothness'' (Definition \ref{dfn:psf}). However, due to the existence of the error term $e_k$, the local linear convergence of \eqref{eq:inexact_ifb} is more complicated than that of \cite{liang2017activity} where the analysis was obtained only for the exact case. 
%%For example, if $\|{e_k}\| = O(1/k^{2+\delta})$ for some $\delta > 0$, then eventually the convergence rate of \eqref{eq:inexact_ifb} will be dominated by $e_k$. 
%Therefore, for this aspect, we forgo the theoretical analysis and only provide numerical discussions in Section \ref{sec:rank_cont}. 
%\end{remark}

\subsection{Generalization of Algorithm \ref{alg:alg}}\label{sec:generlization}

Throughout this paper, our main focus is \eqref{prblm:wlr_nuclear} whose loss function is a simple weighted least square. In this scope, we discuss several generalization of Algorithm \ref{alg:alg}, including more general loss function (\eg \eqref{prblm:wlr_general_loss}), the non-convex setting and non-linear $\Psi$.

\paragraph{General loss function.} 
Algorithm \ref{alg:alg} is loss function agnostic, as for proximal gradient descent type methods, the condition required for the smooth part is that the function should be smoothly differentiable with gradient being Lipschitz continuous. Therefore, we can apply Algorithm \ref{alg:alg} to solve the more general model \eqref{prblm:wlr_general_loss}, and the only change we need to make to Algorithm \ref{alg:alg} is {\tt Line 4} for which we now have
\[
{Z}_{k} = Y_{k} - \gamma \nabla_1 f(Y_{k},F,W)  ,
\]
where $\nabla_1$ denotes the gradient of $f(X,F,W)$ with respect to $X$. The global convergence result stays the same for the above update.

\paragraph{Non-convex loss function.}
Algorithm \ref{alg:alg} does not require convexity for the loss function. In the literature, the convergence properties of proximal gradient descent type methods for non-convex optimization are well studied, most of them are obtained under \KL inequality owing to the pioneered work \cite{attouch2013convergence}. 
As Algorithm \ref{alg:alg} is a special case of inexact proximal gradient descent, it can also be applied to solve problems where the loss function $f(X,F,W)$ is non-convex. Parameter-wise, there are two main differences between the non-convex and convex cases:
\begin{itemize}[leftmargin=1cm]
\item For step-size $\gamma$, different from the convex case whose upper bound is $2/L$, it reduces to $1/L$ for the non-convex case.
\item The conditions on the error is different.~For the non-convex case, the error should be such that a descent property of certain stability function (see \eg \cite{attouch2013convergence}) should be maintained. As a result, {\it line search} might be needed for the number of inner loop iteration. 
\end{itemize}

\begin{remark}
While keeping the loss function as least square, it still can be non-convex because of the operator $\Psi$, for which case $\Psi$ is a {\it non-linear smooth mapping} instead of being linear. 
For this case, as long as $\Psi$ is such that the gradient is Lipschitz continuous, global convergence of Algorithm \ref{alg:alg} can be guaranteed. 
\end{remark}

%{\color{blue} 

\section{Rank continuation}\label{sec:rank_cont}

As we discussed above, the choice of parameter $r$ is crucial to the performance of Algorithm \ref{alg:alg}: let $r^\star$ be the rank of a solution $X^\star$ of the problem \eqref{prblm:wlr_nuclear}, then the closer the value of $r$ to $r^\star$, the better practical performance of Algorithm \ref{alg:alg} (the best performance if $r = r^\star$). 
However, in general it is impossible to know $r^\star$ {\it a priori}, and usually an overestimation of $r^\star$ is provided in practice which damps the efficiency of the algorithm. 
In this section, we first discuss the rank identification property of Algorithm \ref{alg:alg}, and then discuss a rank continuation strategy which asymptotically achieves the optimal per iteration complexity. 

To simplify the discussion, we introduce an auxiliary variable for iteration \eqref{eq:inexact_ifb} of Proposition \ref{prop:equivalence}, 
\begin{equation}\label{eq:Rk}
R_{k+1} = \prox_{\tau\gamma \|\cdot\|_*} ({Z}_k) . 
\end{equation}

\subsection{Rank identification of \eqref{eq:inexact_ifb}}

In \cite{liang2017activity}, for proximal gradient descent type methods dealing with low-complexity promoting regularization, it was shown that the sequence generated by these methods has a so-called ``{\it finite time activity identification property}''. For nuclear norm, this means that for all $k$ large enough there holds $\mathrm{rank}(R_k) = \mathrm{rank}(X^\star)$, where $X^\star$ is the solution to which $R_k, X_k$ converge.  

\begin{remark}
In practice for \eqref{eq:inexact_ifb}, rank identification can also be observed for the sequence $X_k$, however, due to the lack of structure for the error $e_k$, we cannot prove it for $X_k$.
\end{remark}

In the theorem below, we show the rank identification property of Algorithm \ref{alg:alg}.  
Recall the notations $f(X)= \frac{1}{2}\|\pa{ \Psi (X) - F }\odot W\|^2$ and $g(X)=\|X\|_{*}$ of Section~\ref{subsection:alg}. 

\begin{theorem}[Rank identification]\label{thm:rank_iden}
For Algorithm \ref{alg:alg}, suppose the conditions of Proposition \ref{prop:convergence} hold, then $R_{k}$ converges to $X^\star \in \Arg\min(f+\tau g)$. If, moreover, the following non-degeneracy condition holds
\begin{equation}\label{eq:cnd-nondeg}
- \nabla \Psi(X^\star) \Pa{ (\Psi (X^\star) - F) \odot \widetilde{W} } \in \tau \mathrm{ri}\big( \partial \|X^\star\|_{*} \big)  ,
\end{equation}
%where $\mathrm{ri}(\cdot)$ stands for the relative interior of a closed compact set. 
then there exists a $K > 0$ such that for all $k \geq K$ there holds $\mathrm{rank}(R_{k}) = \mathrm{rank}(X^\star)$. 
%\[
%\mathrm{rank}(X_{k}) = \mathrm{rank}(X^\star) .
%\]
% holds for all $k \geq K$. 
\end{theorem}

To prove the result, we need the help of \emph{partly smoothness}, which was first introduced in \cite{Lewis-PartlySmooth}. 
% This concept, as well as that of identifiable surfaces \cite{Wright-IdentSurf}, captures the essential features of the geometry of non-smoothness which are along the so-called active/identifiable manifold. For convex functions, a closely related idea is developed in \cite{Lemarechal-ULagrangian}. 
%For nuclear norm, though it is non-smooth, however along the set of fixed-rank matrices, it is smooth \cite{Lewis-PartlySmooth}. Such a property can be exactly captured by partial smoothness.
%
Let $\mathcal{M}$ be a $C^2$-smooth embedded submanifold of $\mathbb{R}^n$ around a point $X$. To lighten notation, henceforth, we use $C^2$-manifold instead of $C^2$-smooth embedded submanifold of $\mathbb{R}^n$. The natural embedding of a submanifold $\mathcal{M}$ into $\mathbb{R}^n$ permits to define a Riemannian structure on $\mathcal{M}$, and we simply say $\mathcal{M}$ is a Riemannian manifold. $\mathcal{T}_{\mathcal{M}}(X)$ denotes the tangent space to $\mathcal{M}$ at any point near $X$ in $\mathcal{M}$. 

\begin{definition}[Partial smoothness]\label{dfn:psf}
Let $g$ be proper closed and convex,
$g$ is said to be \emph{partly smooth at $X$ relative to a set $\mathcal{M}$} containing $X$ if $\partial g(X) \neq \emptyset$, and the following smoothness, sharpness, and continuity of $g$ at $X$ relative to $\mathcal{M}$ holds:
\begin{itemize}[leftmargin=2.75cm]%[label={\rm (\roman{*})}]
\item[{\textbf{Smoothness}}:] %\label{PS:C2} 
$\mathcal{M}$ is a $C^2$-manifold around $X$, $g$ restricted to $\mathcal{M}$ is $C^2$ around $X$;
\item[{\textbf{Sharpness}}:] %\label{PS:Sharp} 
The tangent space $\mathcal{T}_{\mathcal{M}}(X)$ coincides with $T_{X} = \mathrm{par}\big(\partial g(X) \big)^\perp$;
\item[{\textbf{Continuity}}:] %\label{PS:DiffCont} 
The set-valued mapping $\partial g$ is continuous at $X$ relative to $\mathcal{M}$.  
\end{itemize}
\end{definition}

For nuclear norm, it is partly smooth along the set of fixed-rank matrices \cite{Lewis-PartlySmooth}. 
Other examples of partly smooth functions including $\ell_1$-norm for sparsity, $\ell_{1,2}$-norm for group sparsity, etc; We refer to \cite{liang2017activity} and the references therein for more examples of partly smooth functions. 

% 
%%%%%%%%%%%%%%%%%%%%%%%%%%%%%%%%%%%%%%%%%%%%%%

\begin{proof}[Proof of Theorem \ref{thm:rank_iden}]
Based on the result of \cite[Theorem~5.3]{hare2004identifying}, to prove rank identification, we need the following conditions: let $X^\star$ be a global minimizer,
\begin{itemize}
\item[(i)] $f + \tau g$ is partial smoothness at $X^\star$ relative to $\mathcal{M}_{X^\star} \eqdef \{ X \in \mathbb{R}^{m\times n} : \mathrm{rank}(X) = \mathrm{rank}(X^\star) \}$.
\item[(ii)] $R_k \to X^\star$ and $(f + \tau g)(R_k) \to (f + \tau g)(X^\star)$.
\item[(iii)] Non-degeneracy condition $0 \in \mathrm{ri}\big(  \nabla f(X^\star) + \tau \partial g(X^\star) \big)$ and 
\[
\mathrm{dist}\big(0, \nabla f(R_k) + \tau \partial g(R_k) \big) \to 0 .
\]
\end{itemize}
%
%%\[
%%\begin{aligned}
%%Y_{k} &= X_k + a_k (X_{k} - X_{k-1}) , \\
%%{Z}_k &= Y_k - \gamma \nabla\Psi \Pa{ (\Psi (Y_k) - F)\odot\widetilde{W} }  , \\
%%R_{k+1} &= \prox_{\tau\gamma \|\cdot\|_*} ({Z}_k) , \\
%%X_{k+1} &= %P \mathcal{T}_{\tau\gamma}\big(S + e_k \big)  Q^T = \prox_{\tau\gamma \|\cdot\|_*} ({Z}_k + c_k) , \quad
%%R_{k+1} + e_k  .
%%\end{aligned}
%%\]
%
%Next, we first prove the identification of $R_{k+1}$ in \eqref{eq:inexact_ifb} by varying the above conditions, and then prove the identification property of $X_{k+1}$. 
Next we prove the identification of $R_{k}$ in \eqref{eq:inexact_ifb} by varying the above conditions. 
\begin{itemize}
\item[--]  
Since $f$ locally is $C^2$-smooth around $X^\star$, the smooth perturbation rule of partly smooth functions \cite[Corollary 4.7]{Lewis-PartlySmooth}, ensures that $f + \tau g$ is partial smoothness at $X^\star$ relative to $\mathcal{M}_{X^\star} \eqdef \{ X \in \mathbb{R}^{m\times n} : \mathrm{rank}(X) = \mathrm{rank}(X^\star) \}$.

\item[--]
By assumption, sequences $R_{k}, {X_{k}}$ of \eqref{eq:inexact_ifb} converge to $X^\star \in \Arg\min\pa{f+\tau g}$. 

\item[--]
The non-degeneracy condition \eqref{eq:cnd-nondeg} is equivalent to $0 \in \mathrm{ri}\big(\partial ( (f+\tau g)(X^\star) )\big)$. 
For $R_{k+1}$, \eqref{eq:Rk} is equivalent to
\[
\begin{aligned}
& Z_k - R_{k+1} \in \gamma\tau  \partial g(R_{k+1}) \\
\Leftrightarrow~~ & Y_k - \gamma \nabla f(Y_k) - R_{k+1} \in \gamma\tau  \partial g(R_{k+1}) \\
\Leftrightarrow~~ & \Pa{ Y_k - \gamma\nabla f(Y_{k}) } - \Pa{ R_{k+1} - \gamma\nabla f(R_{k+1}) }  \in \gamma \partial\pa{f + \tau g}(R_{k+1}) . 
\end{aligned}    
\]
By Baillon-Haddad theorem~\cite{baillon1977quelques}, $\mathrm{Id} - \gamma \nabla f$ is non-expansive, whence we get
\[
\begin{aligned} 
\mathrm{dist}\Pa{0, \partial (f+\tau g)(R_{k+1})}
&\leq \| (\mathrm{Id}-\gamma\nabla f )(Y_{k}) -(\mathrm{Id}-\gamma\nabla f )(R_{k+1}) \|  \\
&\leq \|Y_{k} - R_{k+1} \| \\ 
&\leq \|X_{k} - X_{k+1} \| + a_k \|X_{k}-X_{k-1} \| + \|e_{k} \|   .
\end{aligned}
\]
Since $X_{k}$ is convergent and $\|e_k\| \to 0$, we have $\mathrm{dist}\Pa{0, \partial (f+\tau g)(R_{k+1})} \to 0$.

\item[--]Owing to our assumptions, $f+\tau g$ is sub-differentially continuous at every point in its domain, and in particular at $X^\star$ for $0$, which in turn entails $(f+\tau g)(R_{k}) \to (f+\tau g)(X^\star)$. 

\end{itemize}
Altogether, the above conditions (i)-(iii) are fulfilled, and the rank identification of $R_{k}$ follows. 
\end{proof}
To demonstrate the rank identification property of Algorithm \ref{alg:alg}, the following low-rank recovery problem is considered as an illustration: 
\[
\min_{X \in \mathbb{R}^{100\times 100}} \sfrac{1}{2}\|(\Psi (X) - F)\odot W\|^2+\tau\|X\|_{*} ,
\]
where we have $\Psi \in \mathbb{R}^{2352 \times 10000}$ and 
\[
F = \Psi \mathrm{vec}(\cX) + \varepsilon 
\]
with $\mathrm{rank}(\cX) = 4$ and $\varepsilon$ being random Gaussian noise. Moreover, we choose $\tau = 2\|\varepsilon\|$.

The problem is solved with \program~with $a_k\equiv0, I=1$ and standard proximal gradient descent (PGD) \cite{lions1979splitting}. The step-sizes of both methods are set as $1/L$. The observation is shown below in Figure \ref{Fig:rank_iden}. 
We observed that, both schemes have rank identification property, as the rank of $X_k$ for both schemes eventually becomes constant. Note that, \program~has slower rank identification than PGD, which is caused by the inner loop error. 
%\end{example}

\begin{figure}[!t]
\centering
\includegraphics[width=0.5\textwidth]{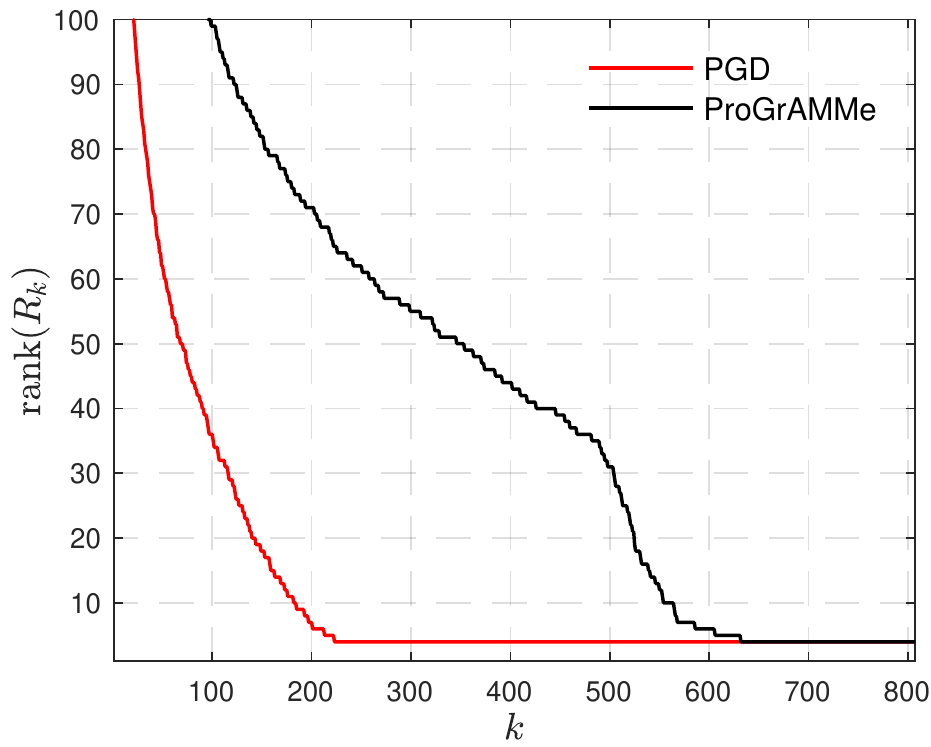}\\[-1mm]
\caption{Rank identification of Algorithm \ref{alg:alg}.}\label{Fig:rank_iden}
\end{figure}

\subsection{Rank continuation}

As we remarked earlier that the choice of $r$ in Algorithm \ref{alg:alg} is crucial to its practical performance. 
To overcome the difficulty of a tight estimation of the rank of minimizers. In this section, we introduce a rank continuation strategy, see Algorithm \ref{alg:alg_rc}, which adaptively adjusts the rank of the output and asymptotically attains the optimal per iteration complexity.

%\begin{algorithm}[!ht]
%	\SetAlgoLined
%	\SetKwInOut{Input}{Input}
%	\SetKwInOut{Output}{Output}
% \SetKwInOut{Init}{Initialize}
% \SetKwInOut{Compute}{Compute}
%\nl\Input{Data matrix $A \in \mathbb{R}^{m \times n}$, weight matrix $W \in \mathbb{R}_{+}^{m \times n}$; Choose $\tau >0$, $r > 0$ and $I \in \mathbb{N}_{+}$\; \jingwei{Care about the dimension, which should be modified from the introduction section...}}
%	 \nl\Compute {$\widetilde{W}=W\odot W, ~L$; Let $\gamma \in ]0, 2/L[$\;}
%	 \nl \While{not converged}
%	{
%		%\nl $\nabla f = -(F-X_k)\odot \widetilde{W}$\;
%		\nl $Y_{k} = X_k + a_k (X_{k} - X_{k-1}) $\;
%		\nl ${Z}_{k} = Y_{k} - \gamma \Psi^T \Pa{ (\Psi Y_k - F)\odot \widetilde{W} }$\;
%		\nl Initialize $U_0\in\R^{m\times r}, V_0\in\R^{r\times n}$\;
%		\nl \For{$i=0,...,I-1$}
%		{
%		\nl $U_{i+1}={{Z}_k}V_{i}^\top \Pa{ V_{i}V_{i}^\top+ {\tau\gamma}\Id_r }^{-1}$\;
%	    \nl $V_{i+1}=\Pa{ U_{i+1}^\top U_{i+1}+{\tau\gamma}\Id_r }^{-1}U_{i+1}^\top{{Z}_k}$\;
%	    }
%	    \nl $X_{k+1}=U_{I}V_{I}$\;
%		\nl \textcolor{blue}{$r=\mathrm{rank}(U_{I})$}\;
%	}
%	\BlankLine
%	\nl \Output{$X_{k+1}$}
%	\caption{\program~with rank continuation}\label{alg:alg_rc}
%\end{algorithm}

\begin{center}
	\begin{minipage}{0.975\linewidth}
		\begin{algorithm}[H]
\caption{\program~with Rank Continuation}\label{alg:alg_rc}
% \caption{A Proximal Gradient and Alternating Minimization Method with Rank Continuation}\label{alg:alg_rc}
\begin{algorithmic}[1]
%\STATE{Data matrix $F$, weight matrix $W $; Choose $\tau >0$, $r > 0$ and $I \in \mathbb{N}_{+}$\;}
\STATE{Compute $\widetilde{W}=W\odot W, ~L$ and et $\gamma \in ]0, 2/L[$; Choose $r > 0$ and $I \in \mathbb{N}_{+}$;}
\WHILE{not convergent}
\STATE{$Y_{k} = X_k + a_k (X_{k} - X_{k-1}) $,\hfill $//$\textcolor{blue}{\tt inertial step}}$//$~~~~~
\STATE{${Z}_{k} = Y_{k} - \gamma \nabla \Psi(Y_k) \Pa{ (\Psi (Y_k) - F)\odot \widetilde{W} }$,\hfill $//$\textcolor{blue}{\tt gradient descent}}$//$~~~~~
\STATE{Initialize $U_0\in\R^{m\times r}, V_0\in\R^{r\times n}$,}
%
% \FOR{$i=0,...,I-1$} 
\STATE{{\bf for} $i=1,...,I-1$ {\bf do}\hfill $//$\textcolor{blue}{\tt inner loop}}$//$~~~~~
\STATE{\qquad$U_{i+1}={{Z}_k}V_{i}^\top \pa{ V_{i}V_{i}^\top+ {\tau\gamma}\Id_r }^{-1}$,}
\STATE{\qquad$V_{i+1}=\pa{ U_{i+1}^\top U_{i+1}+{\tau\gamma}\Id_r }^{-1}U_{i+1}^\top{{Z}_k}$,}
\STATE{{\bf end for}}
% \ENDFOR
%
\STATE{$X_{k+1}=U_{I}V_{I}$,}
\STATE{\textcolor{blue}{$r=\mathrm{rank}(U_{I})$}. \hfill $//$\textcolor{blue}{\tt rank continuation}}$//$~~~~~
\ENDWHILE
\RETURN $X_{k+1}$
\end{algorithmic}
\end{algorithm}
	\end{minipage}
\end{center}

% \jingwei{Test inertial for the inner loop...} 

\begin{remark}
In {\tt Line 11}, instead of using $\mathrm{rank}(X_{k+1})$ to update $r$, we choose to use $\mathrm{rank}(U_I)$, since $\mathrm{rank}(X_{k+1}) \leq \min\{ \mathrm{rank}(U_I), \mathrm{rank}(V_I) \}$ and it is less computational demanding to evaluate the rank of $U_I$ than that of $X_{k+1}$. 
\end{remark}

%{\color{blue}
\begin{remark}
Note that in Theorem \ref{thm:rank_iden}, we only have rank identification property for $R_k$ and not for $X_k$. However, practically this is not an issue since we can start the rank continuation late enough such that $e_k$ is small enough and $\mathrm{rank}(X_{k+1}) \geq \mathrm{rank}(R_{k+1})$.
\end{remark}
%}

\begin{remark}
Though the initial value of $r$ is no longer as important as that of Algorithm \ref{alg:alg} whose $r$ is fixed, it is still beneficial to have a relatively good estimate of $\mathrm{rank}(X^\star)$ as it can further reduce the computational cost of the algorithm. % See our numerical example below in Figure \ref{fig:comp_rc_1}. 
Also, it is not desirable to compute $\mathrm{rank}(U_I)$ every iteration and a practical approach is to do it every certain number of steps. 
\end{remark}

\begin{remark}
In \cite{JMLR:v11:mazumder10a}, Mazumder et al. showed that their SOFT-Impute algorithm lies in the two-dimensional maximum margin matrix factorization (MMMF) algorithm family \cite{NIPS2004_e0688d13}. That is, for each given maximum rank, SOFT-IMPUTE performs rank reduction and shrinkage simultaneously. However, we note that, our rank continuation strategy is different, because, we use a general weight, and perform an alternating minimization under the framework of accelerated proximal gradient. 
\end{remark}

%\begin{example}[label={ex:example2},listing only]{Comparison}
To illustrate the performance of rank continuation, we consider a low-rank recovery problem with randomly missing entries: 
\begin{equation}\label{eq:mc}
\min_{X \in \mathbb{R}^{2000\times 2000}} \sfrac{1}{2}\|(\Psi \odot X - F)\odot W\|^2+\tau\|X\|_{*} ,
\end{equation}
where $\Psi \in \mathbb{R}^{2000 \times 2000}$ is a random binary mask with $50\%$ entries equal to $0$ and 
\[
F = \Psi \odot \cX + \varepsilon 
\]
with $\mathrm{rank}(\cX) = 10$ and $\varepsilon$ being random Gaussian noise. For this case we set $\tau = \|\varepsilon\|$. 

%\begin{table}[tbhp]
%{\footnotesize
%  \caption{Example table}\label{tab:simpletable}
%\begin{center}
%  \begin{tabular}{|c|c|c|} \hline
%   Species & \bf Mean & \bf Std.~Dev. \\ \hline
%    1 & 3.4 & 1.2 \\
%    2 & 5.4 & 0.6 \\ \hline
%  \end{tabular}
%\end{center}
%}
%\end{table}

\paragraph{Comparison to SVD based PGD/FISTA.}
We first compare the performances of PGD/FISTA, Algorithm \ref{alg:alg} (\program) and Algorithm \ref{alg:alg_rc} (\program-RC), with the following settings:
\begin{itemize}
\item 
The step-size $\gamma$ for all these schemes are chosen as $\gamma = 1/L$. Note that this choice of $\gamma$ exceeds the upper bound of step-size of FISTA, however for this example FISTA converges. 
\item 
For FISTA, we update $a_k$ using $a_k = \frac{k-1}{k+20}$ as proposed in \cite{liang2018improving} which   provides faster performance than that of the standard FISTA scheme. 
\item 
For both Algorithm \ref{alg:alg} (\program) and Algorithm \ref{alg:alg_rc} (\program-RC), we choose $I = 1$ and $a_k \equiv 0$. We set $r=1,000$ for \program~which is also the initial value of $r$ for Algorithm \ref{alg:alg_rc}.
\end{itemize} 
All schemes are stopped when the relative error $\|X_{k}-X_{k-1}\|$ reaches $10^{-10}$, and we note the the average wall clock time of $10$ runs for these schemes as:
\[
\begin{matrix}
\hline
 \textrm{Schemes} & \textrm{PGD} & \textrm{FISTA} & \textrm{\program} & \textrm{\program-RC} \\
\hline
\textrm{Time (seconds)} & 276.3 & 198.3 & 47.9 & 34.2 \\
\hline
\end{matrix}
\]
As illustrated in Figure \ref{fig:comp_rc_1}~(a), we observe: 
\begin{itemize}
\item PGD/FISTA is much slower than Algorithm \ref{alg:alg} and Algorithm \ref{alg:alg_rc}. In particular, the rank continuation scheme is about an order faster than PGD/FISTA. 
\item  
%For both algorithms, the inertial version, \ie $a_k = \frac{1}{4}$ is slightly faster than the non-inertial one. 
\program-RC is about $30\%$ faster than \program~which indicates the advantage of rank continuation under the considered setting. 
\end{itemize}
For the {\it magenta} line in Figure \ref{fig:comp_rc_1}~(a) for \program-RC, it has several jumps which is due to the update of $r$.

%\begin{remark}
%For FISTA scheme, one can also consider the adaptive restarted scheme of \cite{o2015adaptive}, which can further improve the performance of FISTA. However, as long as SVD is involved, it cannot compete with \program~for the considered relative large-scale problem. 
%\end{remark}

\begin{figure}[!t]
\centering
\subfloat[Comparison against PGD/FISTA]{ \includegraphics[width=0.45\linewidth]{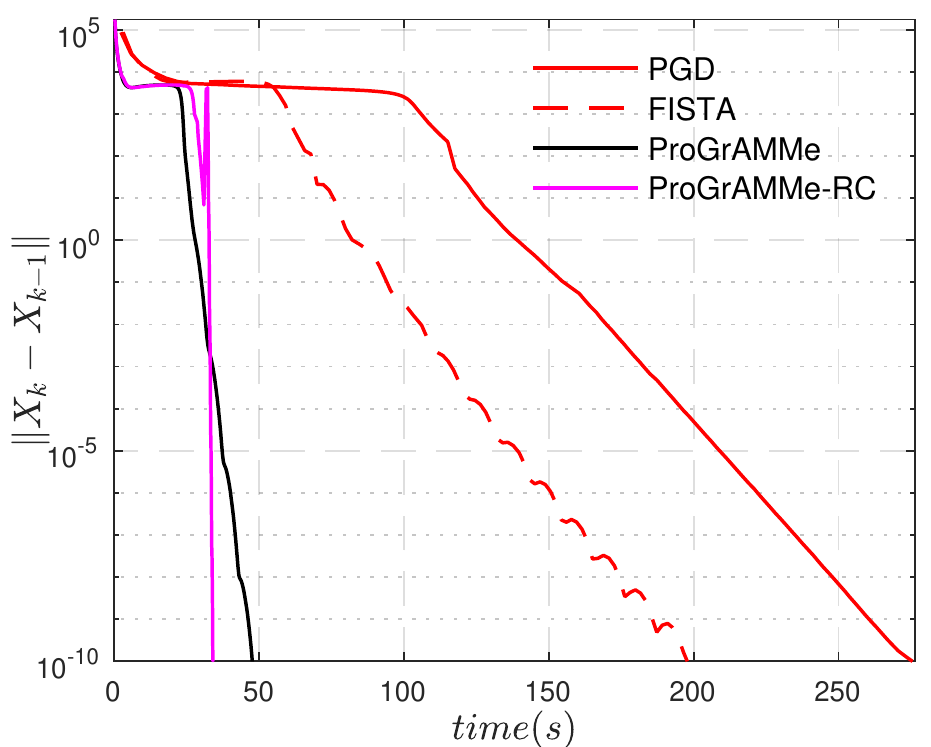} }  \hspace{2pt}
\subfloat[Different initial values of $r$]{ \includegraphics[width=0.45\linewidth]{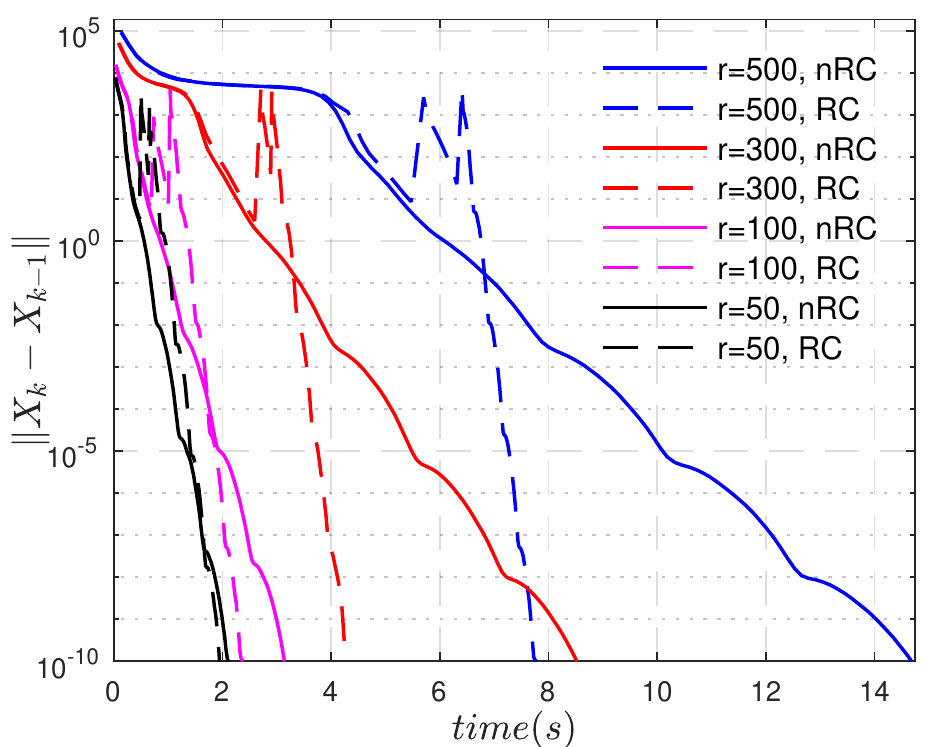} }  \\[-2mm]
%%%%%%%%
\caption{Comparison of PGD/FISTA, Algorithm \ref{alg:alg} and Algorithm \ref{alg:alg_rc}. (a) Comparison against PGD/FISTA; (b) Comparison between rank continuation and no continuation. Note that the lines of rank continuation schemes have several jumps which is due to the update of $r$.}
\label{fig:comp_rc_1}
\end{figure}

\paragraph{Effect of different starting rank.}
To further understand the advantage of rank continuation over the static one, we conduct a comparison of Algorithm \ref{alg:alg_rc} under different initial values for $r$. Precisely, we consider
\begin{itemize}
\item 
four different values of $r = 500, 300, 100, 50$. 
\item $r$ is updated every $10$ steps.
\end{itemize}
We use ``nRC'' to denote \program~{\it without} rank continuation and ``RC'' {\it with} rank continuation, and the result is shown in Figure \ref{fig:comp_rc_1}~(b).~We observe
\begin{itemize}
\item Without rank continuation, the smaller the value of $r$, the better the performance of \program.
\item For $r=500, 300$, the red and black lines in the figure, rank continuation (dashed lines) shows clear advantage over the standard scheme (solid lines).
\item While for $r=100,50$, rank continuation actually becomes slower than the static scheme, and the extra time is mainly the overhead of computing $\mathrm{rank}(U_I)$. 
\end{itemize}
From the above observations, we conclude:
\begin{itemize}
\item For problems where a tight estimation of the rank of the solution can be obtained, one can simply consider Algorithm \ref{alg:alg}; 
\item When the rank of solutions is difficult to estimate, then rank continuation can be applied to achieve acceleration. 
\end{itemize}
%\end{example}
We leave the comparison of \program~with inertial to the next section.

\section{Numerical experiments}\label{sec:results}

To understand the effects of inertial acceleration, validate the strengths and flexibility of our recovery model and algorithm, in this section we perform numerical experiments on several low-rank recovery problems.~Throughout this section, we typically use two different versions of \program---\myNum{i} \program-$1$ which terminates the inner loop in each iteration and \myNum{ii} \program-$\epsilon$ which terminates the inner loop when the relative error of the inner iterates reach an $\epsilon$ precision or maximum inner iteration is achieved, whichever occurs first. Throughout the section, for \program-$\epsilon$, we use $\epsilon=10^{-4}$ and maximum number of inner iteration is set to 20, unless otherwise specified.

\subsection{Effects of inertial acceleration}

We continue the matrix completion problem \eqref{eq:mc} to study the effect of inertial acceleration. Both Algorithm \ref{alg:alg} and \ref{alg:alg_rc} are tested, the setting of the tests are
\begin{itemize}
\item 
For both algorithms, we initialize $r$ with value of $500$; In terms of step-size, we keep the previous choice which is $\gamma = 1/L$.
\item 
In total, $5$ different choices of inertial parameter $a_k$ are considered
\[
a_k \equiv 0 ,\enskip
a_k \equiv \sfrac{1}{4} ,\enskip
a_k \equiv \sfrac{1}{2} ,\enskip
a_k \equiv \sfrac{3}{4} \quad\mathrm{and}\quad
a_k = \sfrac{k-1}{k+20} .
\]
\end{itemize}
The results are shown in Figure \ref{fig:cmp_inertial}, whose left figure is the comparison of Algorithm \ref{alg:alg} without rank continuation
\begin{itemize}
\item 
In general, relatively small inertial parameters ($a_k \equiv \frac{1}{4}, \frac{1}{2}$) provide acceleration inertial schemes. %For example, $a_k\equiv\frac{3}{4}$ is about $30\%$ faster than $a_k\equiv0$.
\item 
For $a_k \equiv \frac{3}{4}$ and $a_k = \sfrac{k-1}{k+20}$ are slower than the case $a_k\equiv0$. 
\end{itemize}
The above observations are quite similar to the comparison of inertial schemes to proximal gradient schemes \cite{liang2017activity}. 
The comparison for rank continuation scheme is provided in Figure \ref{fig:cmp_inertial}~(b), where similar observation as above can be obtained, except that for this case $a_k \equiv \frac{3}{4}$ and $a_k = \sfrac{k-1}{k+20}$ are faster than $a_k\equiv0$. 

\begin{figure}[!t]
\centering
\subfloat[No rank continuation]{ \includegraphics[width=0.45\linewidth]{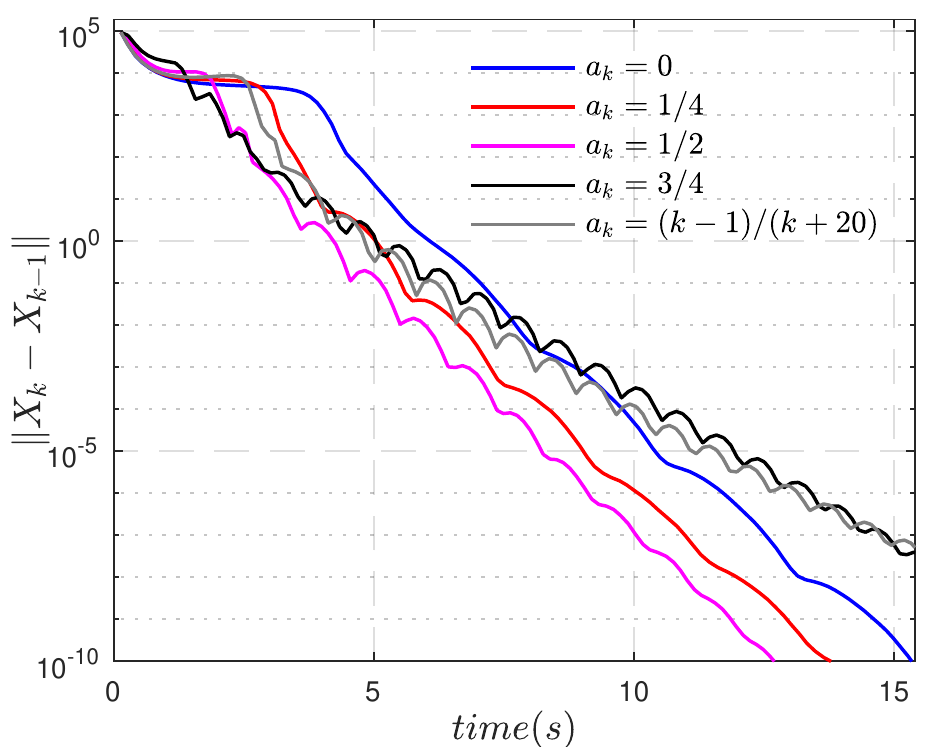} }  \hspace{2pt}
\subfloat[Rank continuation]{ \includegraphics[width=0.45\linewidth]{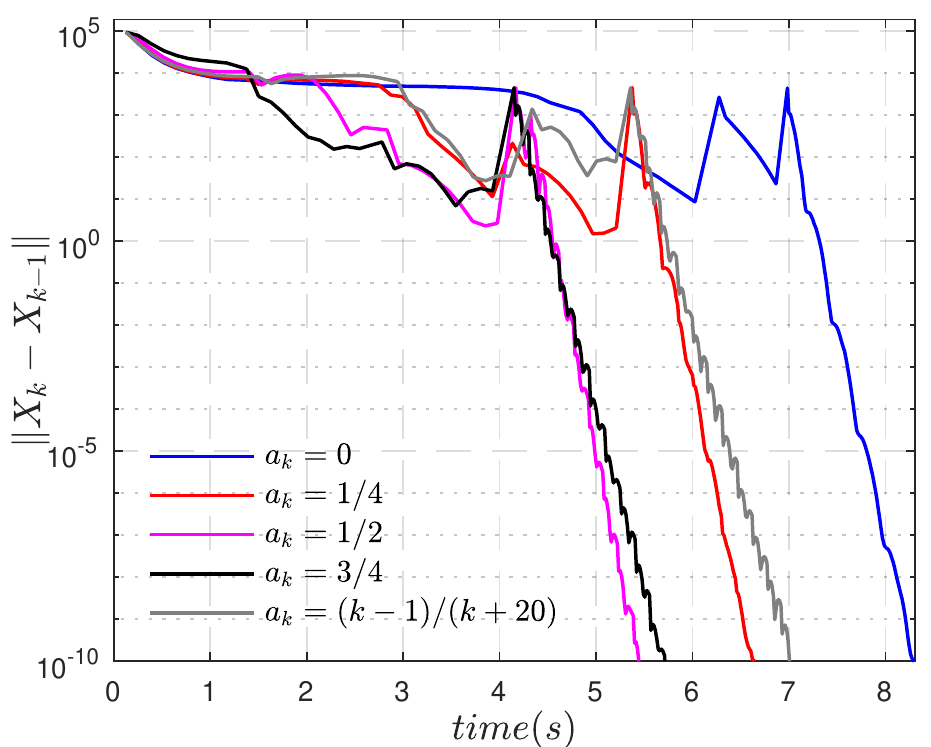} }  \\[-1mm]
%%%%%%%%
\caption{Effects of inertia:~(a) No rank continuation; (b) Rank continuation.}
\label{fig:cmp_inertial}
\end{figure}

\begin{figure}[!ht]
\centering
\subfloat[Random noise]{ \includegraphics[width=0.75\textwidth]{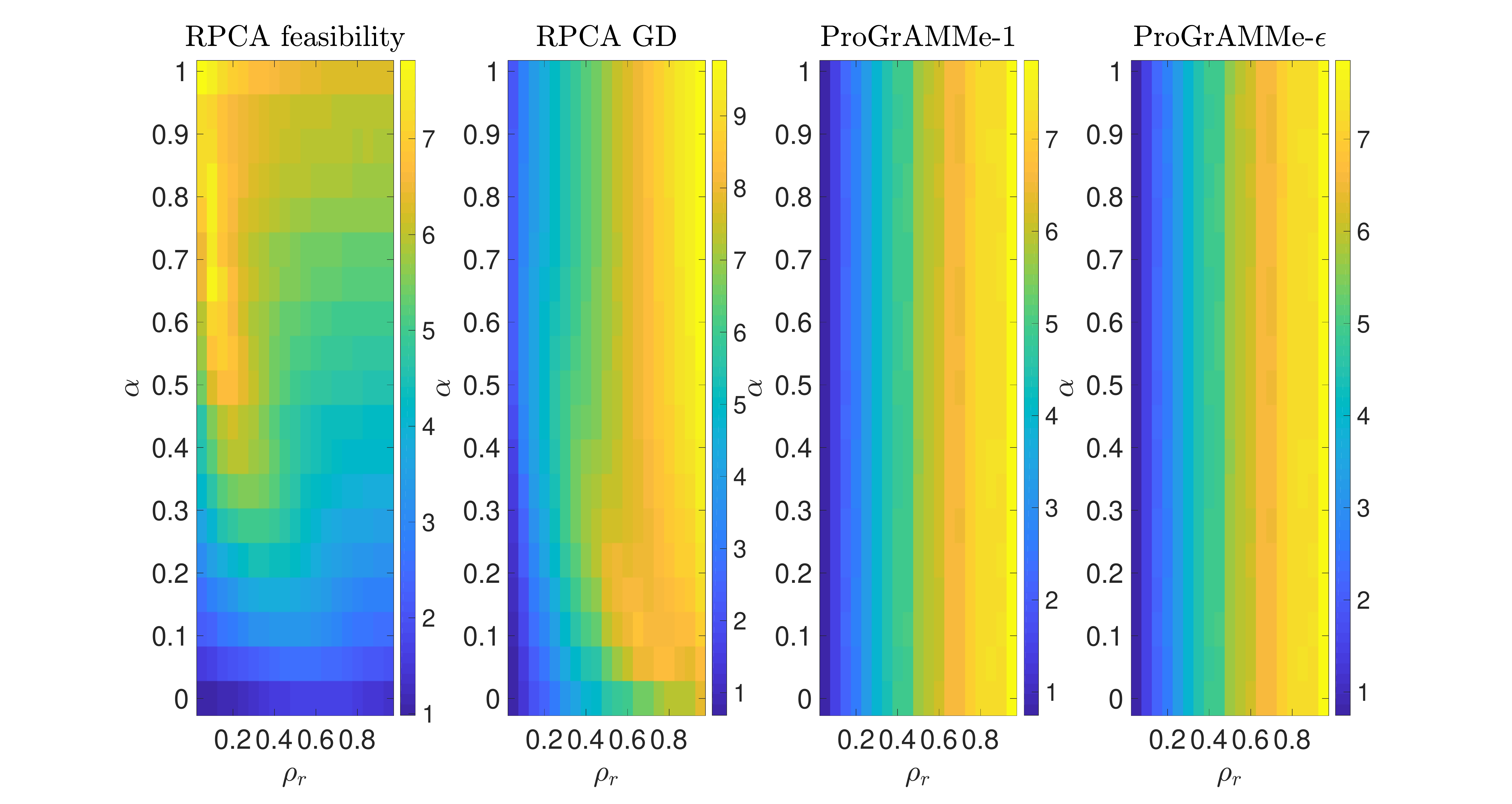} } \\[-0.5pt]
\subfloat[Sparse large noise]{ \includegraphics[width=0.75\textwidth]{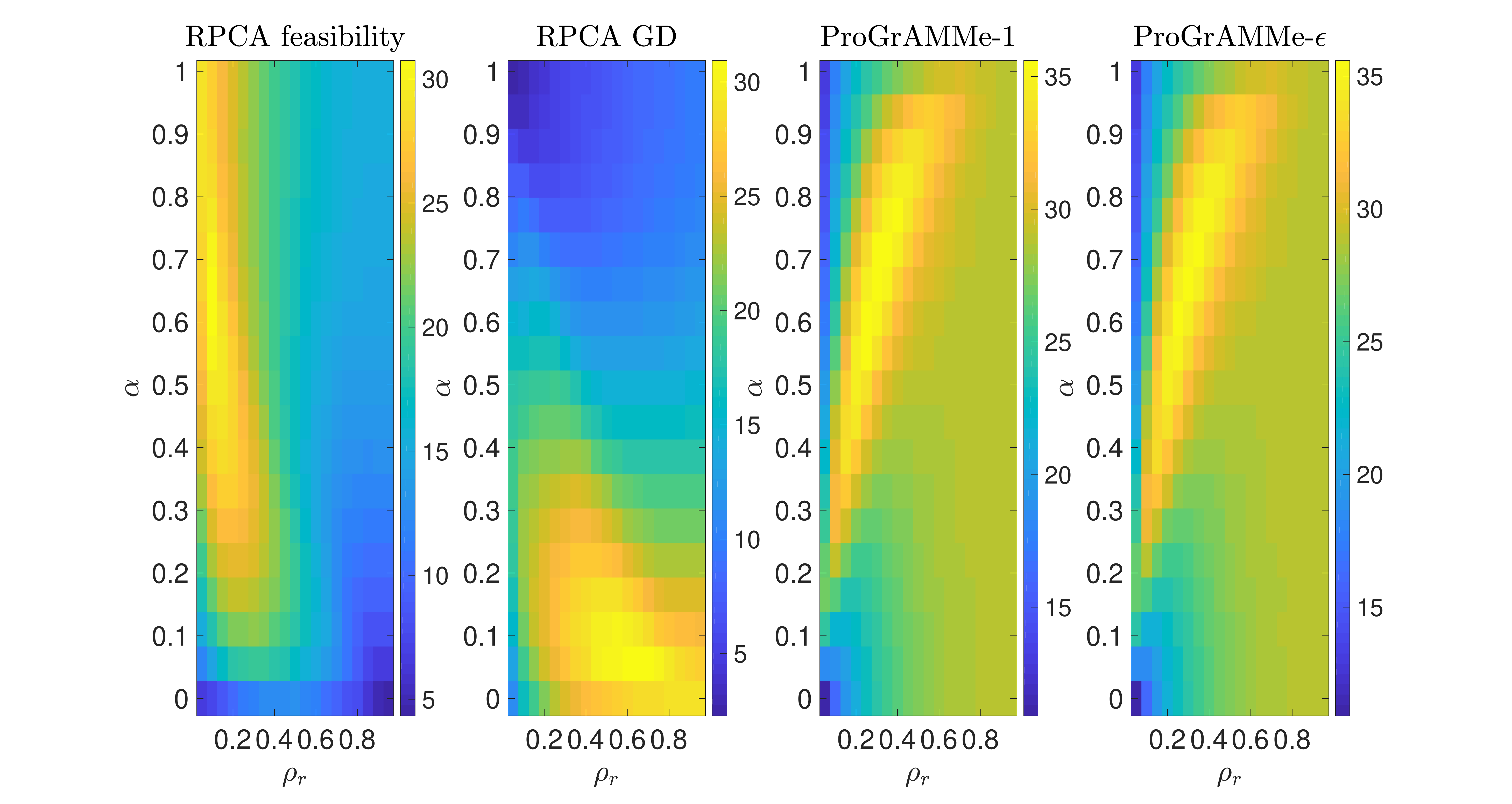} } \\[-2mm]
    %%%%%%%%
\caption{{RMSE: $\|F-X\|/\sqrt{mn}$ for different methods. Top row is for random noise, bottom row is for sparse large noise. For both cases, $\min_{i,j}W_{ij}=5,\max_{i,j}W_{ij}=10.$}}\label{RMSE}
\end{figure}

%\begin{figure}[!ht]
%\centering
%\begin{minipage}{\textwidth}
%    \centering
%    \includegraphics[width=\textwidth,height=1.65in]{Results/total_rmse_random_new_no_inertial.eps}
%\end{minipage}\\
%\begin{minipage}{\textwidth}
%    \centering
%    \includegraphics[width=\textwidth, height=1.65in]{Results/total_rmse_sparse_new_no_inertial.eps}
%\end{minipage}\\[-1mm]
%%% \begin{minipage}{0.3\textwidth}
%%%     \centering
%%%     \includegraphics[width=\linewidth]{Results/RPCA_F_random.eps}
%%% \end{minipage}
%%% \begin{minipage}{0.3\textwidth}
%%%     \centering
%%%     \includegraphics[width=\linewidth]{Results/WLR_ADMM_random.eps}
%%% \end{minipage}\\
%%% \begin{minipage}{0.3\textwidth}
%%%     \centering
%%%     \includegraphics[width=\linewidth]{Results/RPCA_GD_sparse.eps}
%%% \end{minipage}
%%% \begin{minipage}{0.3\textwidth}
%%%     \centering
%%%     \includegraphics[width=\linewidth]{Results/RPCA_F_sparse.eps}
%%% \end{minipage}
%%% \begin{minipage}{0.3\textwidth}
%%%     \centering
%%%     \includegraphics[width=\linewidth]{Results/WLR_ADMM_sparse.eps}
%%% \end{minipage}
%\caption{{RMSE: $\|F-X\|/\sqrt{mn}$ for different methods. Top row is for random noise, bottom row is for sparse large noise. For both cases, $\min_{i,j}W_{ij}=5,\max_{i,j}W_{ij}=10.$}}\label{RMSE}
%\end{figure}

\subsection{Low-rank recovery experiments on synthetic data}
For these experiments, we generate the low-rank matrix, $L$, as a product of two independent full-rank matrices of size $m\times r$ with $r<m$ such that elements are independent and identically distributed (i.i.d.) and sampled from a normal distribution---$\cN(0,1)$.~We used two different types of sparse noise---Gaussian noise and arbitrary large noise.~For each case, we generate the sparse matrix, $S$, such that for a sparsity level $\alpha\in(0,1)$, the sparse support is created randomly. 
\begin{itemize}
\item
For random Gaussian noise, we construct the sparse matrix $S_{\rm random}$ whose elements are i.i.d. $\cN(0,1)$ random variables and form $F$ as: $F=L+\eta S_{\rm random}$, where $\eta$ controls the noise level and we set ${\eta=0.2\max_{i,j}(L)_{ij}}$. 
\item
For large noise, we generate the sparse matrix, $S_{\rm sparse}$, such that its elements are chosen from the interval $[-50,50]$ and construct $F$ as $F = L+S_{\rm sparse}$. 
\end{itemize}
We fix $m=100$, define $\rho_r={\rm rank}(L)/m$, where ${\rm rank}(L)$ varies and set the sparsity level $\alpha\in(0,1)$. For each $(\rho_r,\alpha)$ pair, we apply RPCA GD \cite{RPCAgd},~NCF \cite{duttahanzely}, and \program~to recover a low-rank matrix $X$. We consider RMSE, ${\|F-X\|/\sqrt{mn}}$, as performance measure. 

For each class of noise, we run the experiments for 10 times and plot the average RMSEs for each $(\rho_r,\alpha)$. Note that, RPCA GD and NCF use an operator $\mathcal{T}_{\alpha}[S]$ that does not perform an explicit Euclidean projection onto the sparse support of $S$, as the exact projection on $S$ is expensive \cite{duttahanzely,RPCAgd,dutta2019best}. 
Inspired by this, for sparse noise, we design our weight matrix such that it has large weights for the sparse support of $S$ and 1 otherwise. However, for random noise, we simply use a random weight matrix. From Figure \ref{RMSE} we see that for both random and sparse noise, \program~has the least average RMSEs. Moreover, while two PCP algorithms show significant differences in their RMSE diagrams, our \program~produce almost similar RMSE and obtain lower values compare to PCP algorithms in both types of noises. 

\paragraph{Effect of the condition number of $W$ on the convergence of \program.} Problem \eqref{prblm:wlr_id} is tricky, as the condition number, $\kappa_W$ of the weight matrix, $W$ plays an important role in convergence \cite{Pong, duttaligongshah}. We perform a detailed empirical convergence analysis of \program-$\epsilon$ and \program-$1$ on synthetic data in Appendix \ref{sec:appendix_conv} by varying $\kappa_W$ and compared with proximal algorithms. We observe that proximal algorithms are sensitive to $\kappa_W$ (a higher $\kappa_W$, translates to a slower convergence), but both \program-$\epsilon$ and \program-$1$ are not sensitive to $\kappa_W$, they maintain a stable convergence profile for different $\kappa_W$, and converge faster than proximal algorithms in all cases. While APG takes about 0.13 seconds on an average for each experiment, \program-1 and \program-$\epsilon$ take about 0.01 seconds and 0.04 seconds, respectively. % \ad{Based on the Remark \ref{remark:cond number}, we also show the changed condition number of the problem for different choices of $W$.} 

\subsection{Real-world applications}
To validate the strengths and flexibility of our proposed algorithms, we use three real world problems---\myNum{i} structure from motion~(SfM), \myNum{ii} matrix completion with noise, and \myNum{iii} background estimation from fully and partially observed data. We compared our algorithms against 15 state-of-the-art weighted and unweighted low-rank approximation algorithms~(see Table \ref{algo} in Appendix).  

\begin{figure}[!t]
    \centering
    \includegraphics[width=0.81\textwidth]{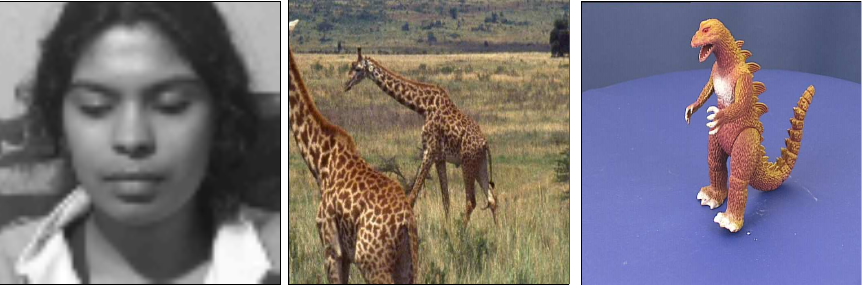} \\[-1mm]
\caption{{Sample frame from the static {\tt Face}, {\tt Giraffe} and {\tt toy dinosaur} sequences. The data matrices, $F$ are of size $2944\times 20$, $240\times 167$, and $319\times 72$, respectively and the prior rank of the sequences are 4, 6, and 4, respectively.}}\label{Fig:sfm_data}
\end{figure}

\paragraph{\myNum{i} Structure from motion and photometric stereo.}
SfM uses local image features without a prior knowledge of locations or pose and infers a three dimensional structure or motion.~For these experiments, we used three popular datasets\footnote{http://www.robots.ox.ac.uk/~abm/}:~non-rigid occluded motion of a giraffe in the background~(for nonrigid SfM), a toy dinosaur~(for affine SfM), and the light directions and surface normal of a static face with a moving light source~(for photometric stereo)~(See more details in Figure \ref{Fig:sfm_data}).~The datasets have 69.35\%, 23.08\%, and 58.28\% observable entries, respectively. Therefore, we use a binary mask as weight $W$ such that $W_{ij}=1$ if the data has an entry at the $(i,j)^{\rm th}$ position, otherwise, $W_{ij}=0.$ With this setup, our formulation works as a matrix completion problem. We compared \program-1 and \program-$\epsilon$ with respect to the damped Newton algorithm in~\cite{Buchanan}. Admittedly, \cite{Buchanan} obtains the best factorization pair, $X=UV$ such that it gives minimum loss (within the observable entry), $\frac{\|(F-X)\odot W\|}{\|W\|}$ for all cases. Additionally, we also calculated the loss outside the observable entries, that is, $\frac{\|(F-X)\odot(\mathbf{1}- W)\|}{\|\mathbf{1}- W\|}$ where the performance of \program~is better in all cases. See results provided in Table \ref{table:sfm} below.

\begin{table}[!t]
\begin{center}
\begin{tabular}{ccc}
\hline
{ \bf Dataset} & {$\frac{\|(X-F)\odot W\|}{\|W\|}$} &{$\frac{\|(X-F)\odot(\mathbf{1}- W)\|}{\|\mathbf{1}- W\|}$}\\
\hline
          & 0.5085~(\program-$1$)               & 284.1509~(\program-$1$)\\
{\tt Giraffe} & 0.5399~(\program-$\epsilon$)                  & {\bf 271.4828}~(\program-$\epsilon$)\\
          & {\bf 0.3228}~\cite{Buchanan}   & 364.2476~\cite{Buchanan}\\
\hline
          & 4.4081~(\program-$1$)               & 244.5437~(\program-$1$)\\
{\tt Toy Dino}    & 5.2453~(\program-$\epsilon$)                 & {\bf 239.6335}~(\program-$\epsilon$)\\
          & {\bf 1.0847}~\cite{Buchanan}    & 318.666~\cite{Buchanan}\\
\hline
          & 0.023~(\program-$1$)                & {\bf 0.7143}~(\program-$1$)\\
{\tt Face}    & 0.0232~(\program-$\epsilon$)                  &~{ 0.7338}~(\program-$\epsilon$) \\
          & {\bf 0.0223}~\cite{Buchanan}    & 0.98~\cite{Buchanan}\\
\hline
\end{tabular}
\end{center}
\caption{{Comparisons between \program-$\epsilon$, \program-$1$, and damped Newton method on structure from motion and photometric stereo datasets.}}\label{table:sfm}
\end{table}

\paragraph{\myNum{ii} Matrix completion with noise on power-grid data.} Matrix completion is one of the important special cases of weighted low-rank estimation problems as for this problem, the weights are reduced to $\{0,1\}$.~This set of experiments are inspired by \cite{regularizedlow_das}. The dataset and the codes for BIRSVD are collected from author's website\footnote{https://homepage.univie.ac.at/saptarshi.das/index.html}. In this experiment, the test dataset contains 48 hours of temperature data sampled every 30 minutes over 20 European cities. The $(i,j)^{\rm th}$ entry of the matrix represents the $i^{\rm th}$ temperature measurement for the $j^{\rm th}$ city. As a prior information, we use the fact that temperature varies smoothly over time and the data matrix is low-rank.

\begin{figure}[!t]
\centering
\subfloat[Within data]{     \includegraphics[width=0.45\linewidth]{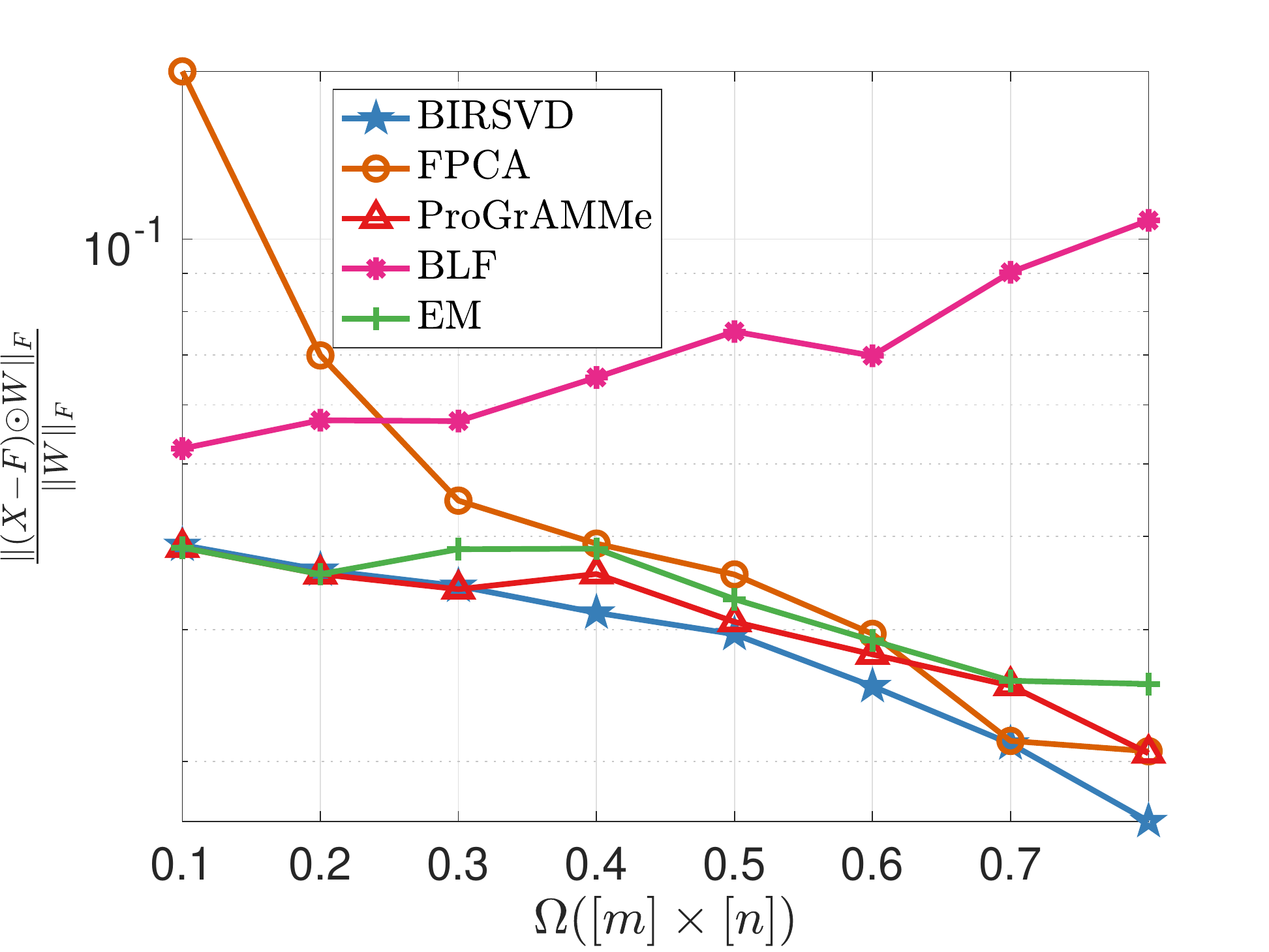} }  \hspace{2pt}
\subfloat[Out of data]{     \includegraphics[width=0.45\linewidth]{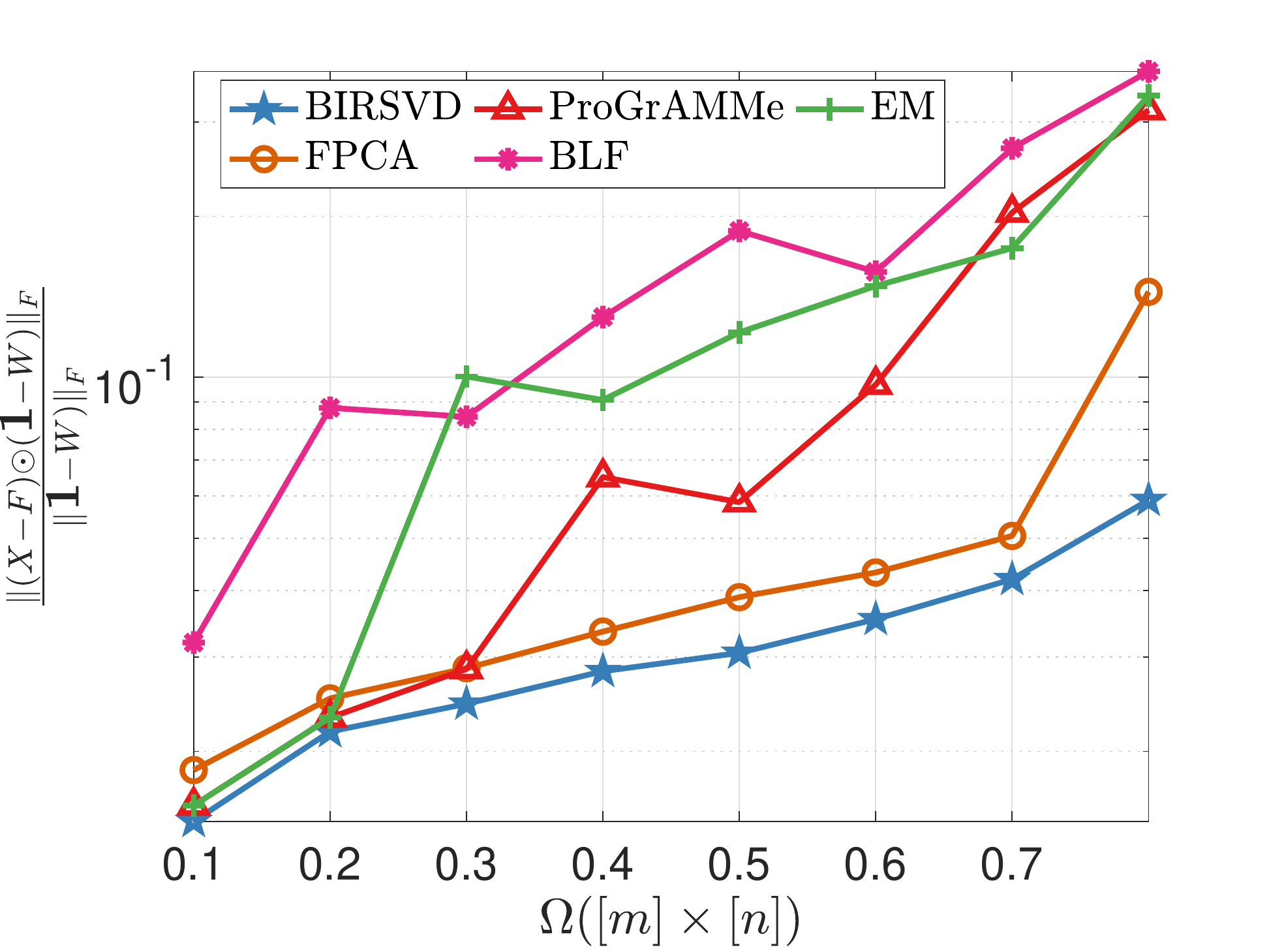} } \\[-2mm]
    %%%%%%%%
\caption{{Fidelity within and out of the data for problem of matrix completion with noise on power-grid data. Here $\Omega$ denotes the percentage of missing data. We use \program-$1$ for these experiments.}}\label{Fig:fid_in_out} 
\end{figure}

For this data set, we use a rank 2 approximation. We run each algorithm with random initialization for 10 times and plot the average results in Figure \ref{Fig:fid_in_out}. Each algorithm is run for 50 global iterations or tolerance set to machine precision, whichever attained first. For BIRSVD, \program-$1$, and~BLF~(with $\ell_2$ loss) we use $\tau = 0.006$. Additionally, for BIRSVD another regularizer is set to 0 (as in \cite{regularizedlow_das}). FPCA \cite{ma2011fixed} computes a low rank approximation by using a prior information that the nuclear norm of the matrix is bounded. As it turns out, the EM algorithm \cite{srebro} has a slower convergence and the performance of \program~is albeit better. For more discussion and detailed comparisons we refer the readers to Section \ref{sec:appendix_weather} in Appendix.

\paragraph{\myNum{iii} Background estimation.}
Background estimation and moving object detection from a video sequence is a classic problem in computer vision and it plays an  important role in human activity recognition, tracking, and video analysis from surveillance cameras. In the conventional matrix decomposition framework used for background estimation, the video frames are concatenated in a data matrix $F$, and the background matrix, $X$, is of low-rank~\cite{oliver1999}, as the background frames are often static or close to static. However, the foreground is usually sparse. The desired target rank of the background is hard to determine due to some inherent challenges, such as, changing illumination, occlusion, dynamic foreground, reflection, and other noisy artifacts. Therefore, robust PCA algorithms such as, iEALM \cite{LinChenMa}, APG \cite{APG}, ReProCS \cite{reprocs} overcome the rank challenge robustly. However, in some cases, the target rank and the sparsity level can be part of the user-defined hyperparameters.~Therefore, instead one might use a different approach as in \cite{godec,duttahanzely,RPCAgd}.

For these set of experiments, we use \program-$\epsilon$ and compared with two different types robust PCP formulations\footnote{See \cite{wen2019nonconvex} for an overview.}. Note that, the PCP formulations have a way to detect the sparse outliers but our formulation does not. We overcome this challenge by using large weights. Similar to the heuristic used in the experiments for synthetic data, we choose a random subset of entries from the set $[m]\times [n]$ and use a large range of weight at those elements. This is similar to the idea used in \cite{duttahanzely,RPCAgd,dutta2019best}, as for specific sparsity percentage $\alpha$, the operator $\mathcal{T}_{\alpha}[S]$ performs an approximate projection onto the sparse support of $S$.~We argue that randomly selecting $\alpha\%$ of elements from ${[m]\times [n]}$ and hitting them by large weights, we obtain the same artifact. Indeed our empirical evidence in Figure \ref{Fig:bg_st} justifies that.

\begin{figure}[!t]
    \centering
    \includegraphics[width=\textwidth, height=2in]{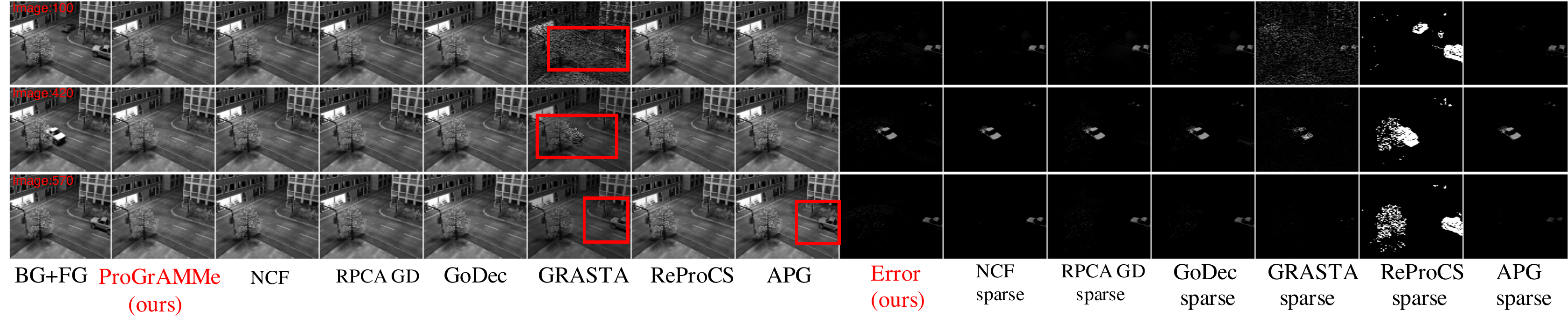} \\[-1mm]
            %\vspace{-5pt}
\caption{{Sample video frames from the Stuttgart {\tt Basic} video sequence.~\program-$\epsilon$ provides a visually high quality background.}}\label{Fig:bg_st}
\end{figure}

% \aritra{BG estimation full data Stuttgart- \program-$\epsilon$\\
%  BG estimation full data CD Net-\program-$\epsilon$\\
%  BG estimation partial data Stuttgart Net-\program-1\\}

In our experiments, we use {\em eight} different video sequences:~(i)~the {\tt Basic} sequence from Stuttgart synthetic dataset~\cite{cvpr11brutzer}, (ii)~four video sequences from~CDNet2014 datasets~\cite{wang2014cdnet}, and (iii) three video sequences from SBI dataset \cite{sbi_a,SBI_web}. We extensively use the Stuttgart video sequence as it is equipped with foreground ground truth for each frame. For iEALM and APG, we set $\lambda=1/\sqrt{m}$, and use $\mu=1.25/\|F\|_2$ and $\rho=1.5$, where $\|F\|_2$ is the spectral norm (maximum singular value) of $F$. For Best pair RPCA, RPCA GD, NCF, GoDec, and our \program-$\epsilon$, we set $r=2$, target sparsity 10\% and additionally, for GoDec, we set $q=2$. For GRATSA, we set the parameters the same as those mentioned in the authors' website\footnote{{https://sites.google.com/site/hejunzz/grasta}}. The qualitative analysis on the background and foreground recovered suggest that our method recovers a visually similar or better quality background and foreground compare to the other methods. Note that, RPCA GD and ReProCS recover a fragmentary foreground with more false positives; moreover, GRASTA, iEALM, and APG cannot remove the static foreground object. See Section \ref{sec:appendix_BG} in Appendix for more qualitative results (Figure \ref{Fig:bg_2}) and detailed quantitative results~(Figure \ref{fig:quant_bg}) of MSSIM and PSNR.

%\vspace{-5pt}
\paragraph{\myNum{iv} Background estimation from partially observed/missing data.}
We randomly select the set of observable entries in the data matrix, $F$ and perform our experiments on Stuttgart {\tt Basic} video. For these experiments, we use \program-$1$. As this is a missing data case, for \program-$1$, we use a binary mask as the weight. Figure \ref{Fig:bg_st_partial} shows the qualitative results on different subsampled video. For a detailed quantitative evaluation of  \program~with respect to the $\epsilon$-proximity metric--$d_{\epsilon}(X,Y)$ as in \cite{duttahanzely} in recovering the foreground objects and to see the execution time for different missing data cases, see Figure \ref{Fig:missng_data_bg} in Section \ref{sec:appendix-missingdataBG} of Appendix.

\begin{figure}[!t]
    \centering
   \includegraphics[width=\textwidth]{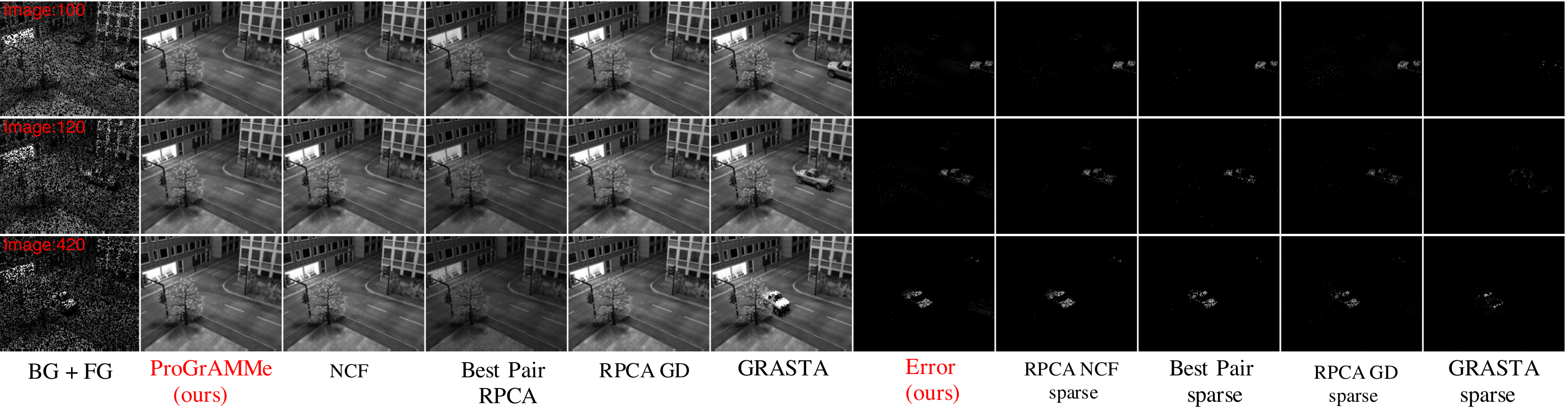} \\[-1mm]
   %\vspace{-5pt}
\caption{{Sample video frames from the Stuttgart {\tt Basic} video sequence for the missing data case. From top to bottom we use $\Omega=0.8,0.7,$ and 0.6, respectively.}}\label{Fig:bg_st_partial}
\end{figure}

\section{Conclusions}

In this paper, we proposed a generic weighted low-rank recovery model and designed an SVD-free fast algorithm for solving the model. 
Our model covers several existing low-rank approaches in the literature as special cases and can easily be extended to the non-convex setting.
Our proposed algorithm combines proximal gradient descent method and the variational formulation of nuclear norm, which does not require to compute the SVD in each step. This makes the algorithm highly scalable to larger data and enjoys a lower per iteration complexity than those who require SVD.~Moreover, based on a rank identification property, we designed a rank continuation scheme which asymptotically achieves the minimal per iteration complexity. Numerical experiments on various problems and settings were performed, from which we observe superior performance of our proposed algorithm compared to a vast class of weighted and unweighted low-rank algorithms.

\appendix

\section{Addendum to the numerical experiments}
In this section, we added some extra numerical experiments that complement our experiments and other claims in the main paper.

\subsection{Table of baseline methods} \label{sec:table}
In Table \ref{algo} we summarize all the methods compared in this paper.

\begin{table}
\small
\begin{center}
\begin{tabular}{lllll}
\hline
{\bf Algorithm} & {\bf Abbreviation} &{\bf Appears in}& {\bf Ref.}\\
\hline
\begin{tabular}{l}
{\hspace{-2mm}}Inexact Augmented Lagrange \\{\hspace{1mm}}
Method of Multipliers 
\end{tabular}
& iEALM & Fig. \ref{fig:quant_bg} &\cite{LinChenMa}  \\
\hline
Proximal Gradient & PG & Fig. \ref{Fig:conv_fn}, \ref{Fig:conv_iter} &\cite{Ji_APG,Shen_anaccelerated} \\
\hline
Accelerated Proximal Gradient & APG & Fig. \ref{Fig:bg_st} &\cite{Ji_APG,Shen_anaccelerated}  \\
\hline
Accelerated Proximal Gradient-II & APG-II & Fig. \ref{Fig:conv_fn}, \ref{Fig:conv_iter} &\cite{APG}  \\
\hline
\begin{tabular}{l}
{\hspace{-2mm}}Grassmannian Robust Adaptive \\{\hspace{1mm}}
Subspace Tracking Algorithm
\end{tabular}
& GRASTA& Fig. \ref{Fig:bg_st}, \ref{Fig:bg_st_partial}, \ref{Fig:bg_2}, \ref{fig:quant_bg}  & \cite{grasta} \\
\hline
Go Decomposition & GoDec & Fig. \ref{Fig:bg_st}, \ref{Fig:bg_st_partial}, \ref{Fig:bg_2}  & \cite{godec}\\
\hline
Robust PCA Gradient Descent & RPCA GD & Fig. \ref{RMSE}, \ref{Fig:bg_st}, \ref{Fig:bg_st_partial}, \ref{Fig:bg_2},  \ref{fig:quant_bg}, \ref{Fig:missng_data_bg}   & \cite{RPCAgd}\\
\hline
Robust PCA Nonconvex Feasibility & NCF &  Fig. \ref{RMSE}, \ref{Fig:bg_st}, \ref{Fig:bg_st_partial}, \ref{Fig:bg_2}, \ref{Fig:missng_data_bg} &\cite{duttahanzely}\\
\hline
Recursive projected compressed sensing & ReProCS &  Fig. \ref{Fig:bg_st}, \ref{Fig:bg_st_partial}, \ref{Fig:bg_2}, \ref{fig:quant_bg}   & \cite{reprocs}\\
\hline
Best pair RPCA & --- &  Fig. \ref{Fig:missng_data_bg}   & \cite{dutta2019best}\\
\hline
Fixed point Bergman & FPCA & Fig. \ref{Fig:fid_in_out}, \ref{Fig:fidelity_2}  & \cite{ma2011fixed}\\
\hline
Expectation Maximization & EM & Fig. \ref{Fig:fid_in_out}, \ref{Fig:fidelity_2}, \ref{Fig:Fidelity_3} & \cite{srebro}\\
\hline
Bi-iterative regularized SVD &BIRSVD & Fig. \ref{Fig:fid_in_out}, \ref{Fig:fidelity_2}, \ref{Fig:Fidelity_3} & \cite{regularizedlow_das}\\
\hline
Bilinear factorization & BLF & Fig. \ref{Fig:fid_in_out}, \ref{Fig:fidelity_2} & \cite{bl_factor}\\
\hline
Damped Newton &---& Table \ref{table:sfm} & \cite{Buchanan}\\
\hline
\end{tabular}
\end{center}
\caption{{Algorithms compared in this paper.}}\label{algo}
\end{table}

\subsection{Convergence behavior on synthetic data}\label{sec:appendix_conv}

In this section, we demonstrate the convergence of our algorithm(s). For this purpose, we generate the low-rank matrix, $L$, as a product of two independent full-rank matrices of size $m\times r$ with $r<m$ such that the elements are independent and identically distributed (i.i.d.) and sampled from a normal distribution---$\cN(0,1)$. In our setup, $m=100$, and $r=5.$ We generate $E$ as a Gaussian noise matrix whose elements are i.i.d. $\cN(0,1)$ random variables and constructed $F$ as: $F=L+E$. We fixed $\min_{i,j}W_{ij}=1$, and choose $\max_{i,j}W_{ij}$ from a set $\Lambda=\{ 10,50,100,500,1000,5000,10^4,5\times10^4,10^5,5\times10^6,10^7\}$. At each instance, we form an $m\times m$ weight matrix, $W$ by using {\tt MATLAB} function {\tt randi} that generates pseudorandom integers from a uniform discrete distribution, $[\min_{i,j}W_{ij},\Lambda(p)]$, where $\Lambda(p)$ is chosen from $\Lambda$ without replacement.

We compare our \program-$\epsilon$ (Algorithm \ref{alg:alg}), its inexact counterpart \program-$1$, proximal gradient~(PG) algorithm and its accelerated version---accelerated proximal gradient (APG) for these experiments. For different condition number~($\kappa_W$) of the weight matrix $W$, we plotted the difference between functional values evaluated at consecutive iterates, that is, $\Phi(X_k)-\Phi(X_{k-1})$ versus iterations in Figure~\ref{Fig:conv_fn}. In Figure~\ref{Fig:conv_iter}, we plot the difference between consecutive iterates, $\|X_{k+1}-X_k\|$, versus iterations. Note that by construction, $\kappa_W$ ranges between 1238.021 to 21043.1574. The convergence plots justify our claims, that, although problem \eqref{prblm:wlr_nuclear} belongs to the class of problems \eqref{prblm:wlr_loss}, the general algorithms used in \cite{Ji_APG,Shen_anaccelerated,Toh2009AnAP,Liu:2014:NNR} fail to provide good convergence results when $\kappa_W$ is large\footnote{The condition number largely varies due to the random arrangement of the ``weight" elements in $W$; some $W$s with a smaller $\lambda_{\rm max}$ have higher condition numbers than those with a larger $\lambda_{\rm max}.$}; that, our approaches are faster compare to those general approaches; and, that the performance of both exact and inexact \program~are the same. Moreover, in all cases, \program-$\epsilon$ and \program-$1$ can  recover the rank 5 low-rank matrix, but PG and APG mostly recover a full-rank matrix. We hypothesize this is due to the sensitivity of PG and APG algorithms on the balancing parameter, $\tau$ which is required to perform the proximal mapping (in this case, the singular value thresholding \cite{caicandesshen,shrinkage}); see also Figure \ref{fig:comp_rc_1}.

\begin{figure}[!t]
\centering
    \includegraphics[width=0.245\linewidth]{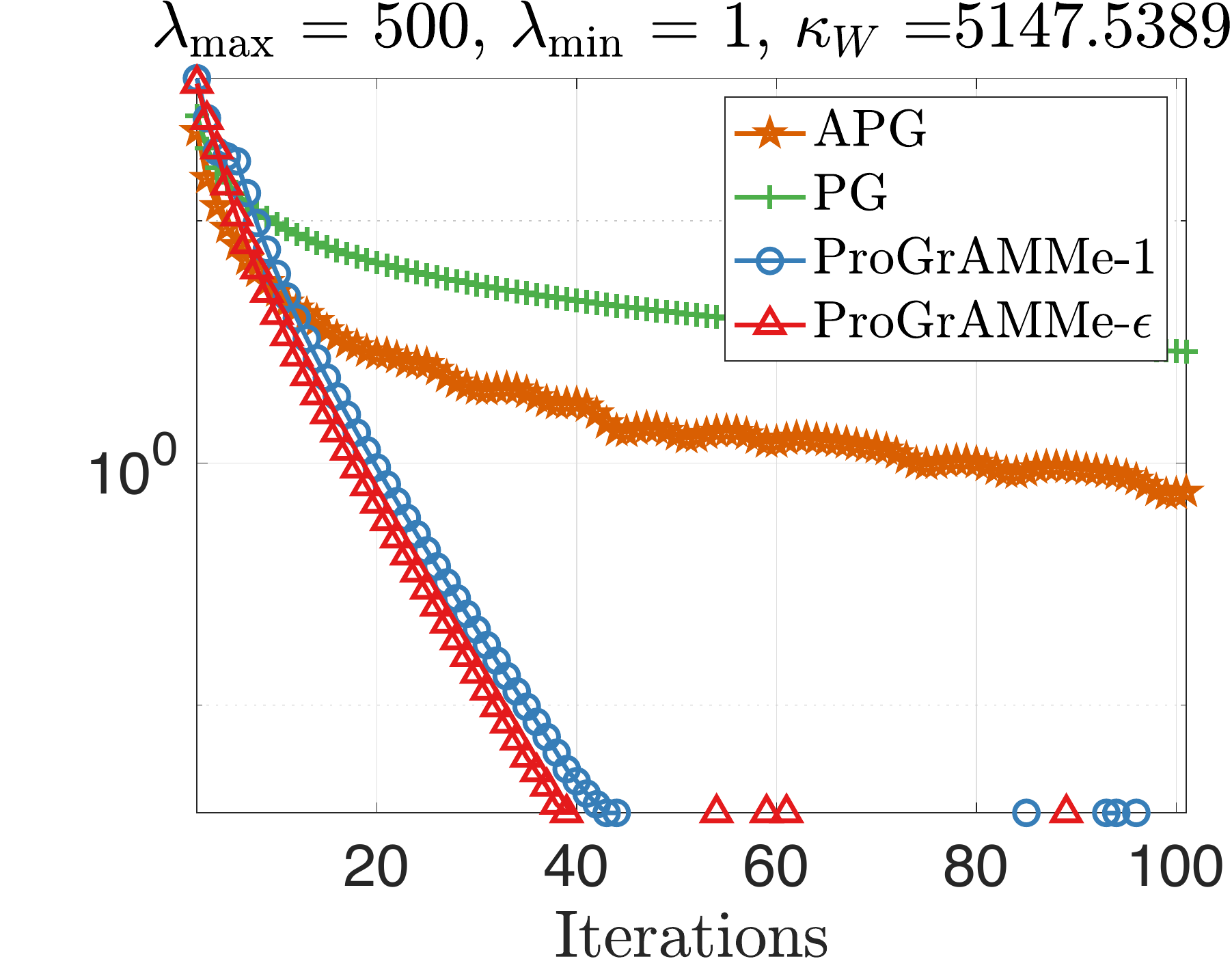} \hspace{-4pt}
    \includegraphics[width=0.245\linewidth]{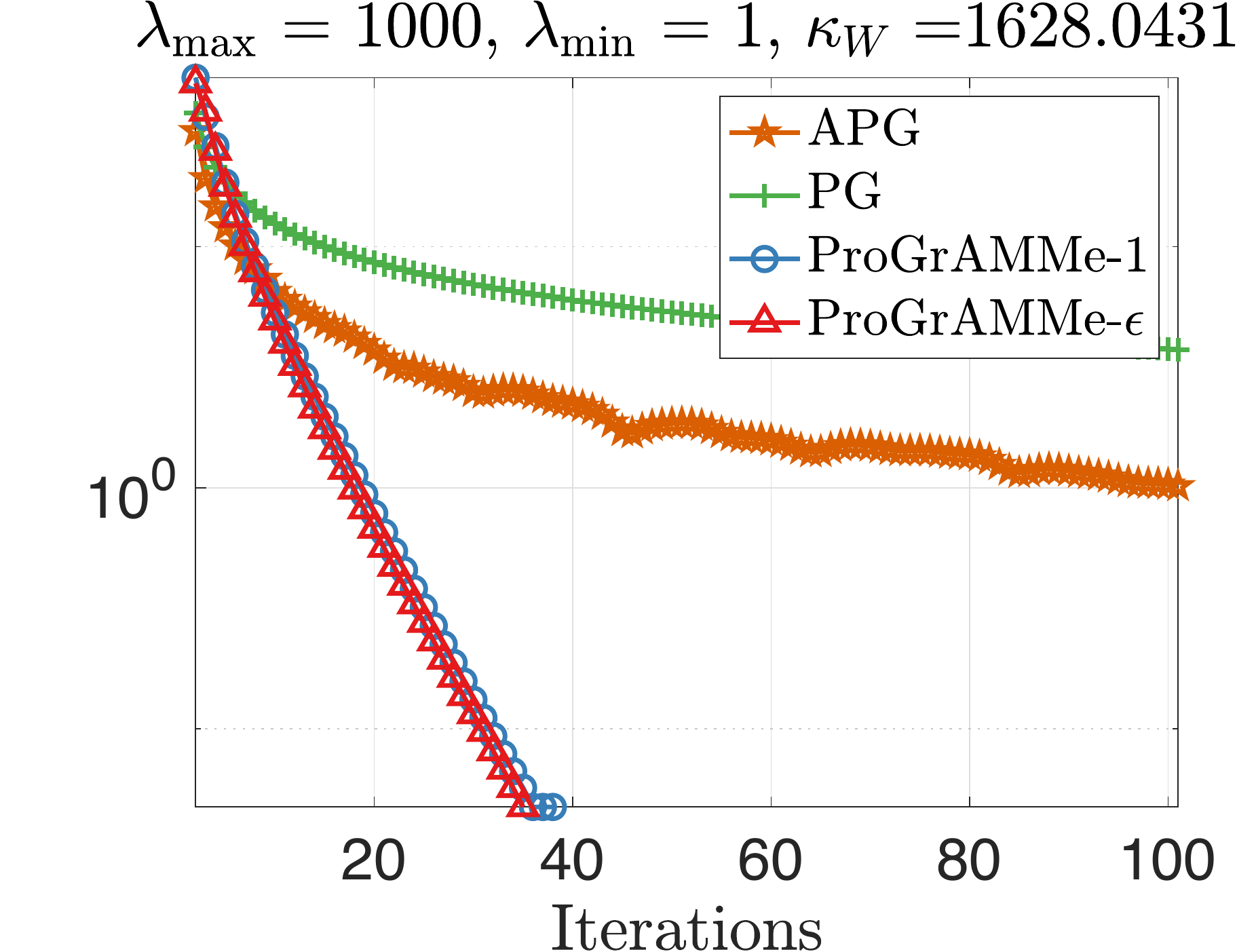} \hspace{-4pt}
    \includegraphics[width=0.245\linewidth]{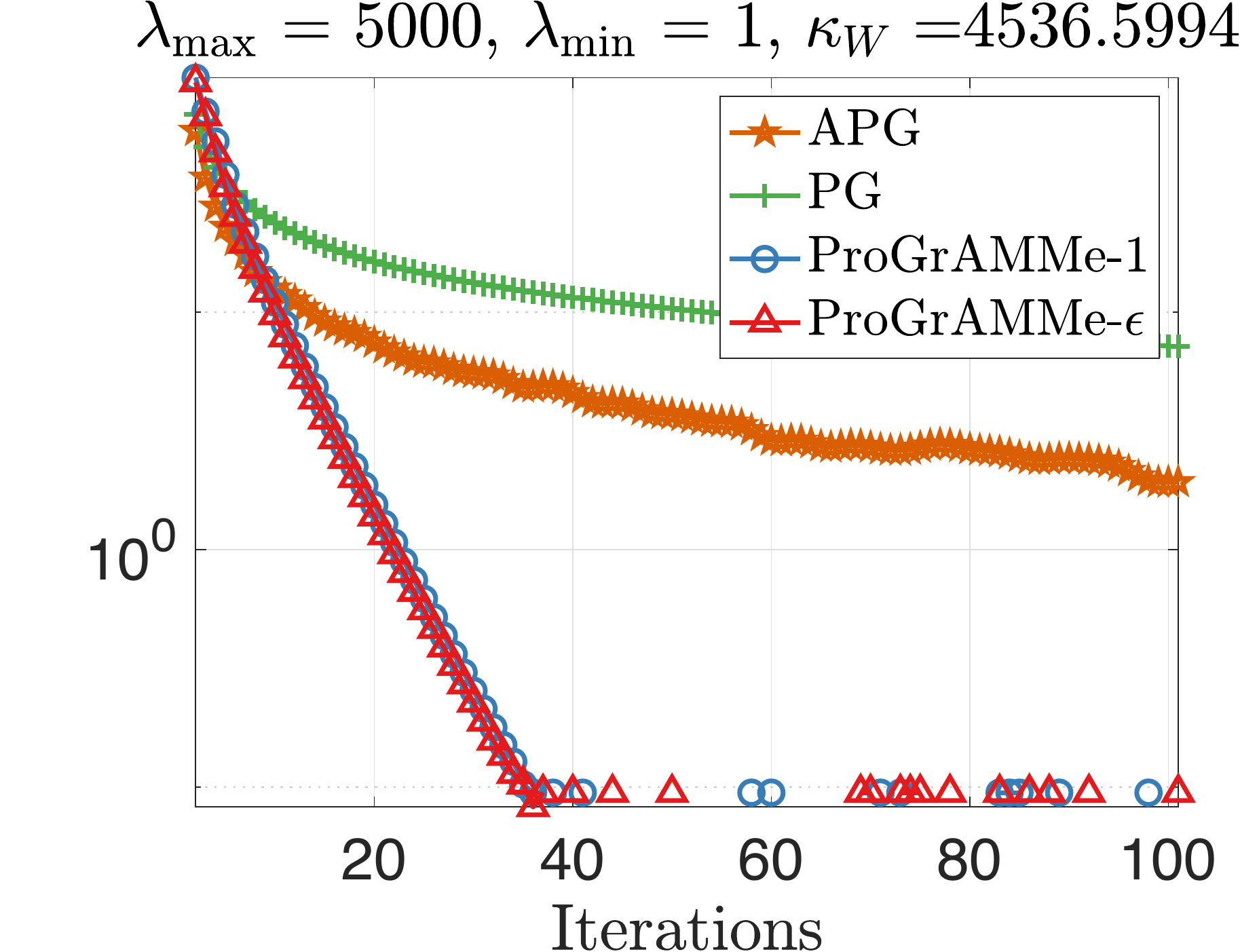} \hspace{-4pt}
    \includegraphics[width=0.245\linewidth]{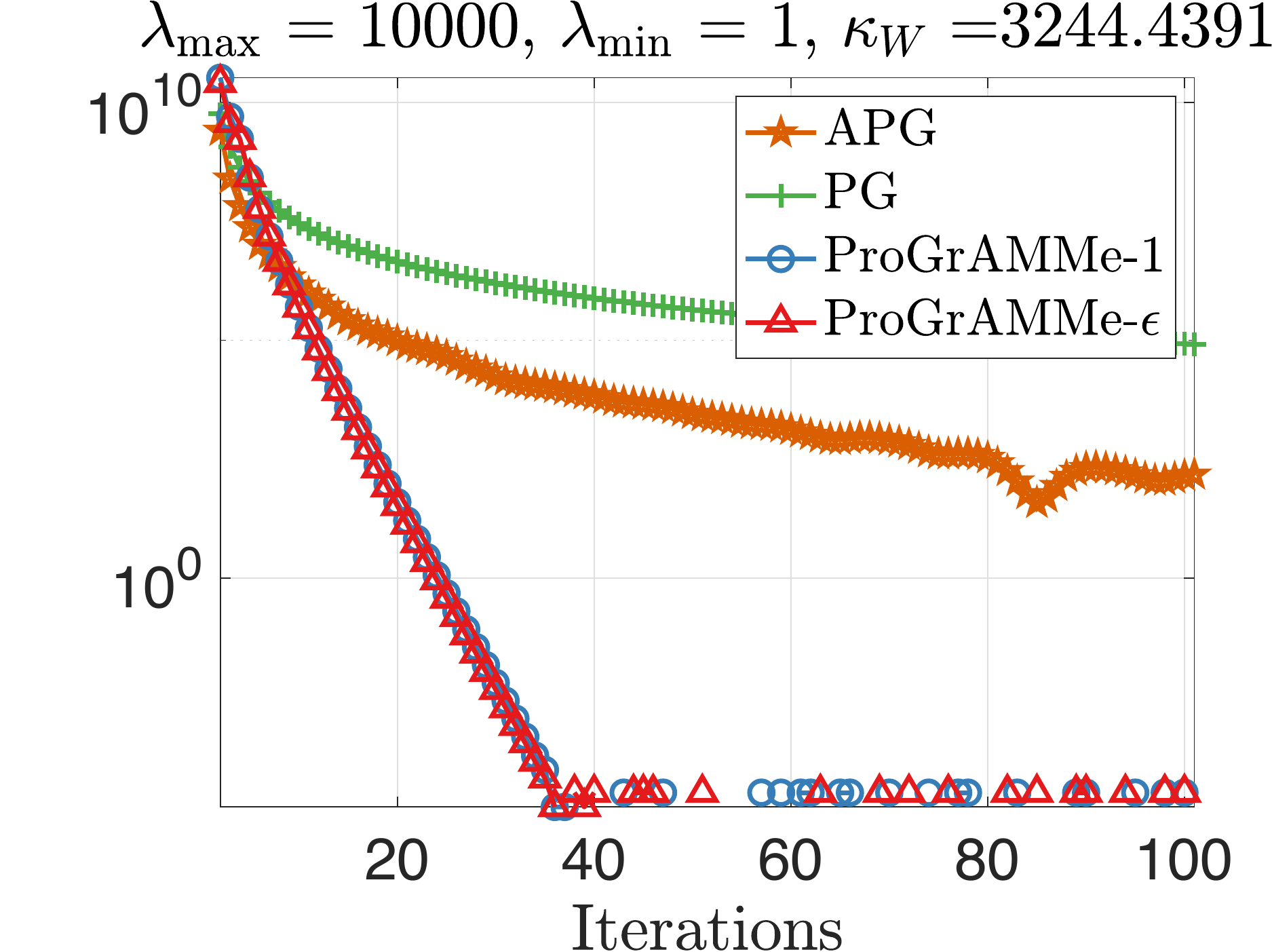} \\
%%%%%%%%%%
    \includegraphics[width=0.245\linewidth]{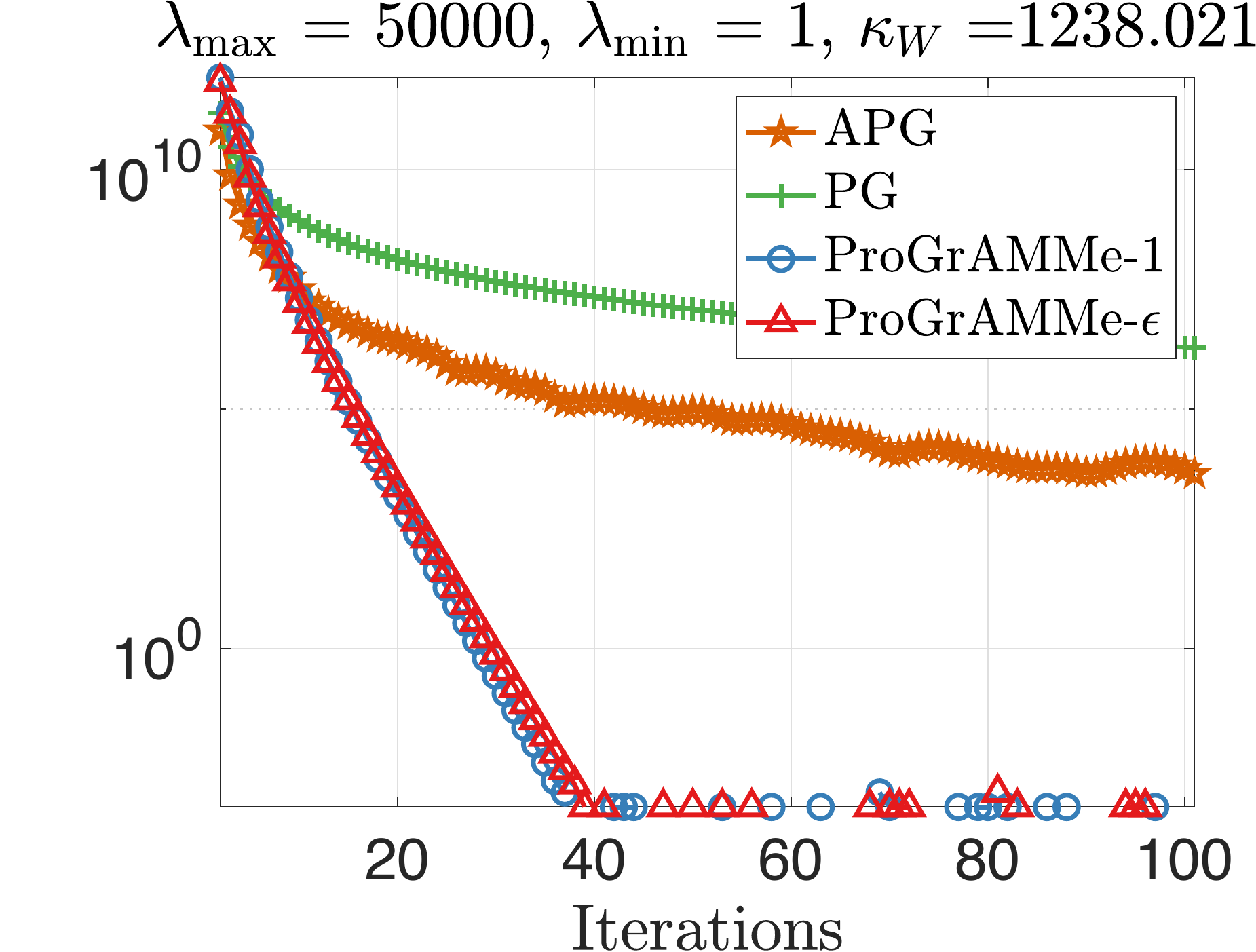} \hspace{-4pt}
    \includegraphics[width=0.245\linewidth]{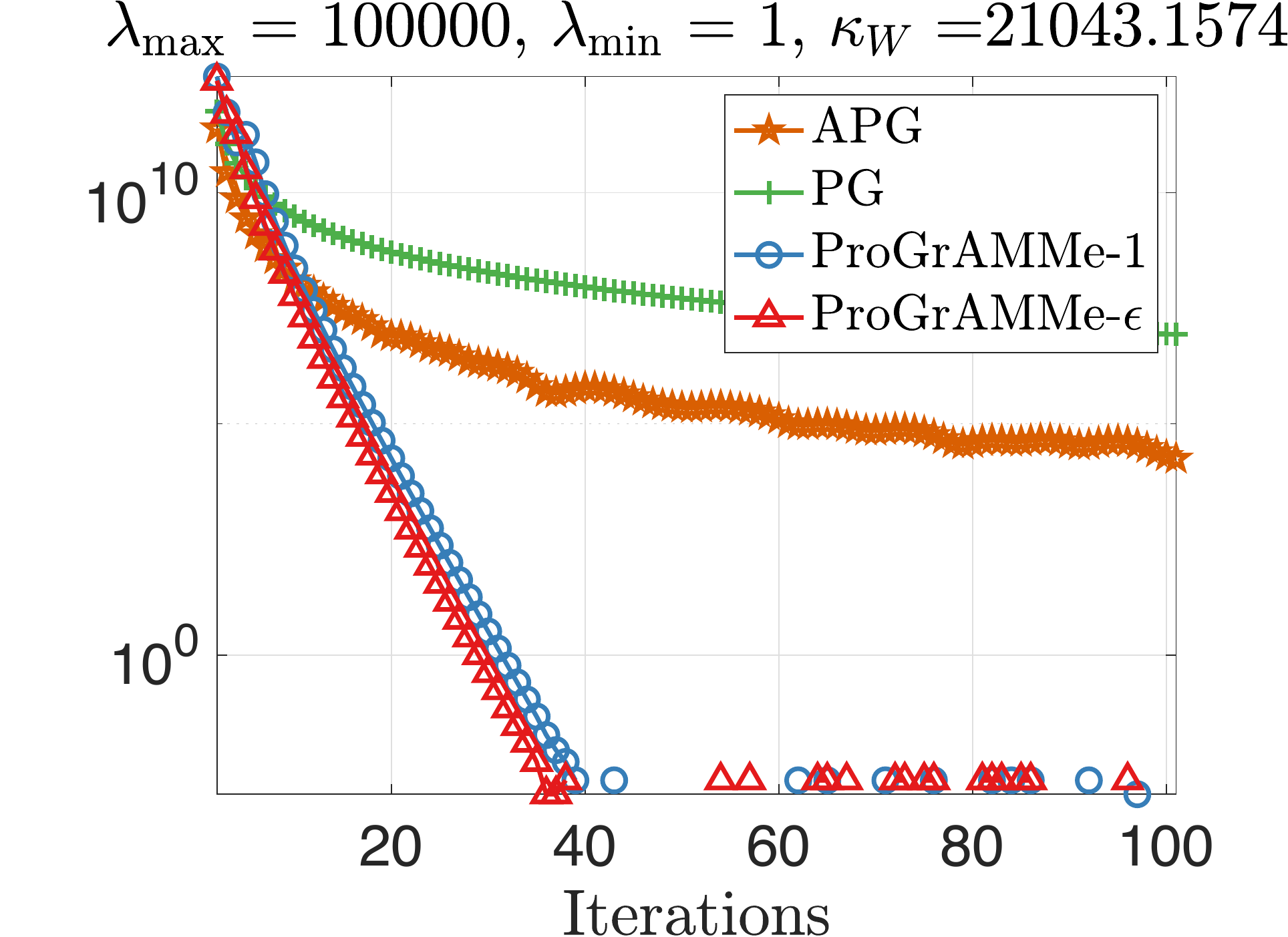} \hspace{-4pt}
    \includegraphics[width=0.245\linewidth]{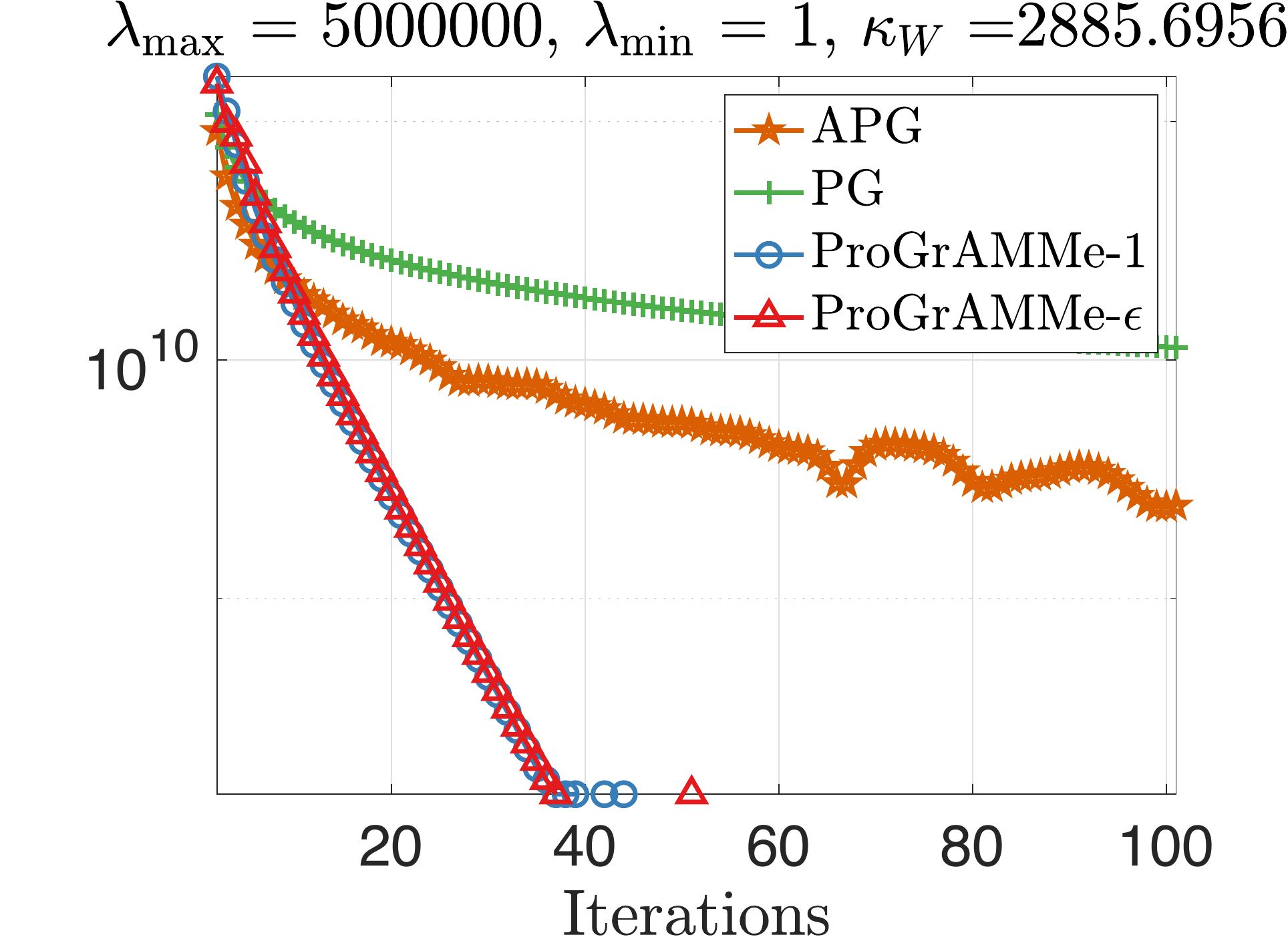} \hspace{-4pt}
    \includegraphics[width=0.245\linewidth]{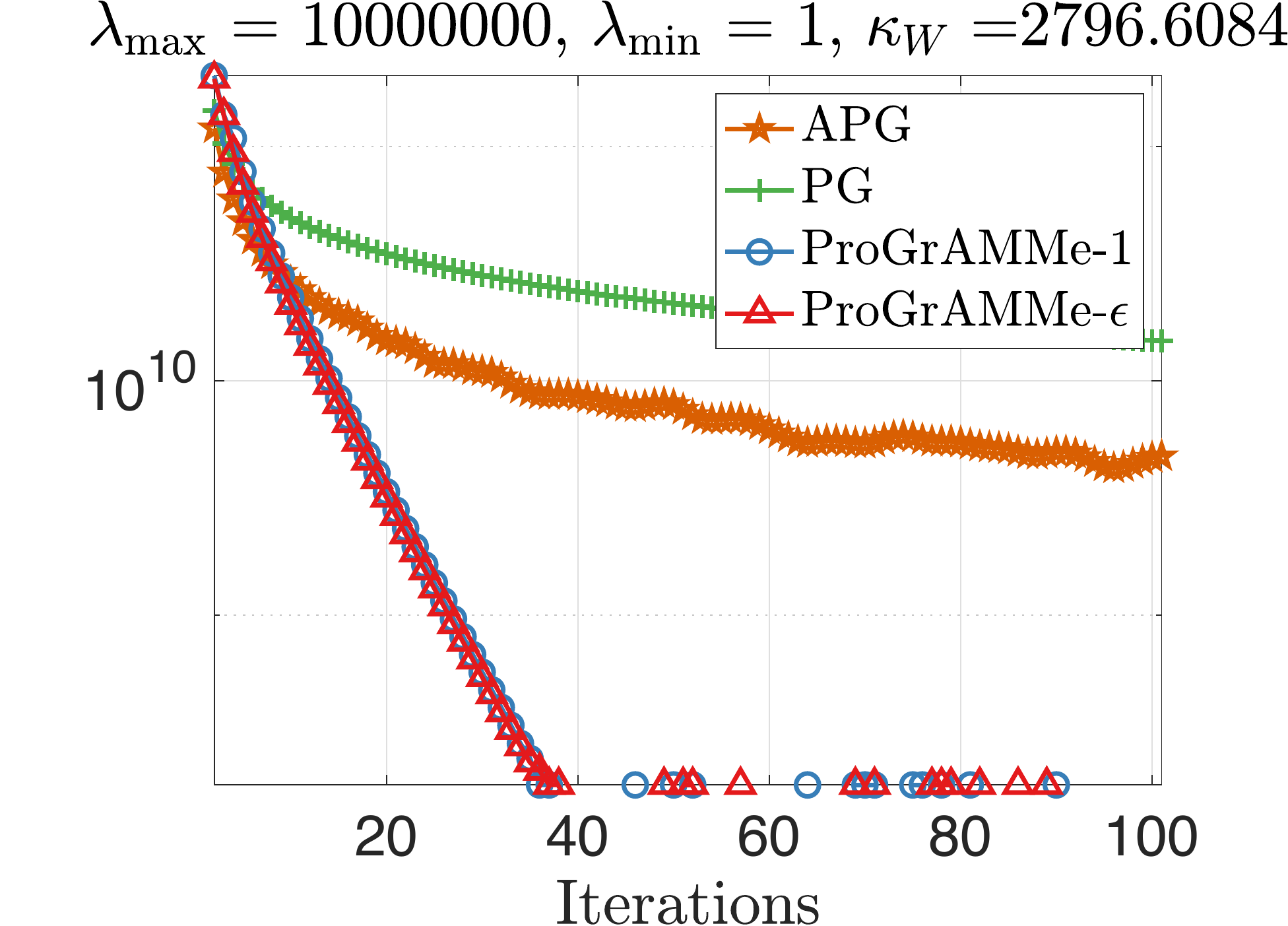} \\
%%%%%%%%%%
    \includegraphics[width=0.245\linewidth]{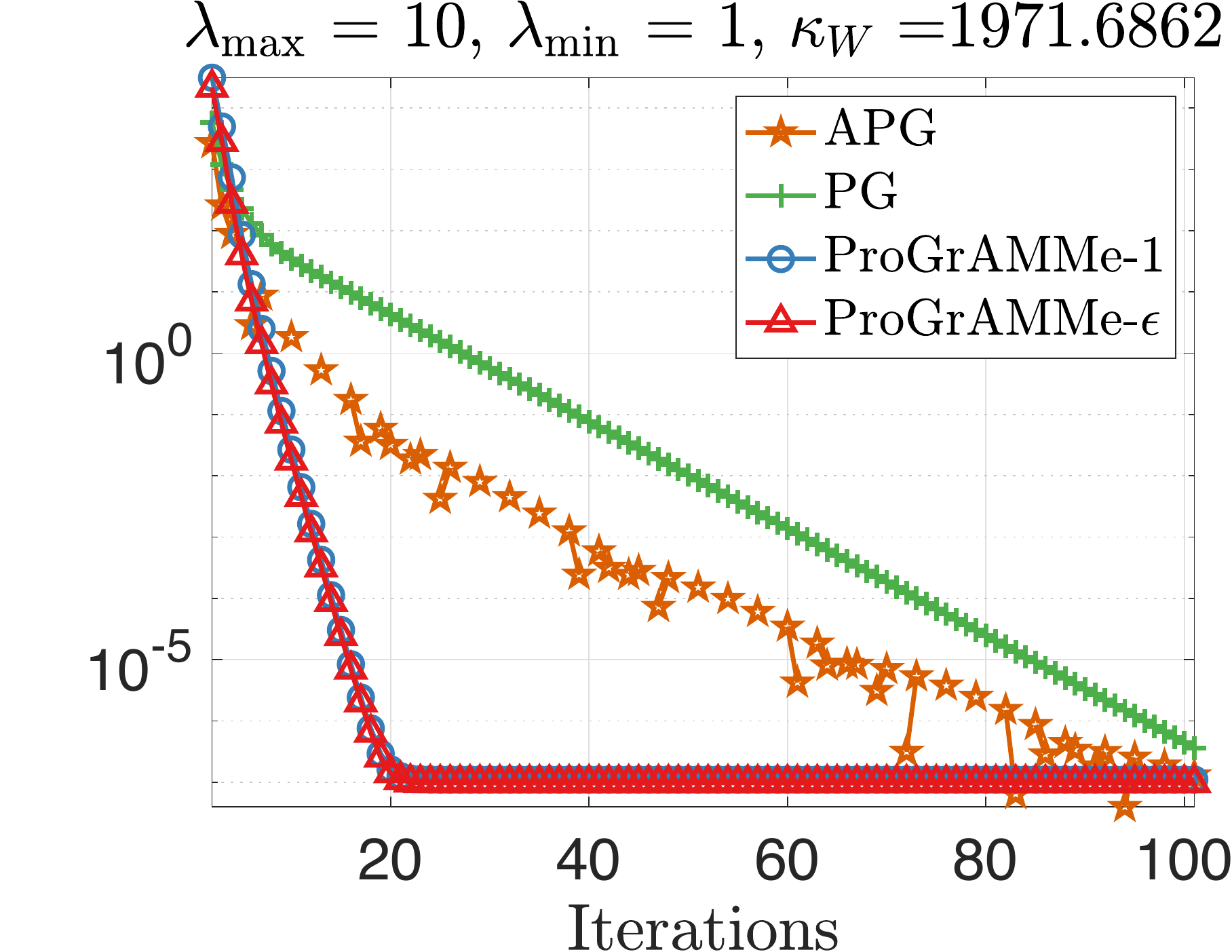} \hspace{-4pt}
    \includegraphics[width=0.245\linewidth]{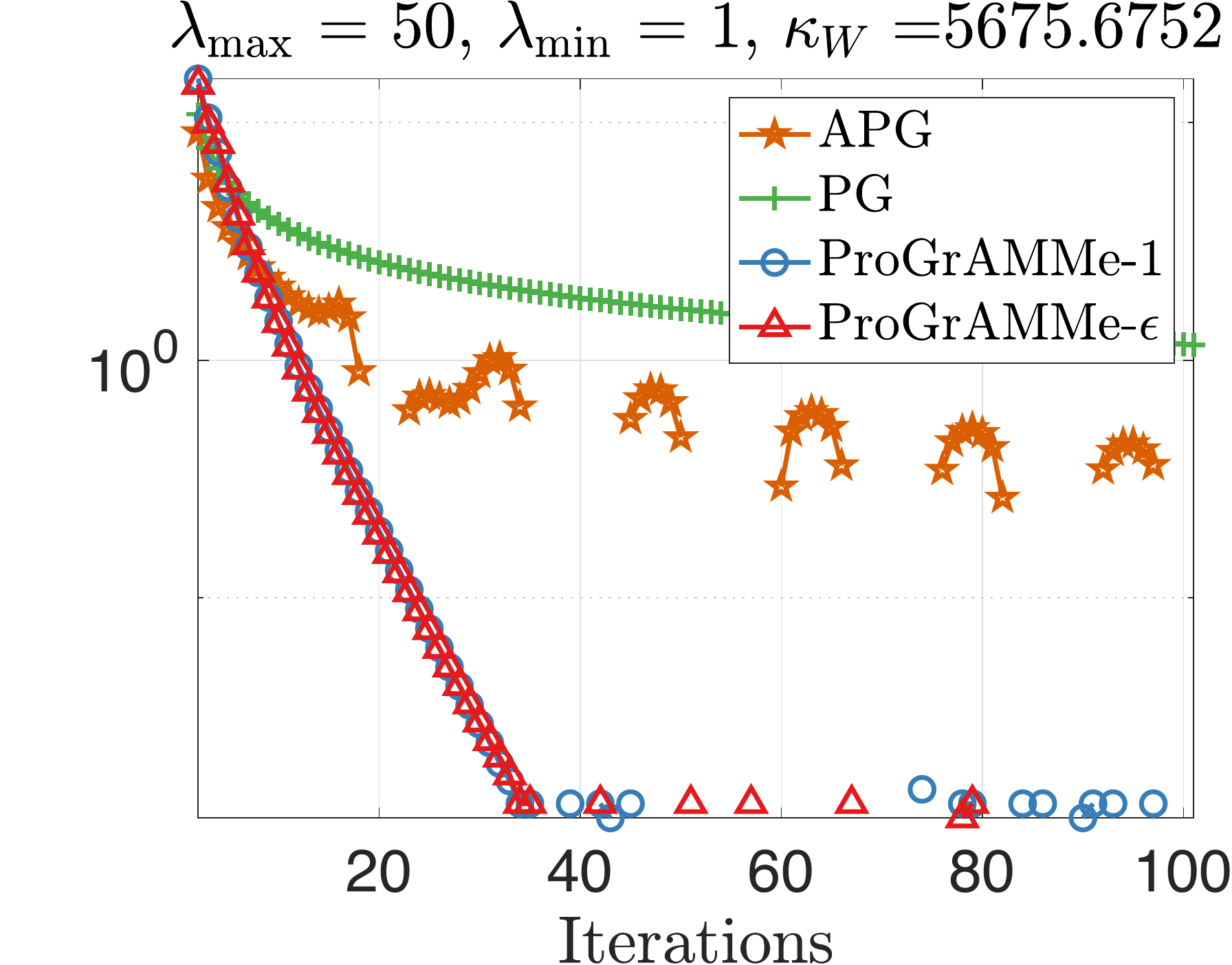} \hspace{-4pt}
    \includegraphics[width=0.245\linewidth]{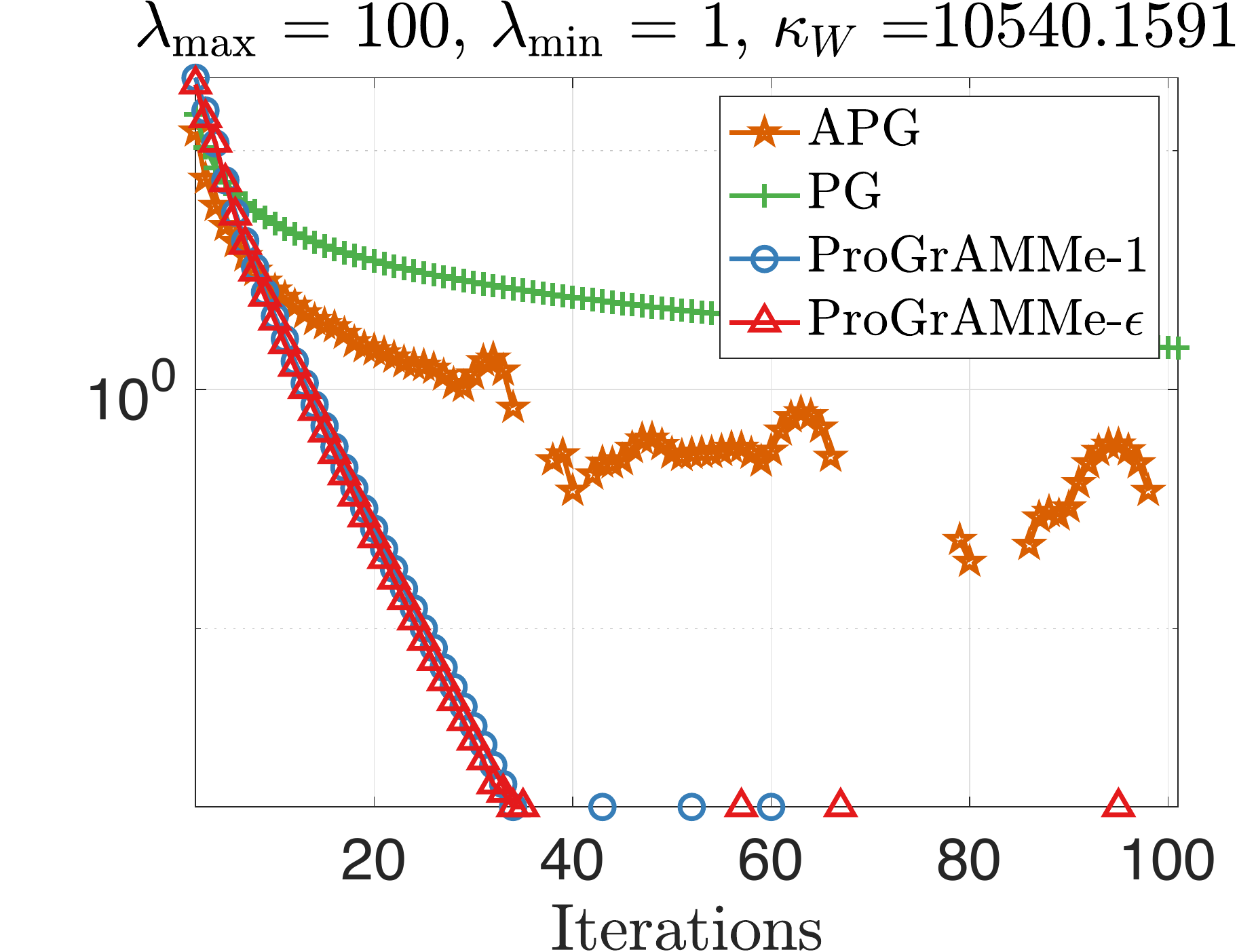} \hspace{-4pt}
    %\includegraphics[width=0.245\linewidth]{conv/fn_val-1.pdf} \\
%%%%%%%%%%
\caption{{Convergence in terms of functional values evaluated at consecutive iterates, that is, $\Phi(X_k)-\Phi(X_{k-1})$ vs. iterations of proximal gradient \eqref{prblm:wlr_nuclear}~(Direct PG), accelerated proximal gradient-II~(APG-II), \program-$1$, and \program-$\epsilon$ applied to original problem. Note that, $\Phi(X)=f(X)+g(X)$; see details in Section \ref{subsection:alg}. We set $\tau = 10^{-2}/{\lambda^2_{\max}}$.}}\label{Fig:conv_fn}
\end{figure}

% Finally, to validate our claim in Remark \ref{remark:cond number} about the convergence rate and changed condition number of the problem, we plot the true and approximated condition number $\kappa_U$ and $\kappa_V$ for both ckasses exact and inexact \program in Figure \ref{Fig:cond_number_true} and \ref{Fig:cond_number_app}. Both figures show that---\myNum{i} For all problems, in spite of a higher condition number, $\kappa_W$ of the weight matrix, $W$ and of the problem $\frac{L}{\mu}$~(by Proposition \ref{prop:convexity}), the values of $\kappa_U$ and $\kappa_V$ for exact and inexact \program are moderate and are in the same order; that, \myNum{ii} the true and approximated values of $\kappa_U$ and $\kappa_V$ are close in all cases indicating the approximation in Remark \ref{remark:cond number}-\myNum{ii} is valid. %These empirical evidences solidify our claim. 

\begin{figure}[!t]
\centering
    \includegraphics[width=0.245\linewidth]{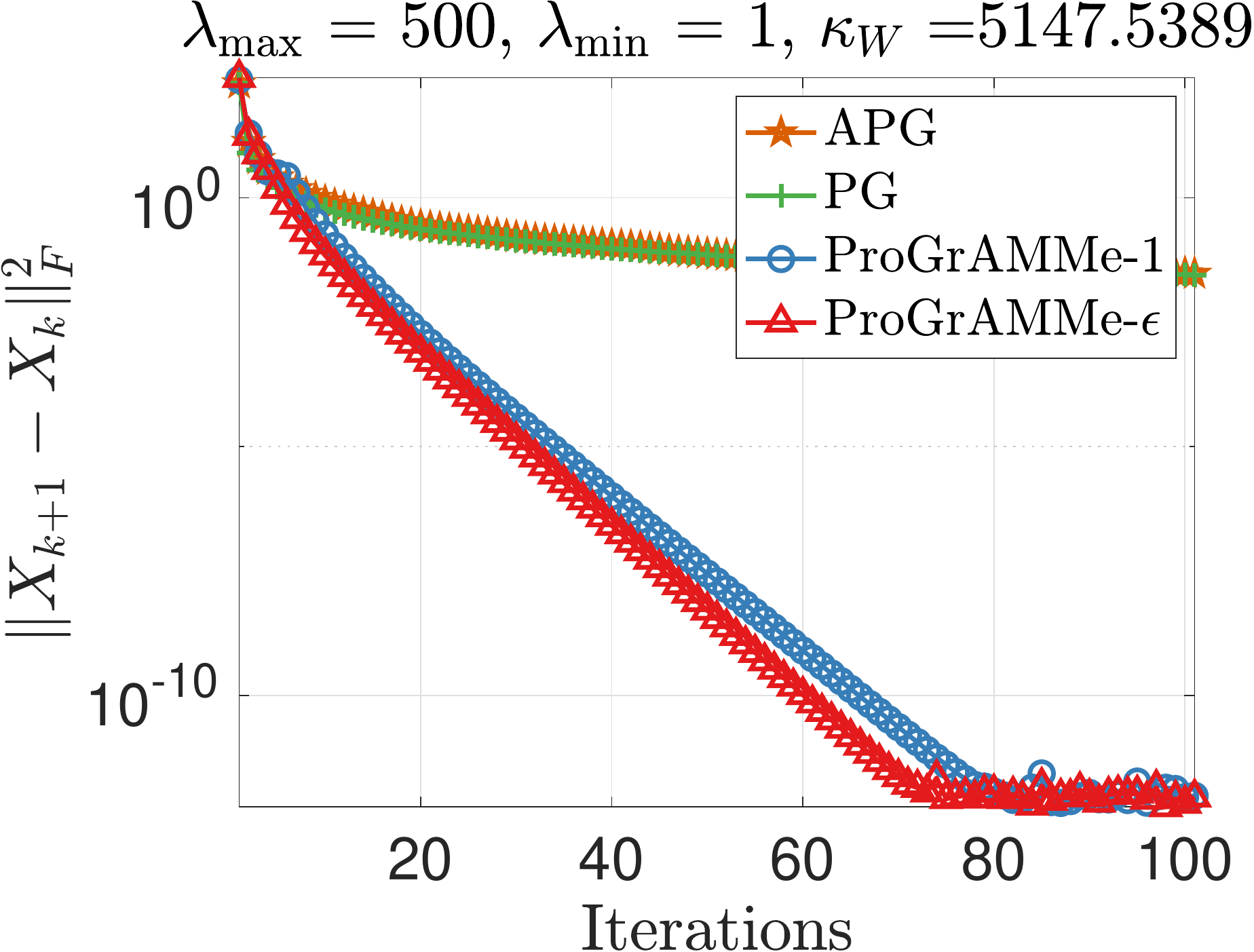} \hspace{-4pt}
    \includegraphics[width=0.245\linewidth]{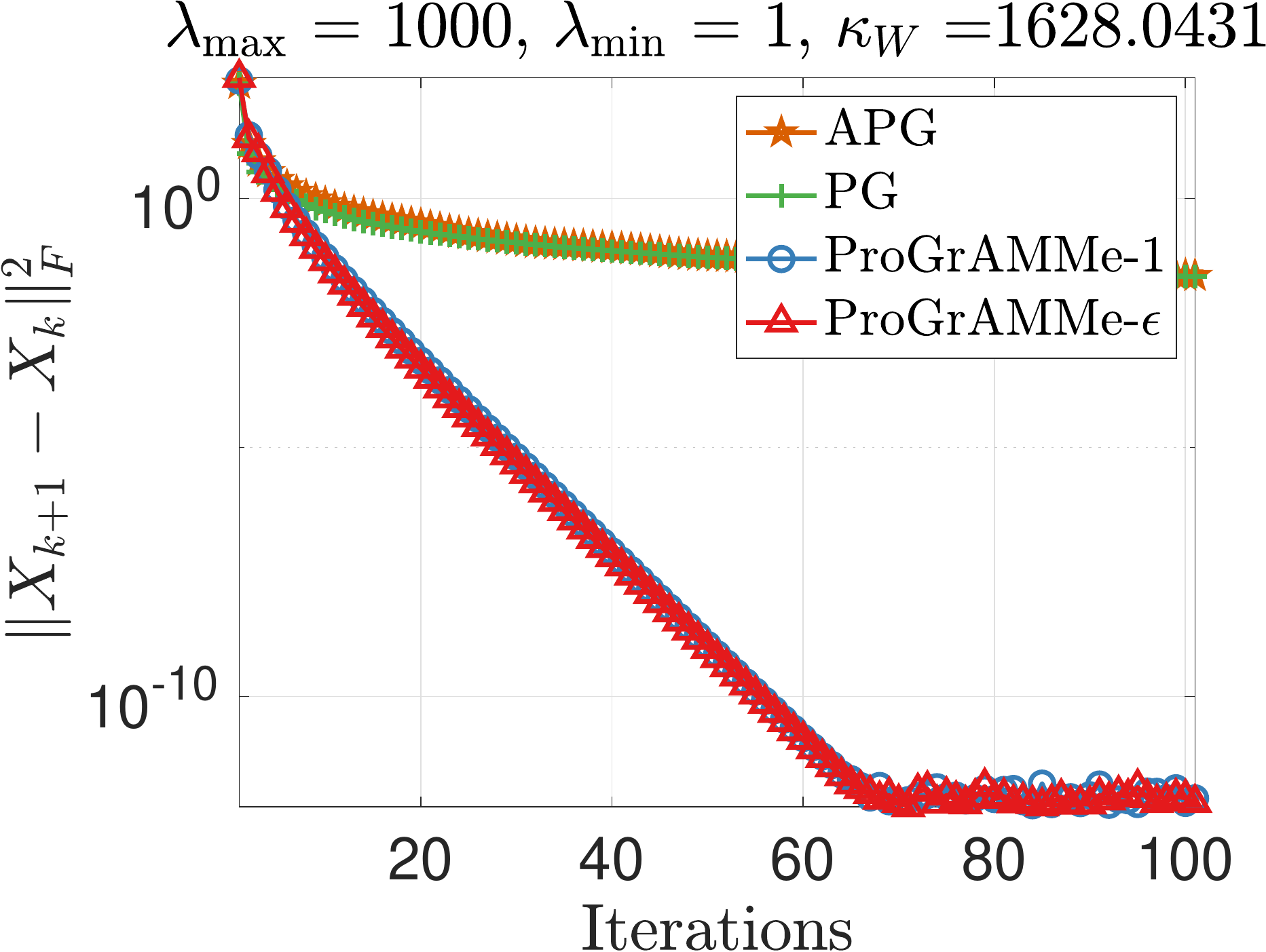} \hspace{-4pt}
    \includegraphics[width=0.245\linewidth]{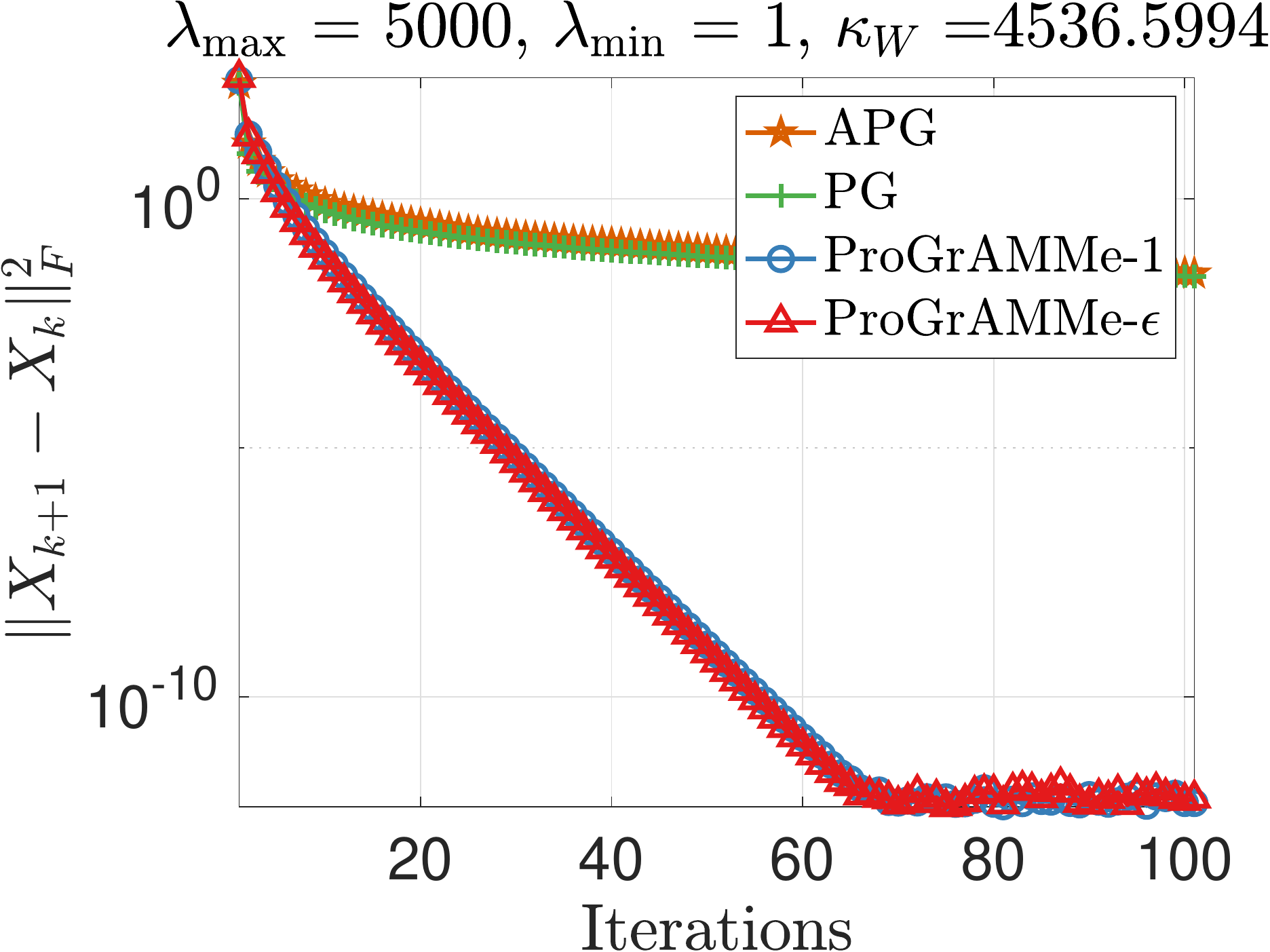} \hspace{-4pt}
    \includegraphics[width=0.245\linewidth]{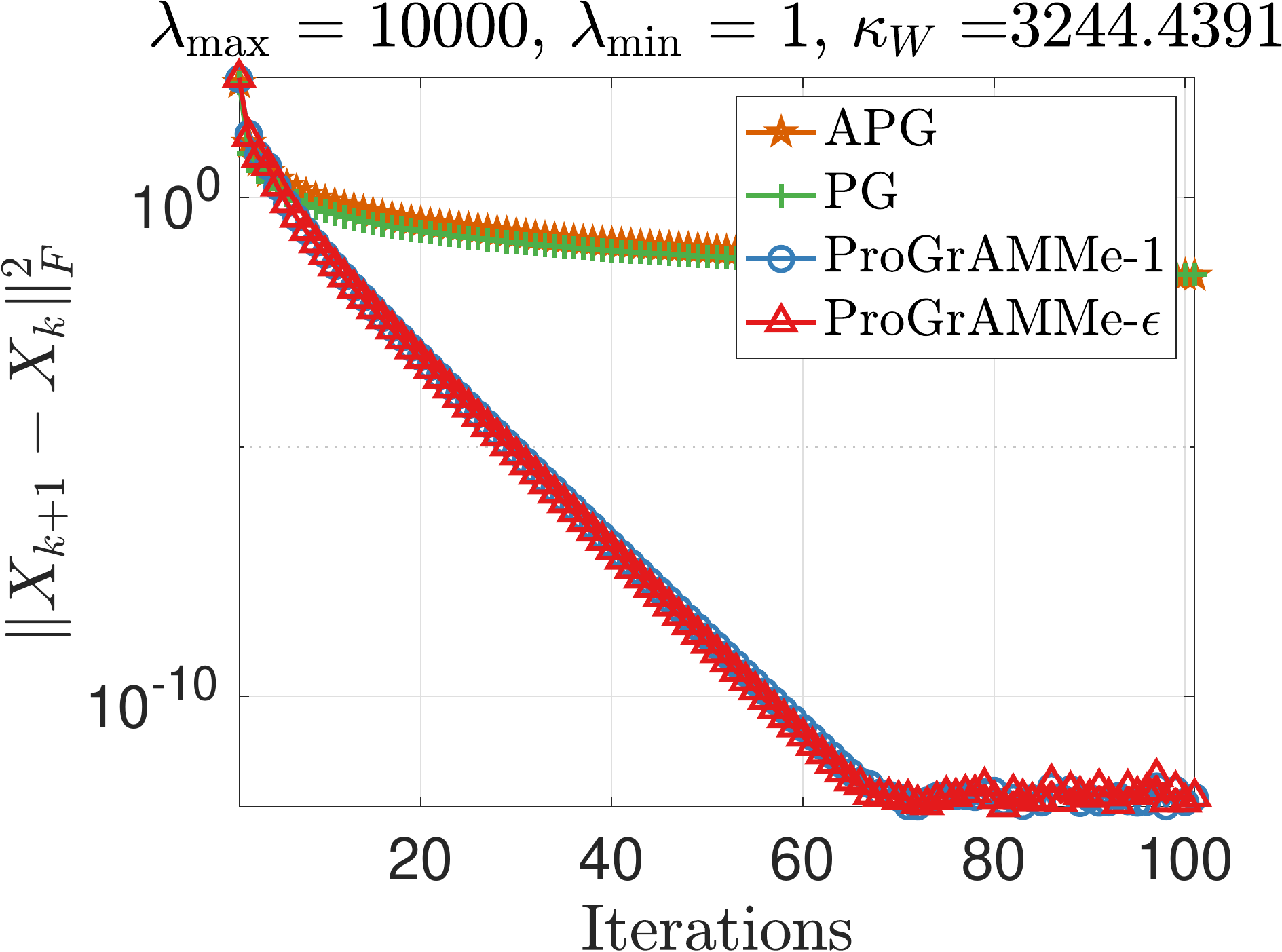} \\
%%%%%%%%%%
    \includegraphics[width=0.245\linewidth]{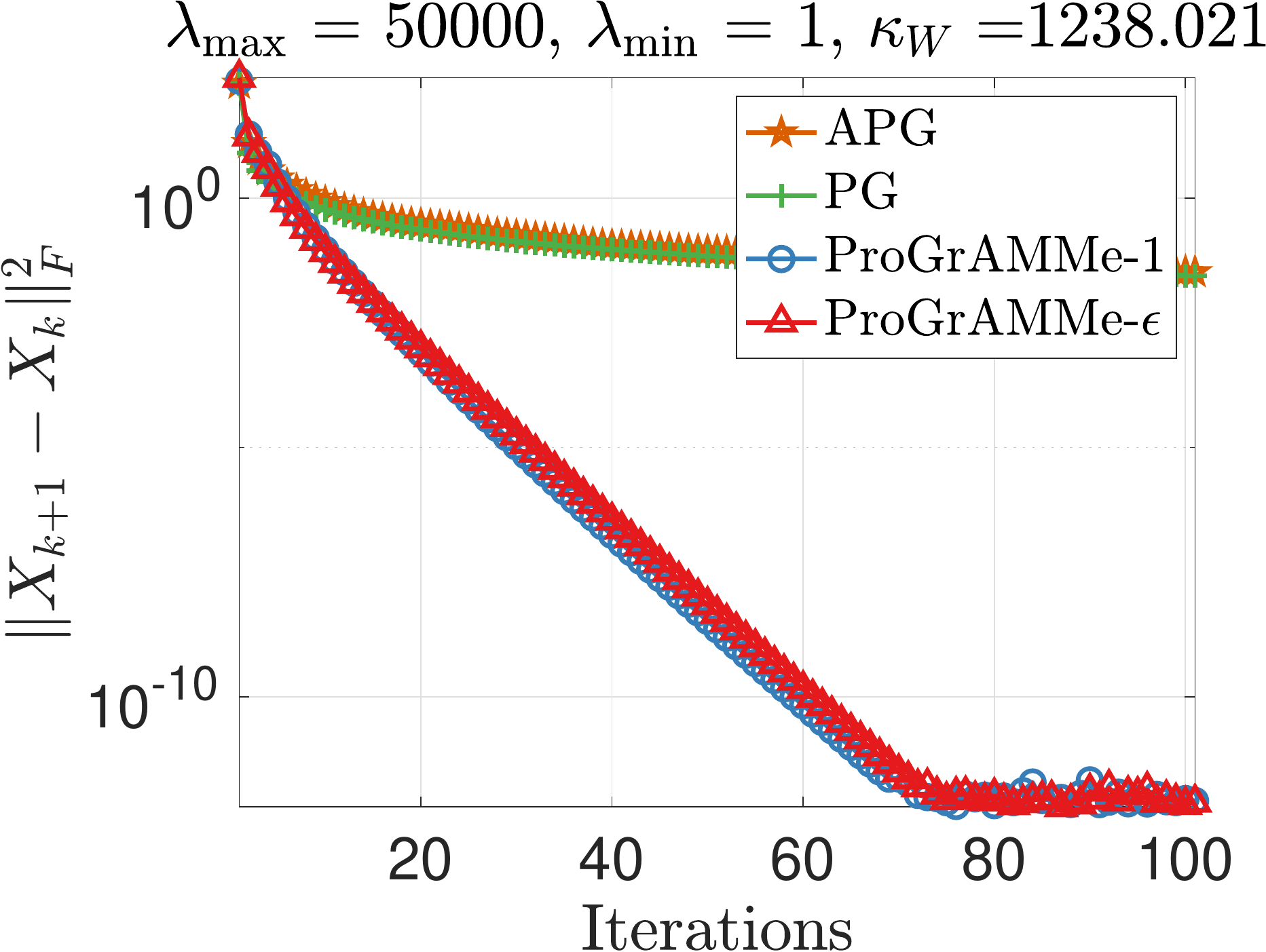} \hspace{-4pt}
    \includegraphics[width=0.245\linewidth]{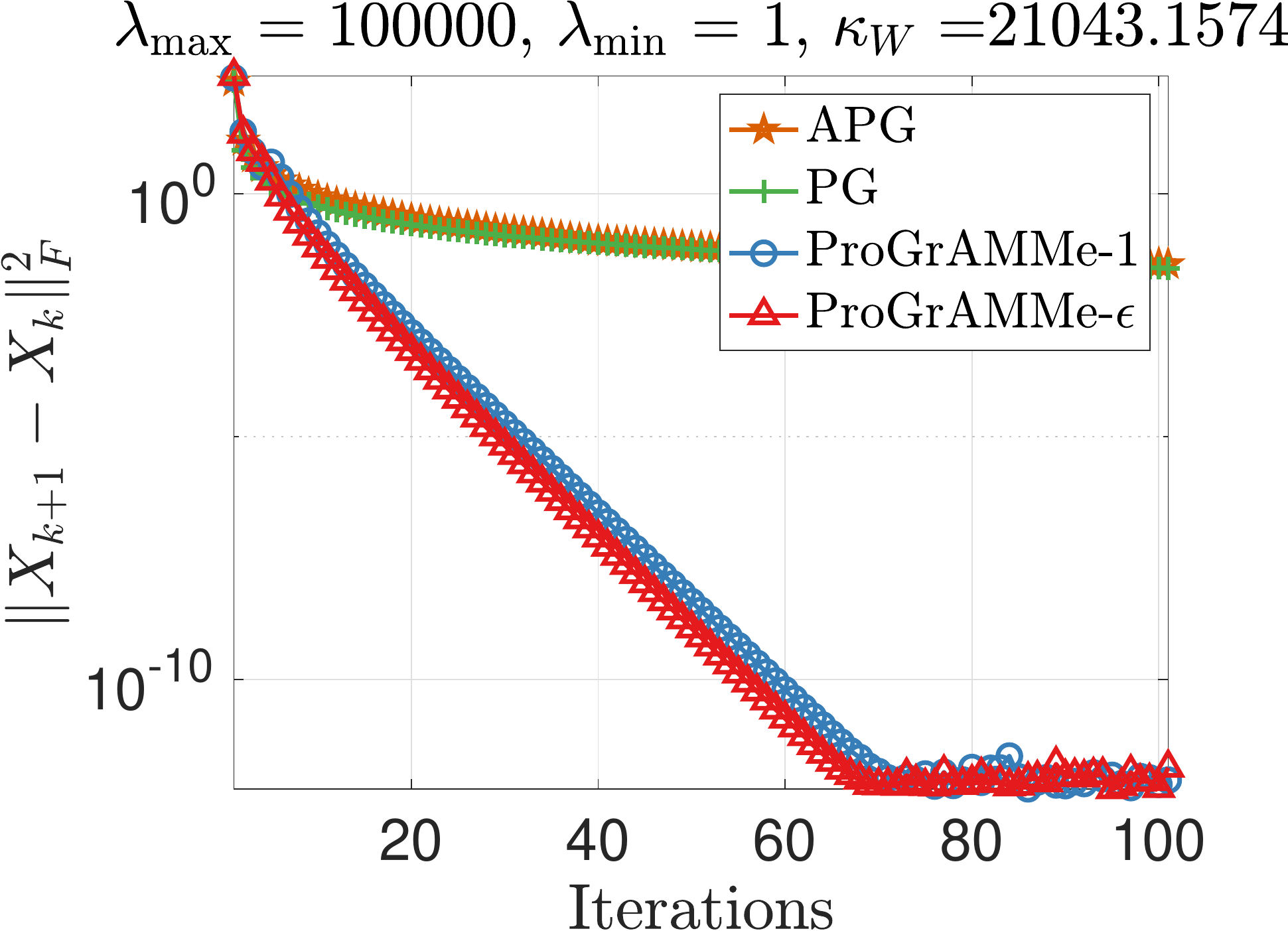} \hspace{-4pt}
    \includegraphics[width=0.245\linewidth]{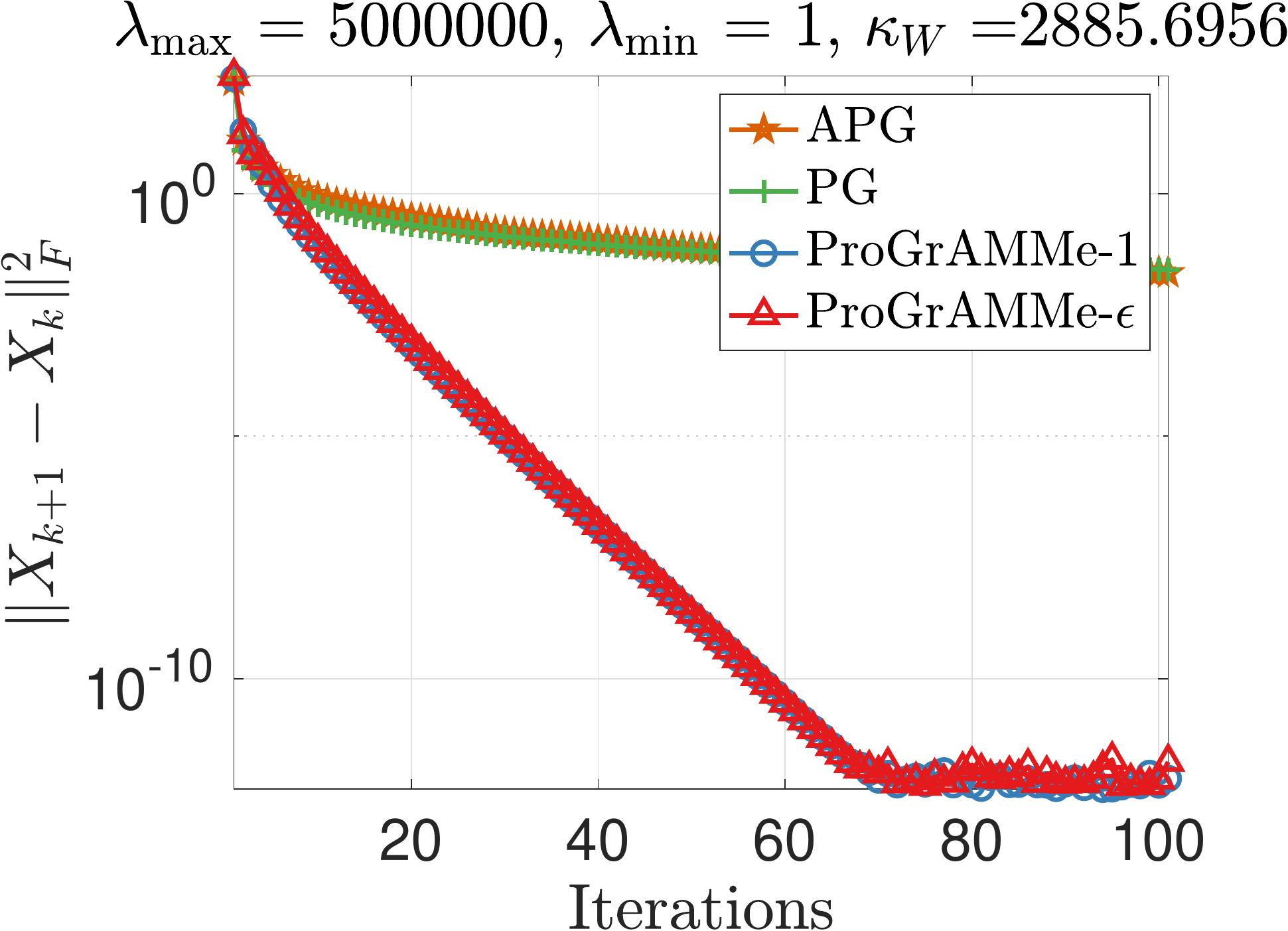} \hspace{-4pt}
    \includegraphics[width=0.245\linewidth]{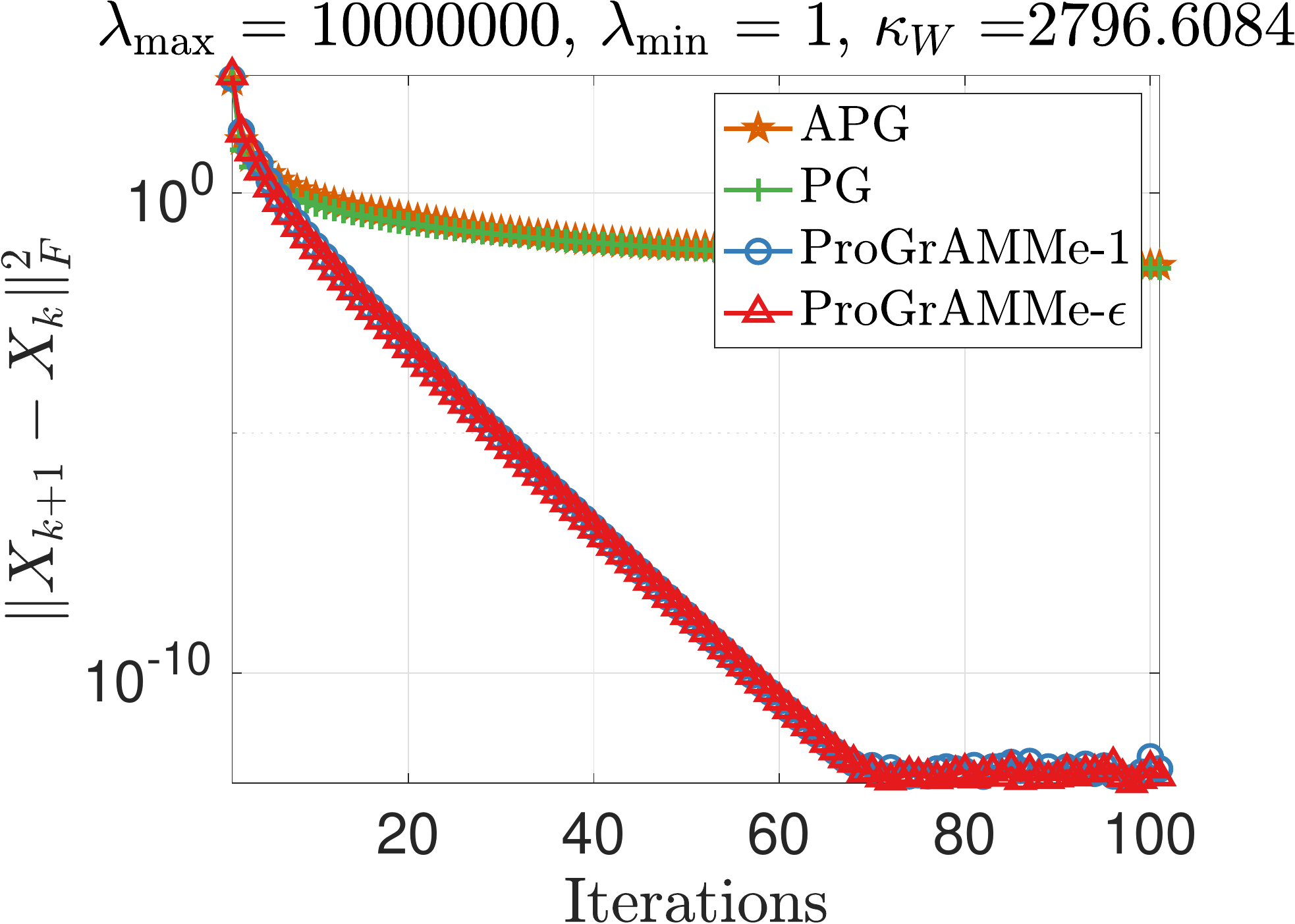} \\
%%%%%%%%%%
    \includegraphics[width=0.245\linewidth]{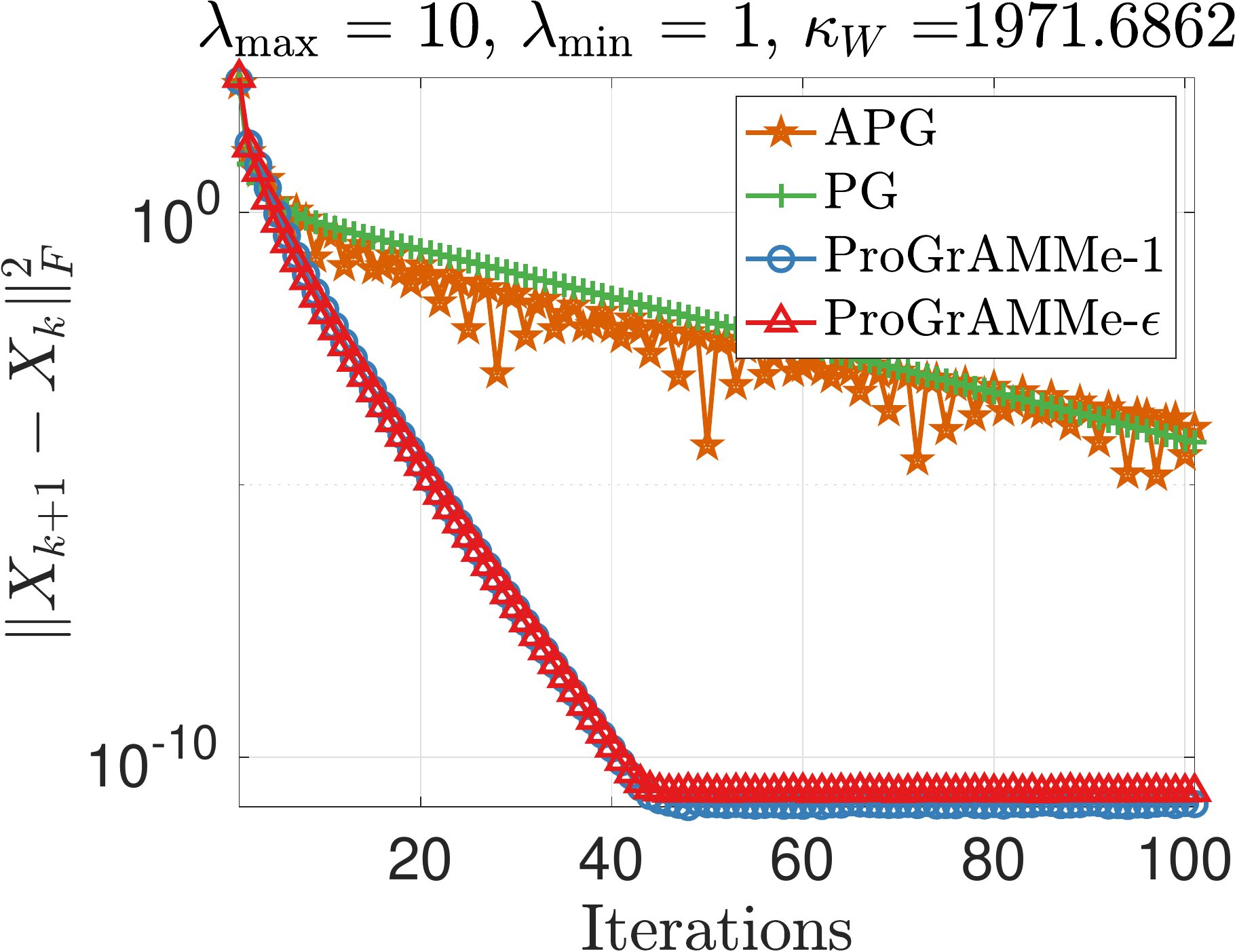} \hspace{-4pt}
    \includegraphics[width=0.245\linewidth]{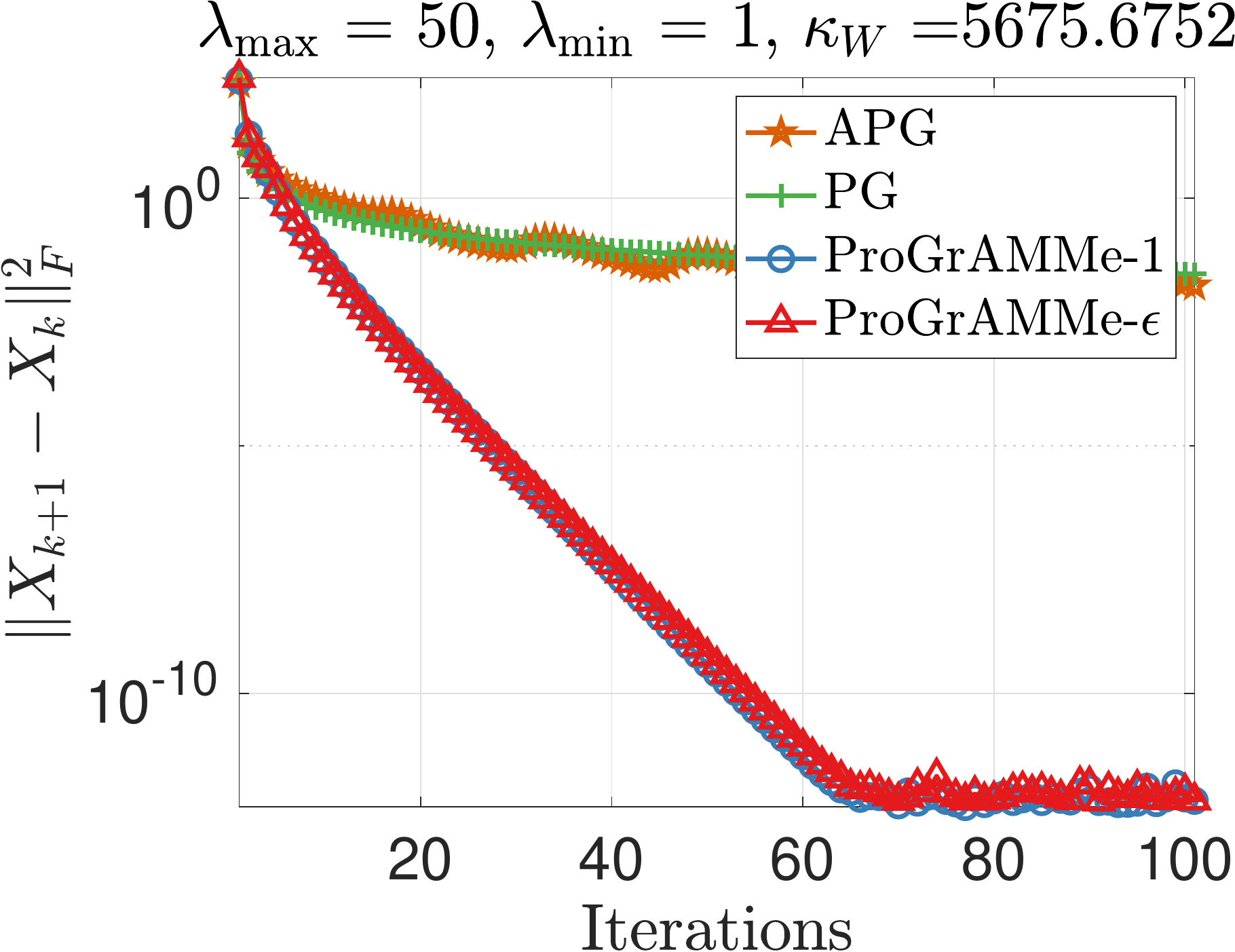} \hspace{-4pt}
    \includegraphics[width=0.245\linewidth]{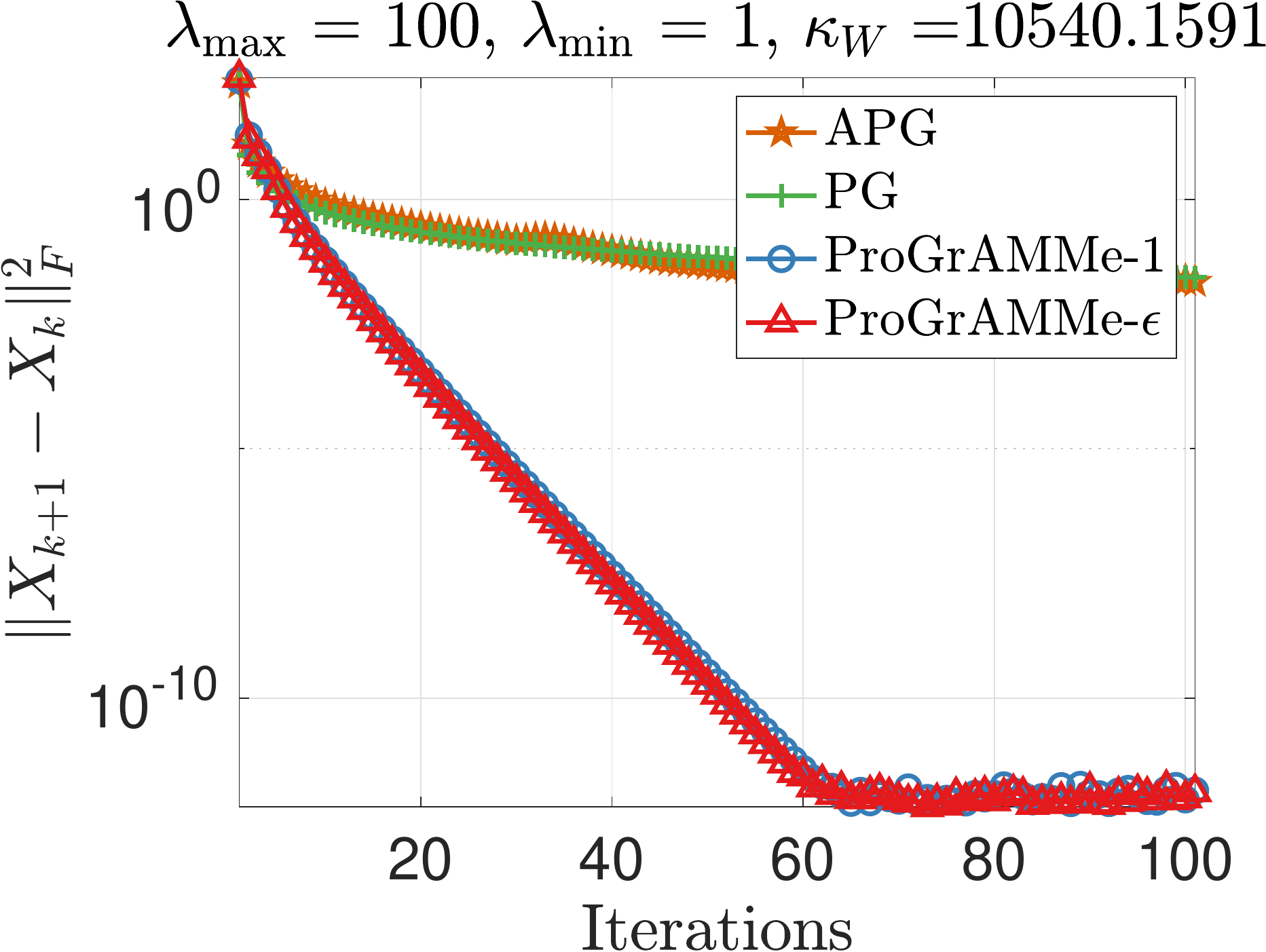} \hspace{-4pt}
    %\includegraphics[width=0.245\linewidth]{conv/iter-1-eps-converted-to.pdf} \\[-1mm]
%%%%%%%%%%
\caption{{Convergence in terms of difference between consecutive iterates ($\|X_{k+1}-X_k\|$) vs. of proximal gradient \eqref{prblm:wlr_nuclear}~(Direct PG), accelerated proximal gradient-II~(APG-II), \program-$\epsilon$ (Algorithm \ref{alg:alg}), and \program-$1$ applied to original problem. We set $\tau = 10^{-2}/{\lambda^2_{\max}}$.}}\label{Fig:conv_iter}
\end{figure}

\subsection{Matrix completion with missing data---Effect of more global iterations}\label{sec:appendix_weather}
In this part, we conduct extensive tests of the power-grid missing data problem. For first set of experiments, we ran the methods until the relative error, $\|X_{k+1}-X_k\|/\|X_k\|$ or consecutive iteration error, $\|X_{k+1}-X_k\|$ is less than the machine precision or a maximum number of iterations~(500) is reached, whichever happens first. Eventually, Figure \ref{Fig:fidelity_2} shows that the performance of \program-$1$ improves as it runs for more global iterations although it has a fractions of the execution time compare to the other methods. Next, in Figure \ref{Fig:Fidelity_3} we ran more samples (50) with the same stopping criteria as before but with an increased number of iterations (1000). From Figure \ref{Fig:Fidelity_3} we find that for fidelity inside the sample, \program-$1$ performs better than the two previously best performing methods---EM and BIRSVD. {Nevertheless, for fidelity outside the sample BIRSVD is still the best, but the behavior of \program~improves compared to the previous cases.}

\begin{figure}[!t]
\centering
\subfloat[ Fidelity within data]{ \includegraphics[width=0.318\linewidth]{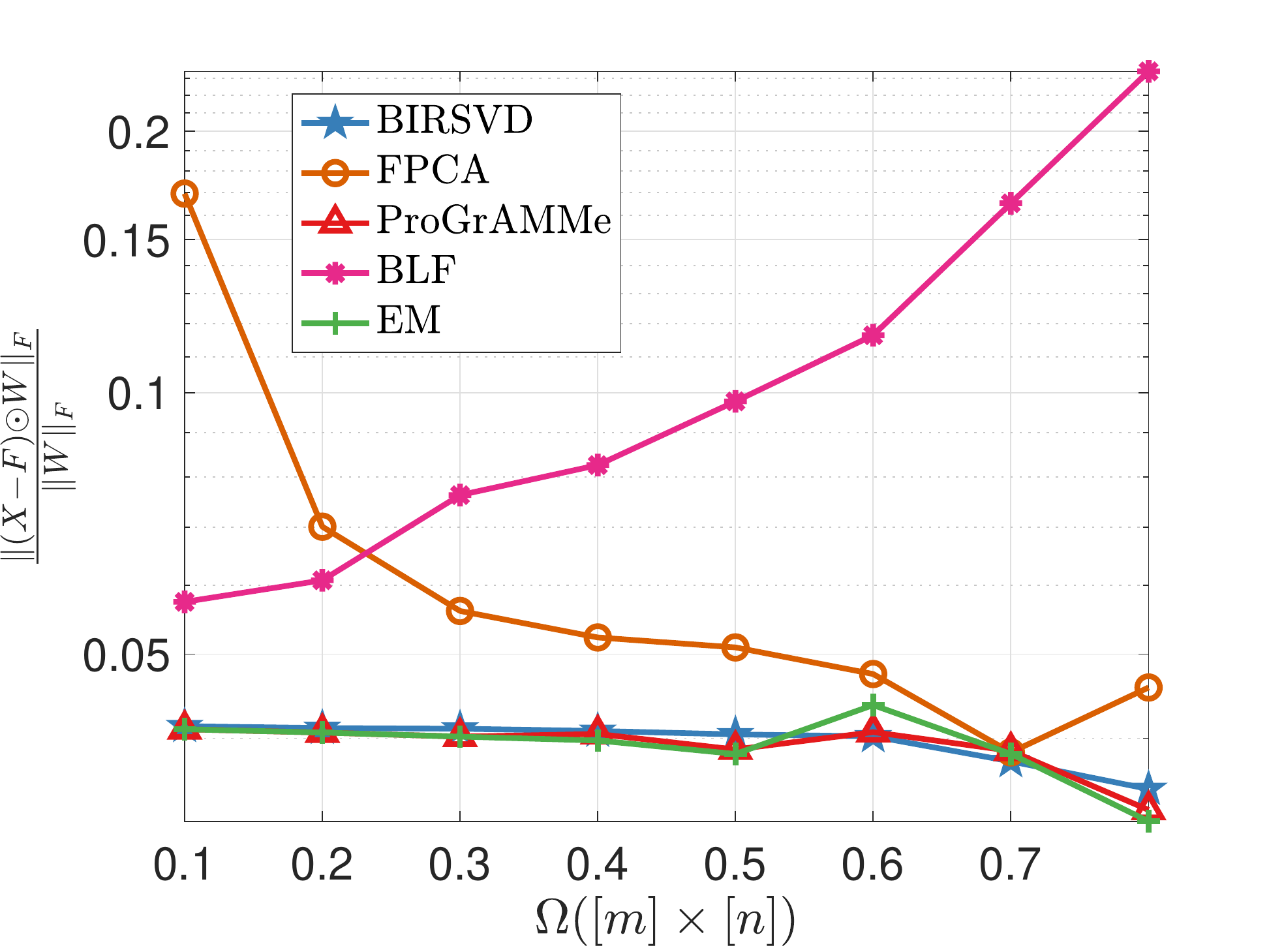} } \hspace{-1.25ex}
\subfloat[ Fidelity out of data]{ \includegraphics[width=0.318\linewidth]{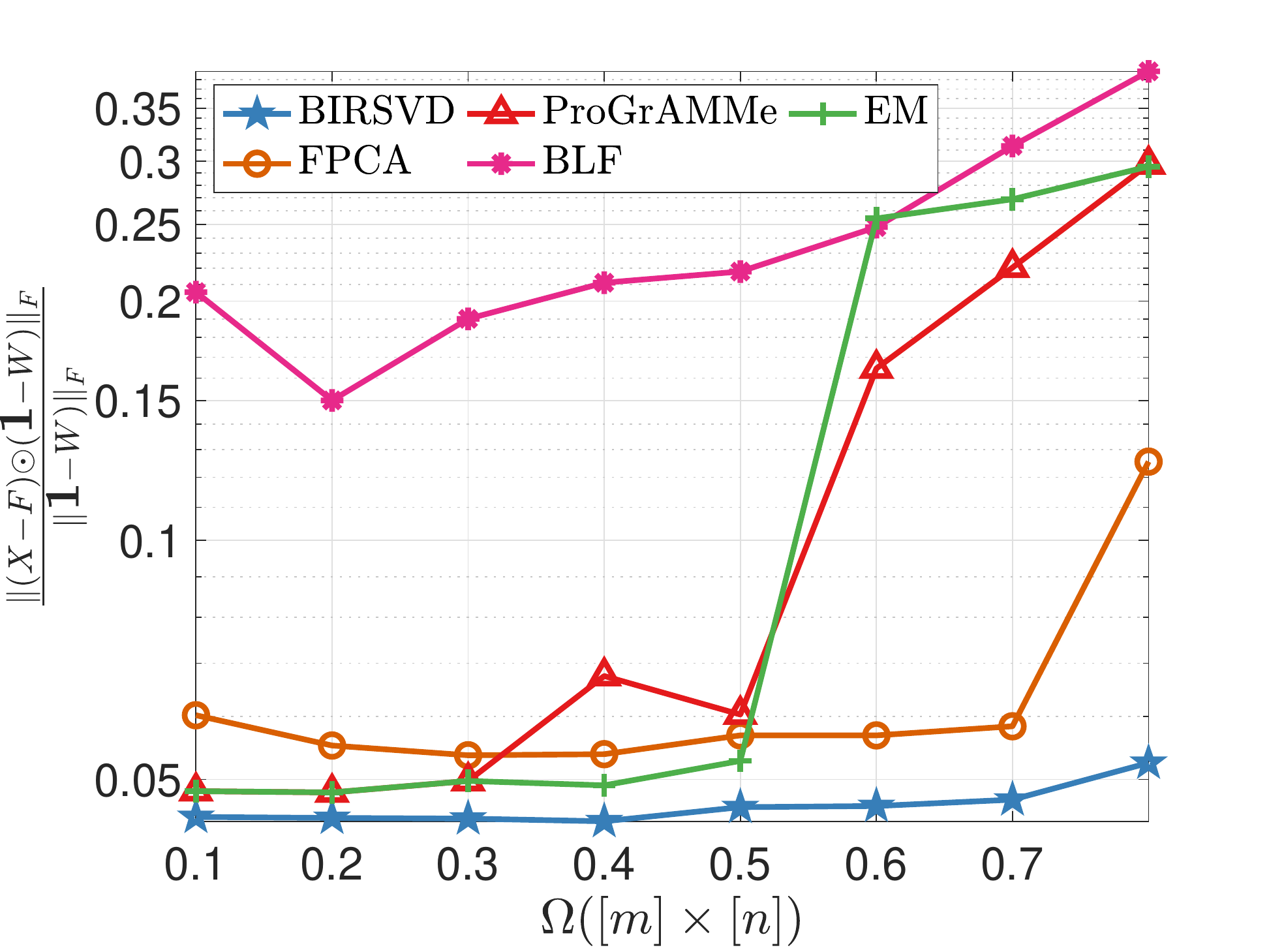} } \hspace{-1.25ex}
\subfloat[ Execution time ]{ \includegraphics[width=0.318\linewidth]{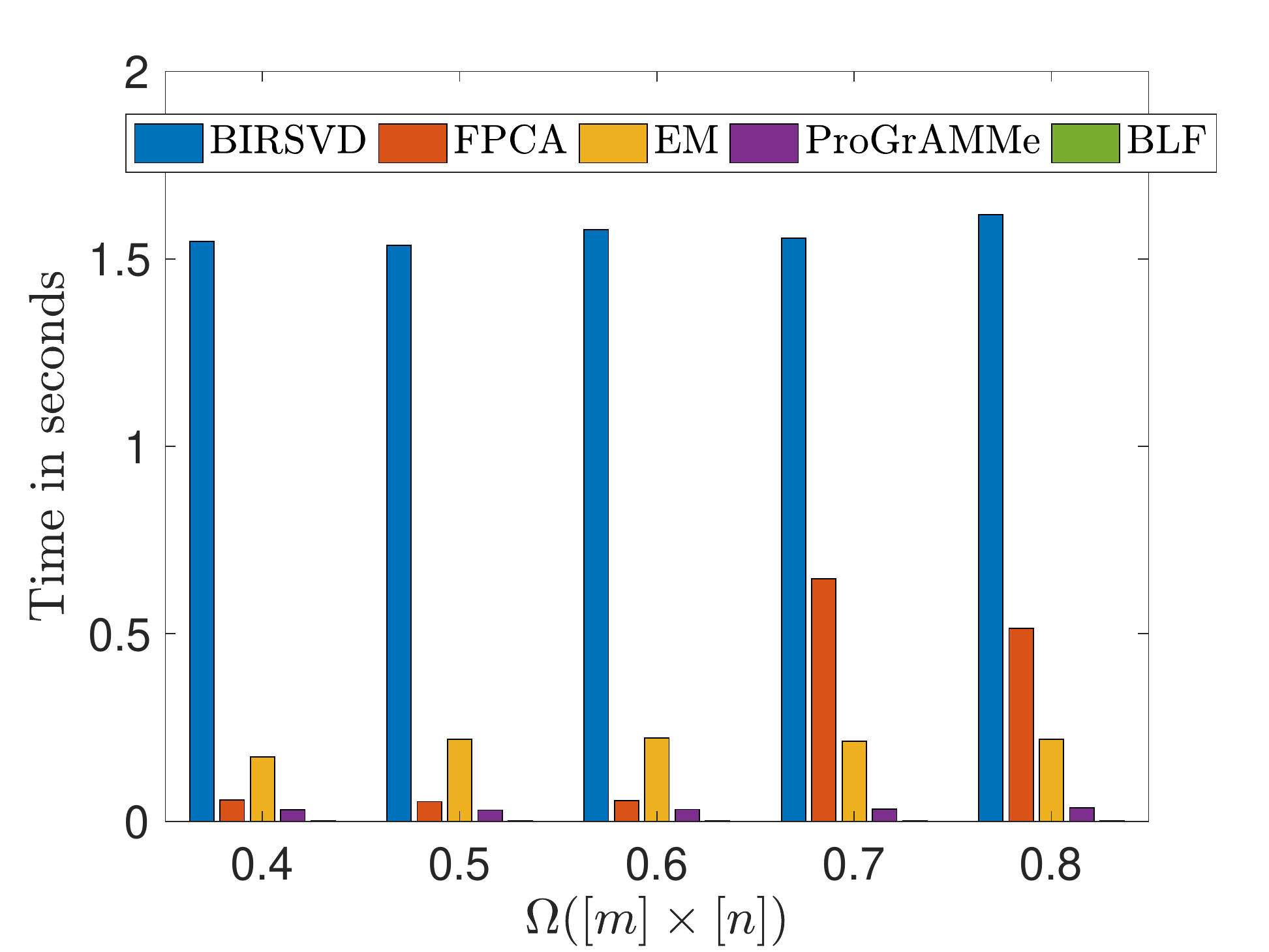} } \\[-2mm]
%%%%%%%%
\caption{{Fidelity within and out of the data, ran for more number of iterations (500).~Here $\Omega$ denotes the percentage of missing data. The last bar diagram shows the execution time of different algorithms for different subsample $\Omega$. Although BLF with $\ell_2$ loss has the least execution time, its performance is not so good for this set of experiments.}}\label{Fig:fidelity_2}
\end{figure}

%\begin{figure}[!ht]
%\centering
%\subfloat[ ]{ \includegraphics[width=0.75\textwidth]{} } \\[-0.5pt]
%\subfloat[ ]{ \includegraphics[width=0.75\textwidth]{} } \\[-1mm]
%    %%%%%%%%
%\caption{...}\label{...}
%\end{figure}

\begin{figure}[!t]
\centering
\subfloat[ Fidelity within the sample ]{ \includegraphics[width=0.45\textwidth]{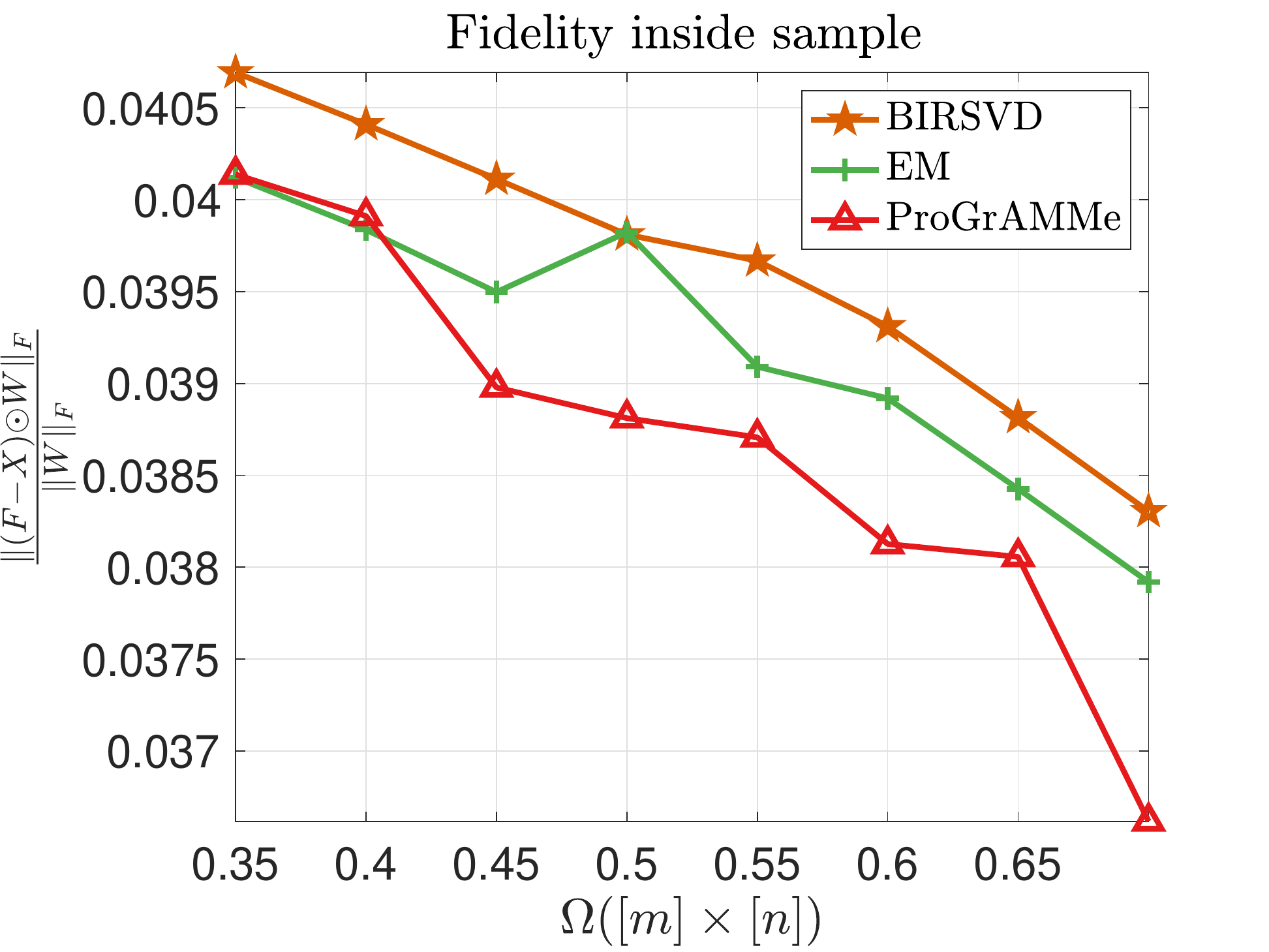} } \hspace{2pt}
\subfloat[ Fidelity out of the sample ]{ \includegraphics[width=0.45\textwidth]{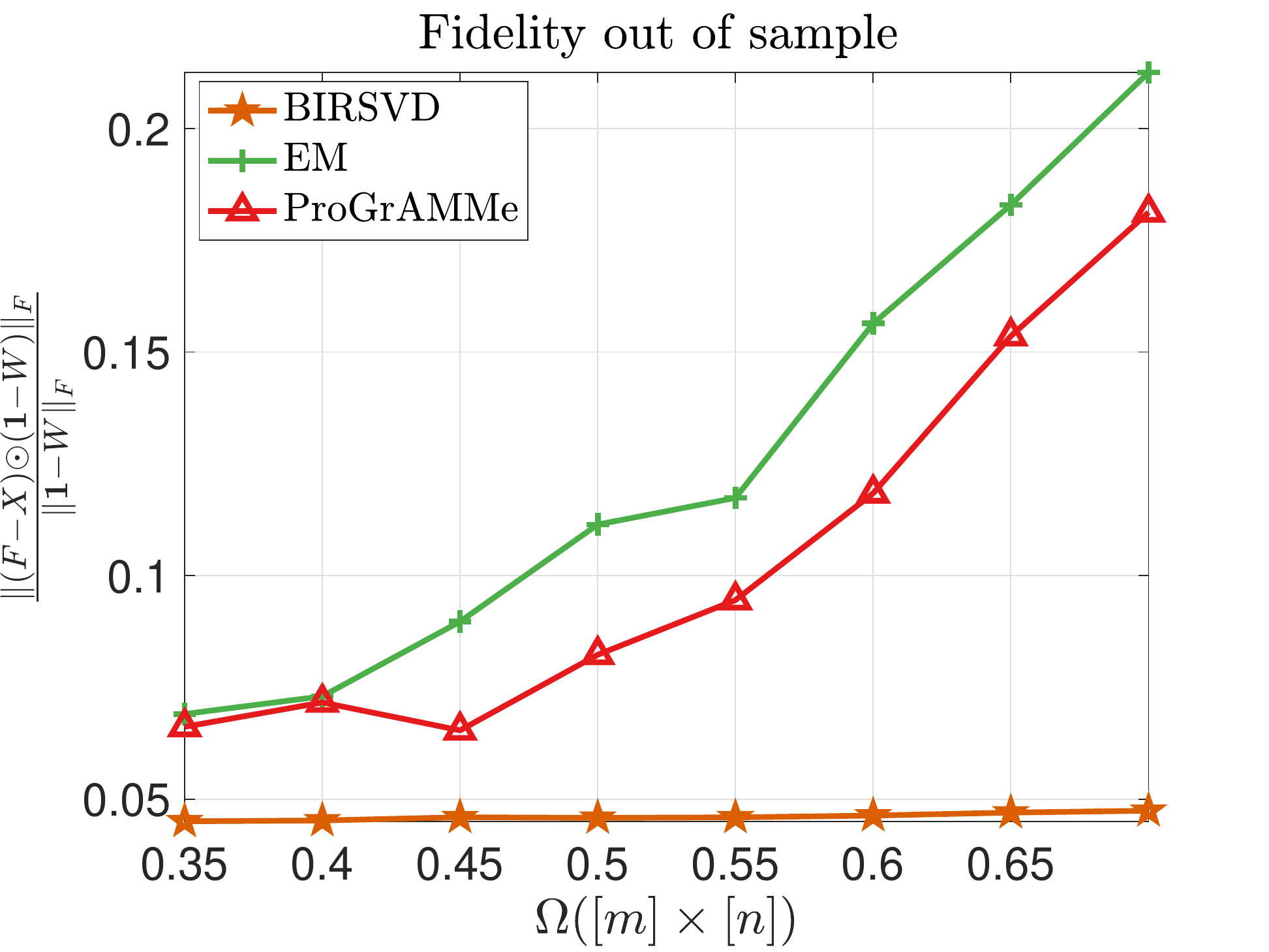} } \\[-2mm]
%%%%%%%
\caption{{Fidelity within and out of the sample, ran for more samples (50) over a higher number of iterations (1000). Here $\Omega$ denotes the percentage of missing data.}}\label{Fig:Fidelity_3}
\end{figure}

\begin{figure}[!t]
    \centering
    \includegraphics[width=\linewidth, height=4in]{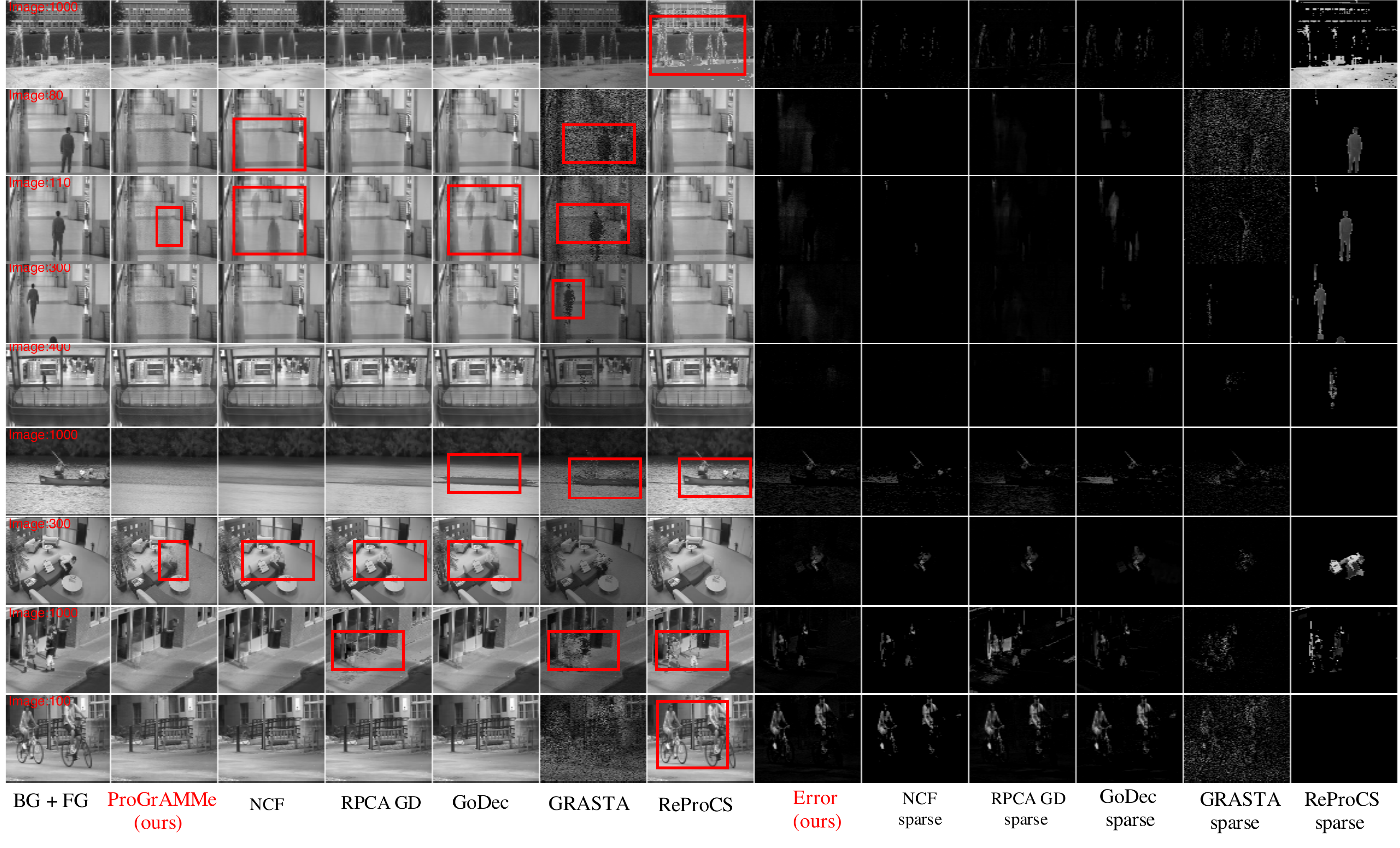}
            %\vspace{-5pt}
\caption{{Sample video frames from the CDNet 2014 and SBI dadasets.~\program~provides a visually high quality background for almost all sequences.~The red bounding boxes in recovered BG denote shadows, ghosting effects of FG objects, static FG object etc.}}\label{Fig:bg_2}
\end{figure}

\begin{figure}
\centering
\subfloat[ Mean SSIM ]{ \includegraphics[width=0.45\textwidth]{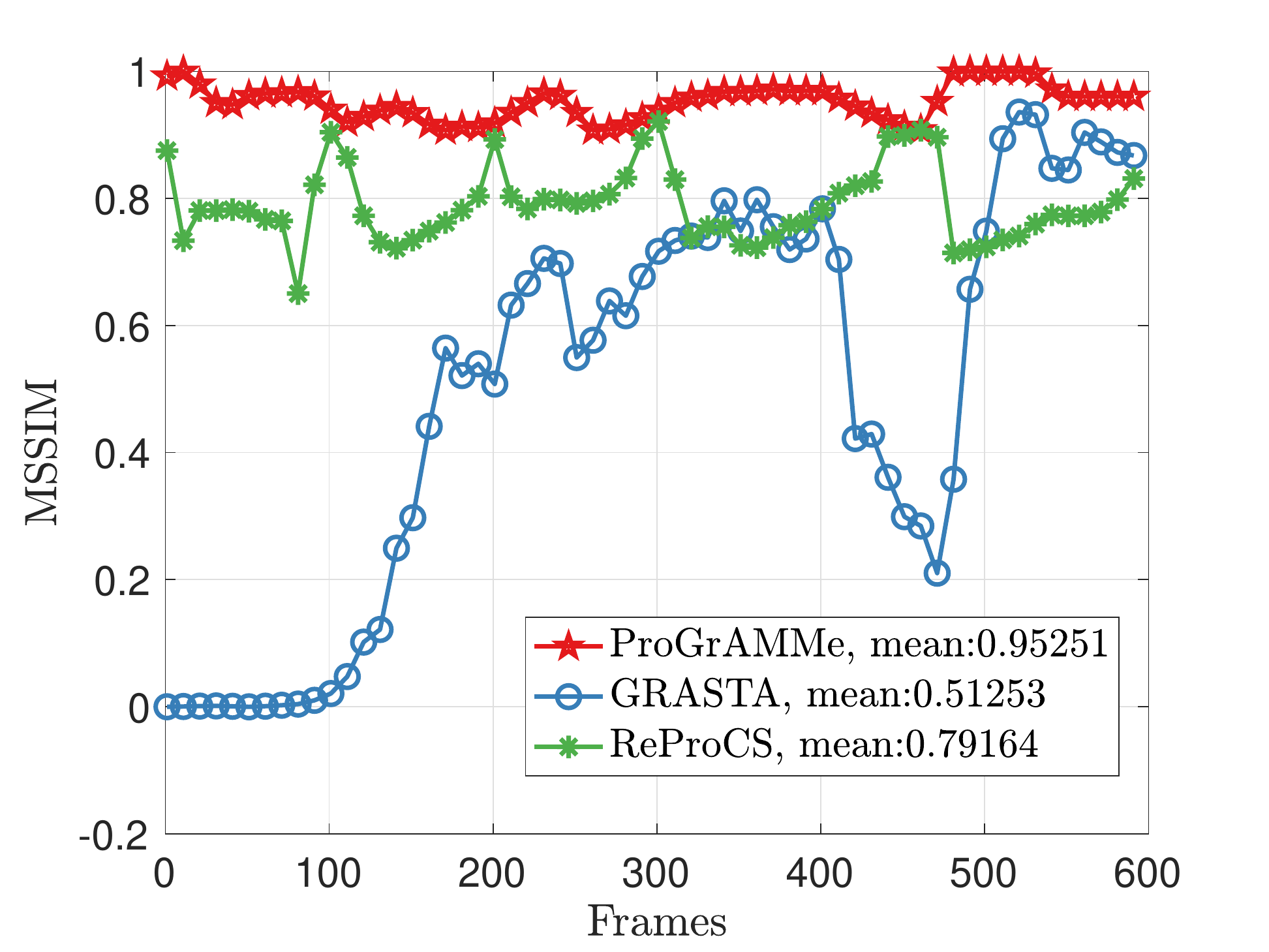} } \hspace{2pt}
\subfloat[ PSNR ]{ \includegraphics[width=0.45\textwidth]{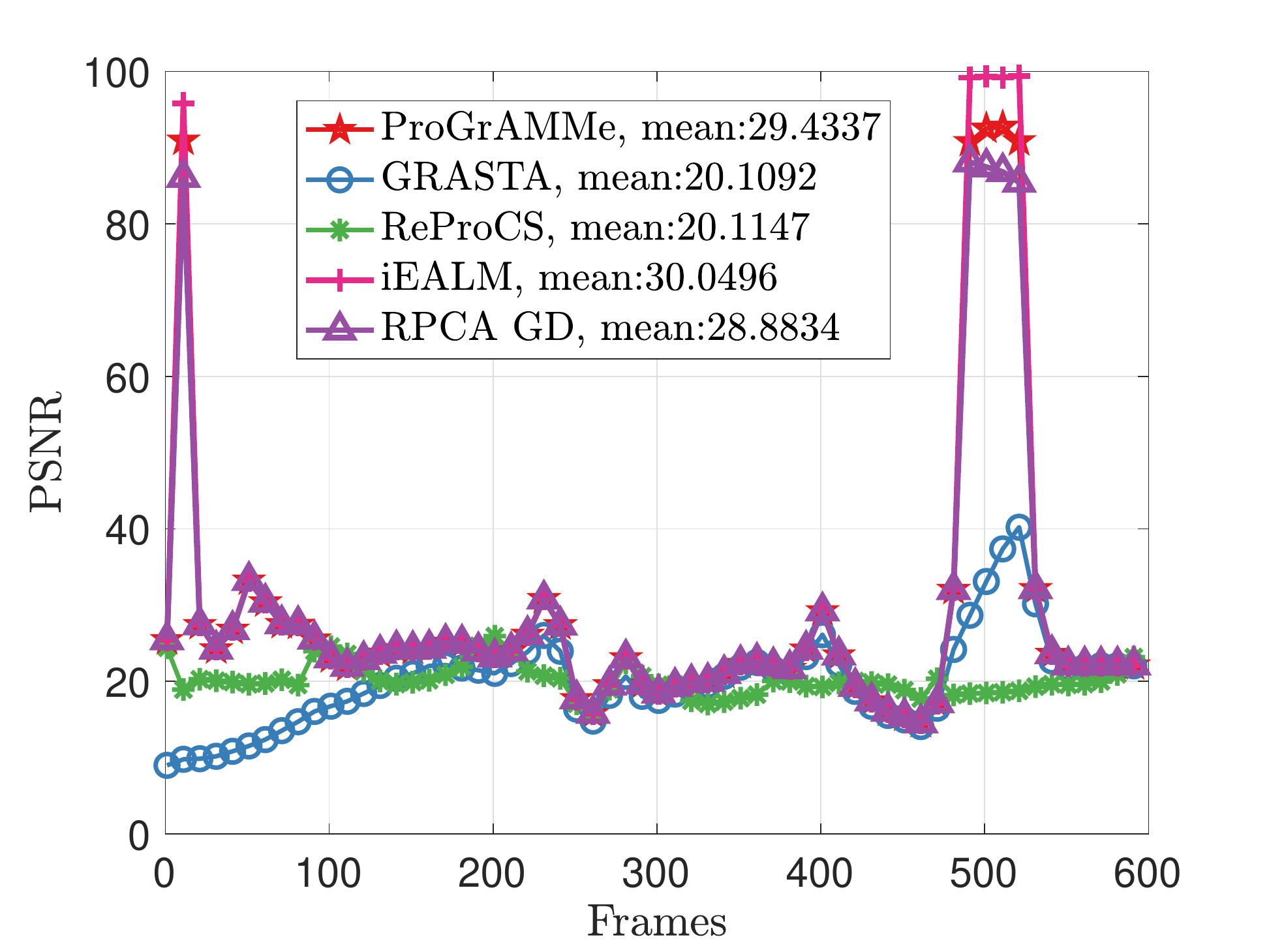} } \\[-2mm]
%%%%%%%%
\caption{{Quantitative comparison of different algorithms on Stuttgart {\tt Basic} sequence. We compare the recovered foreground by different methods with respect to the foreground GT available for each frame on two different metrics: mean SSIM and PSNR.}}\label{fig:quant_bg}
\end{figure}

\begin{figure}[!t]
\centering
\subfloat[ $|\Omega|=0.9(m,n)$ ]{ \includegraphics[width=0.3\linewidth]{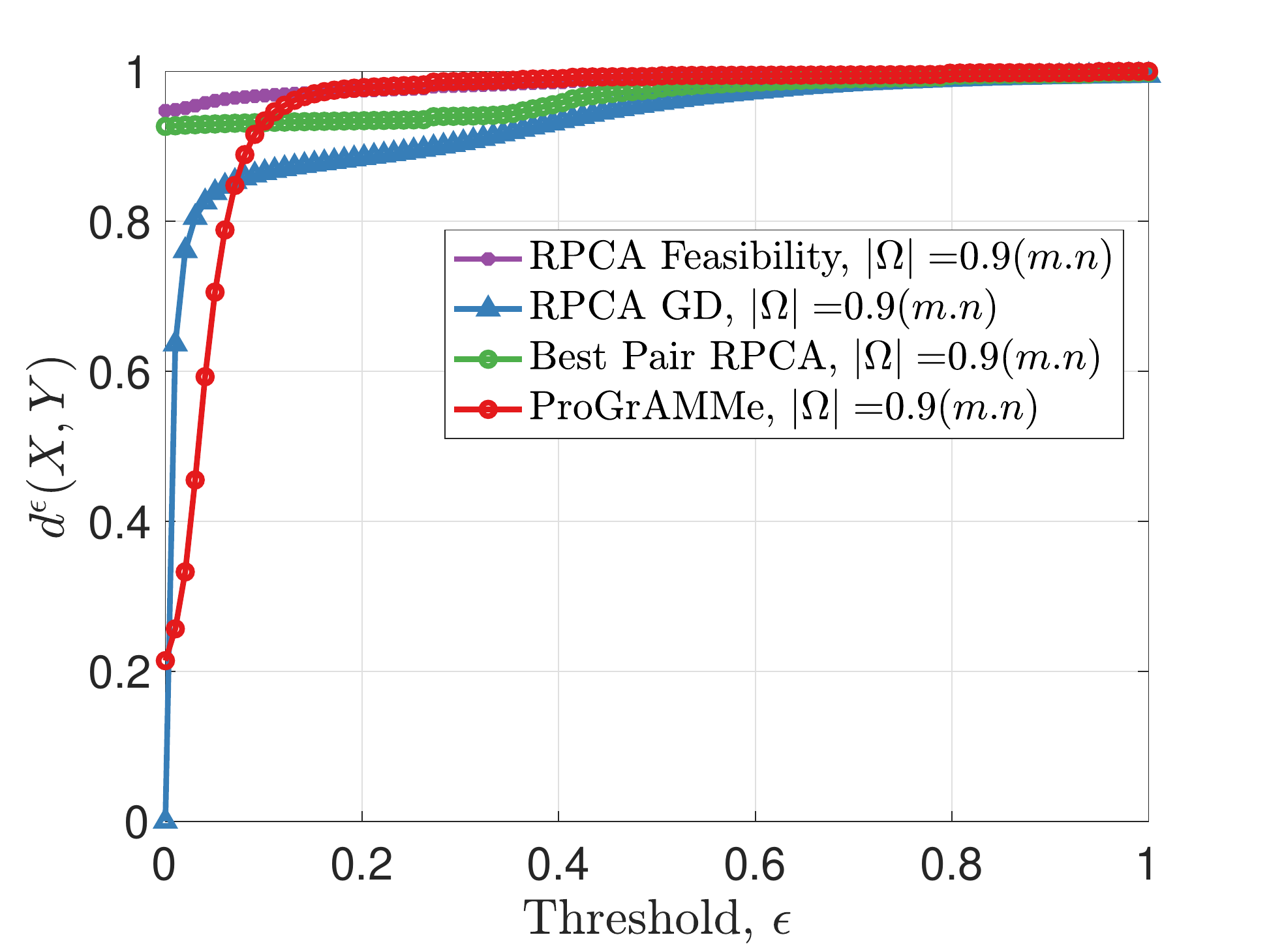} }
\subfloat[ $|\Omega|=0.8(m,n)$ ]{ \includegraphics[width=0.3\linewidth]{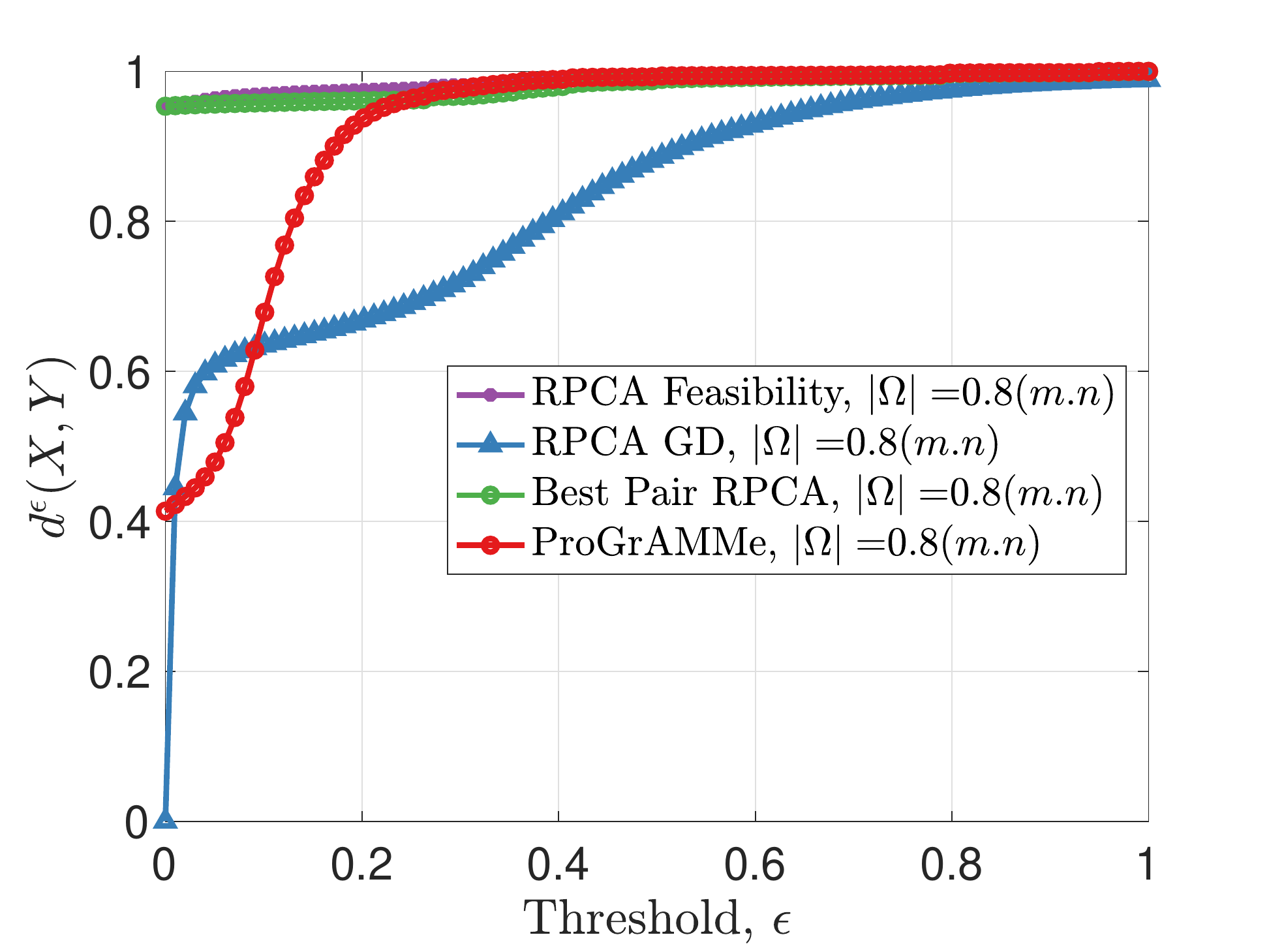} }
\subfloat[ $|\Omega|=0.7(m,n)$ ]{ \includegraphics[width=0.3\linewidth]{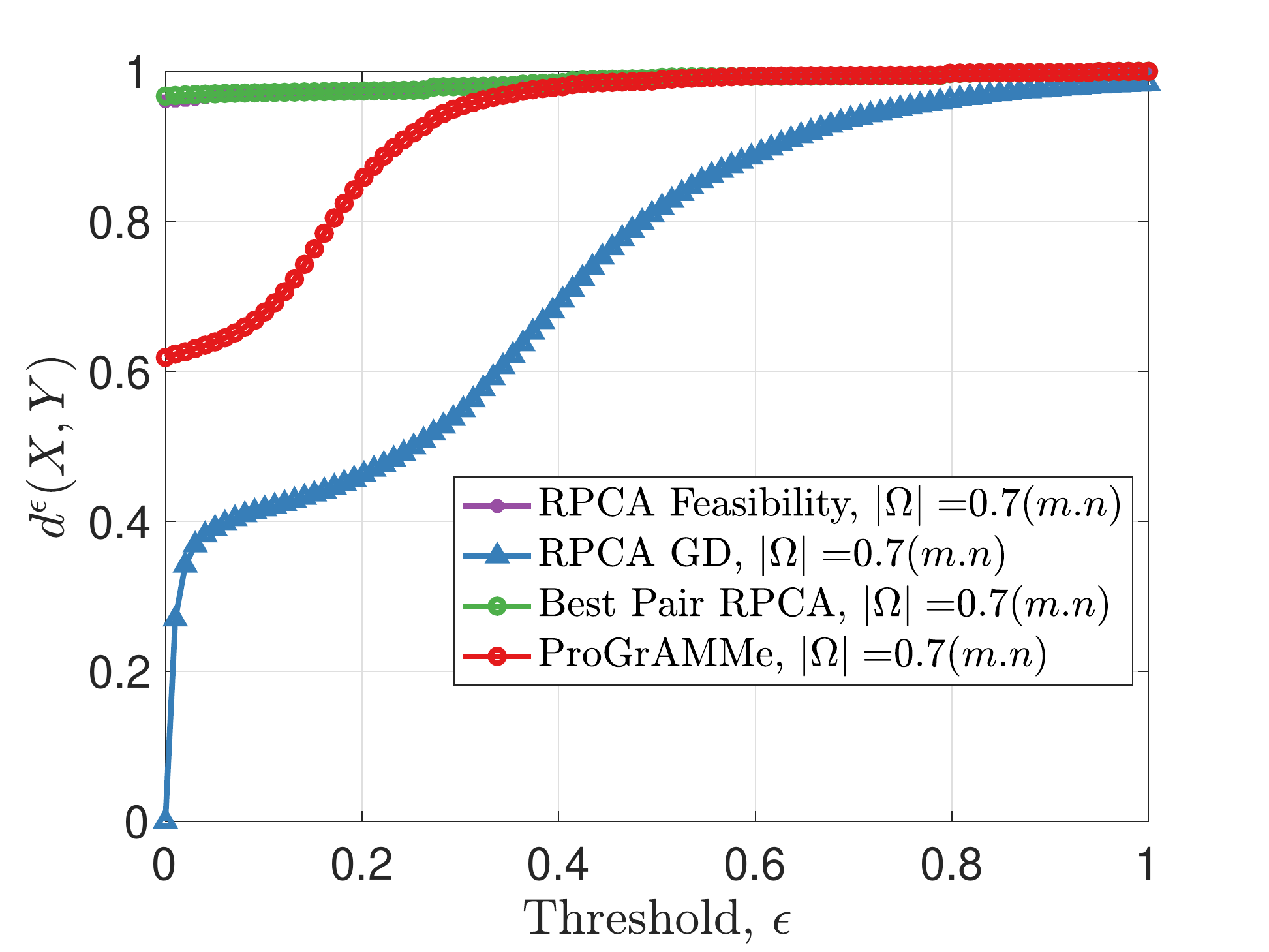} } \\[-2mm]
	%%%%%%%%
\subfloat[ $|\Omega|=0.6(m,n)$ ]{ \includegraphics[width=0.3\linewidth]{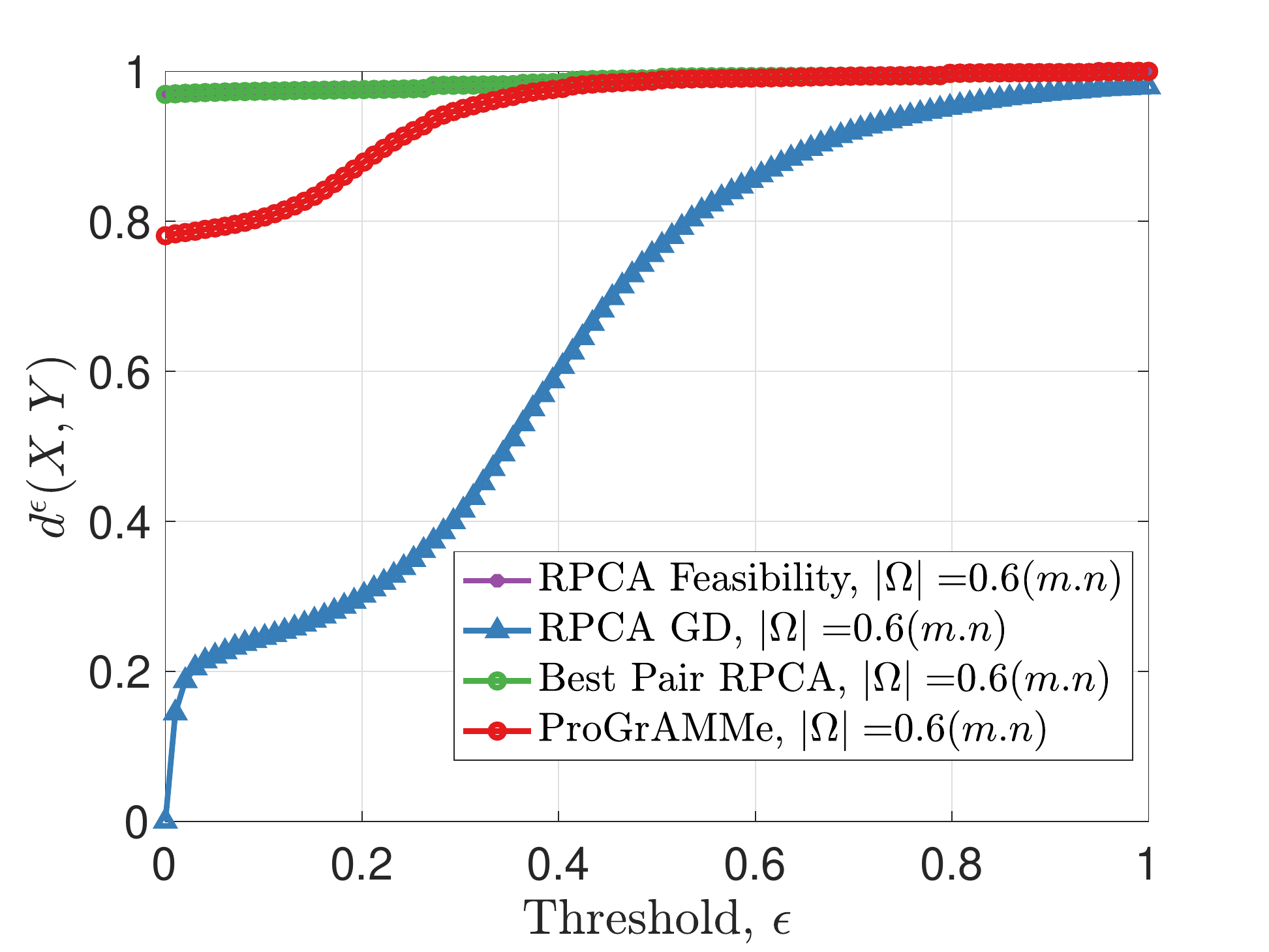} }
\subfloat[ $|\Omega|=0.5(m,n)$ ]{ \includegraphics[width=0.3\linewidth]{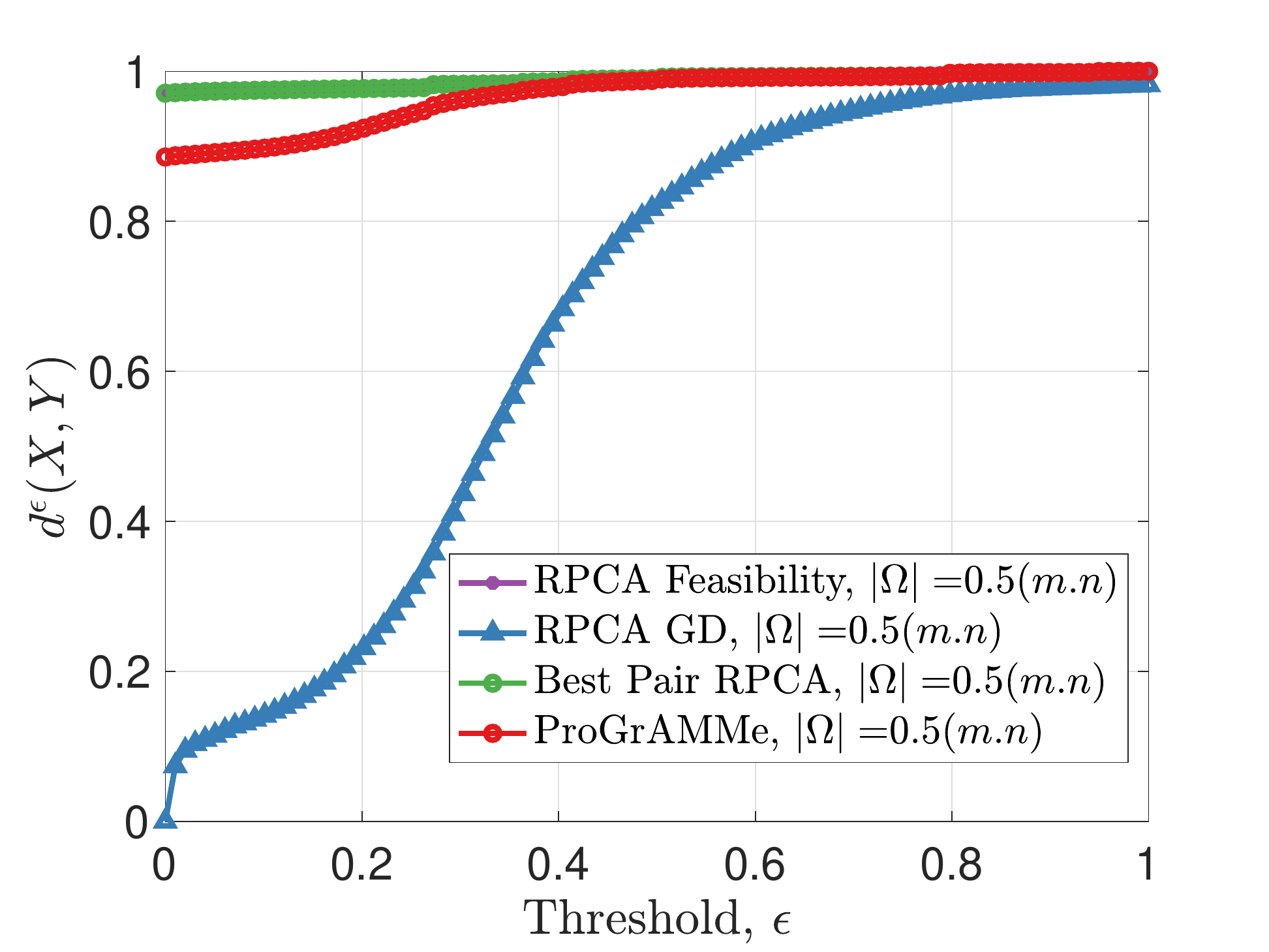} }
\subfloat[ $|\Omega|=0.4(m,n)$ ]{ \includegraphics[width=0.3\linewidth]{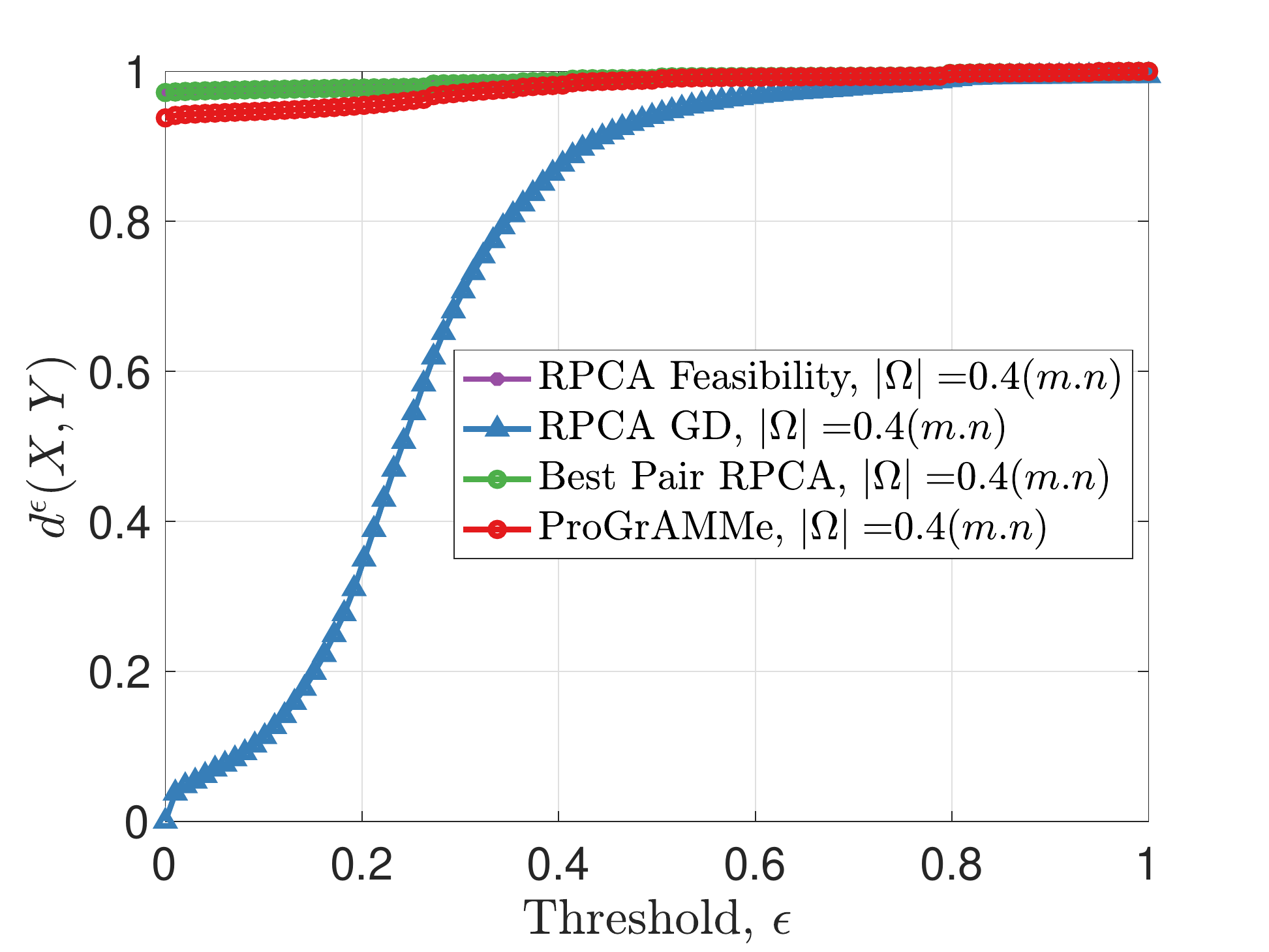} } \\[-2mm]
%%%%%%%%
\subfloat[ $|\Omega|=0.3(m,n)$ ]{ \includegraphics[width=0.35\linewidth]{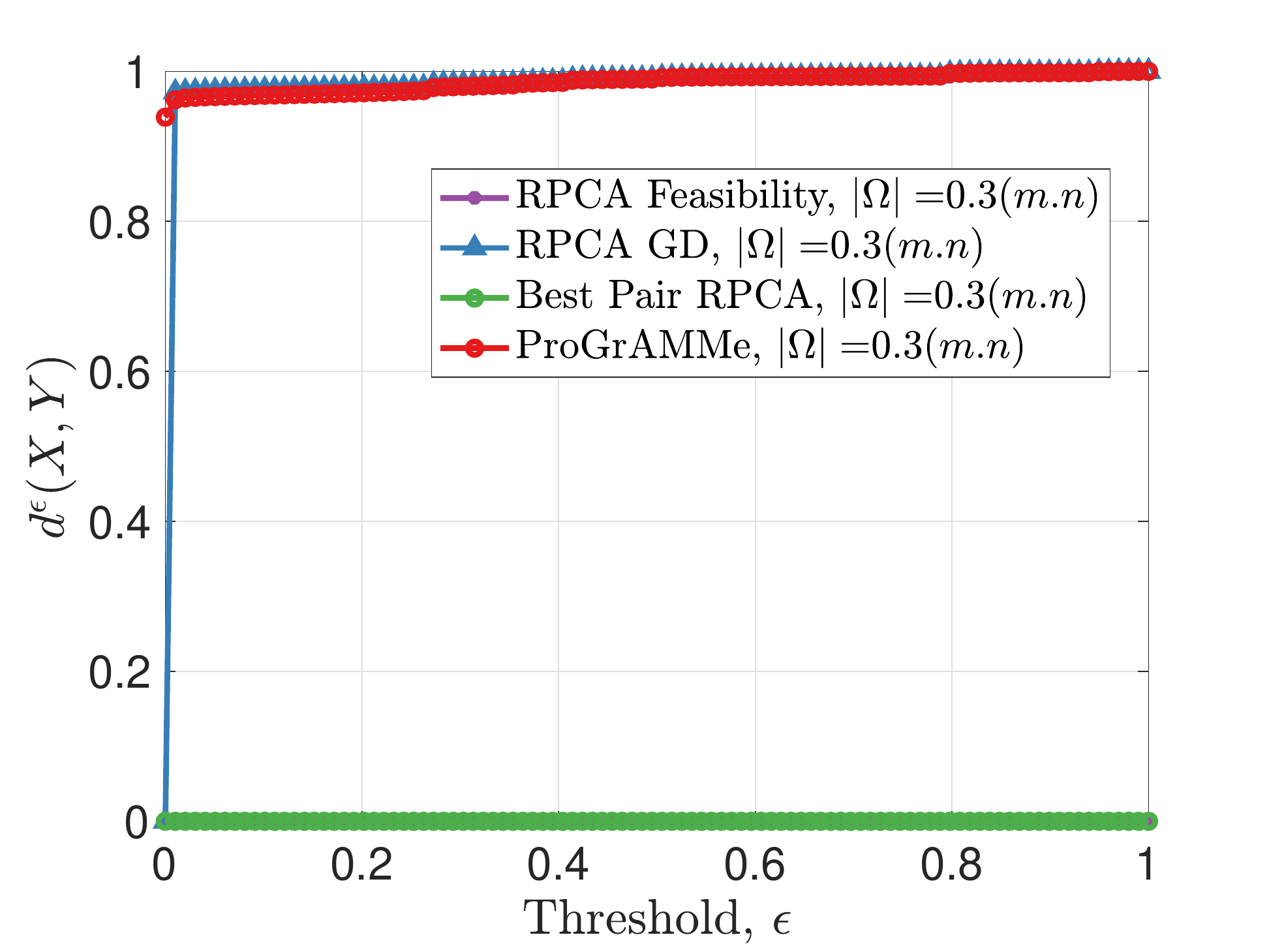} } \hspace{2pt}
\subfloat[ Execution time ]{ \includegraphics[width=0.35\linewidth]{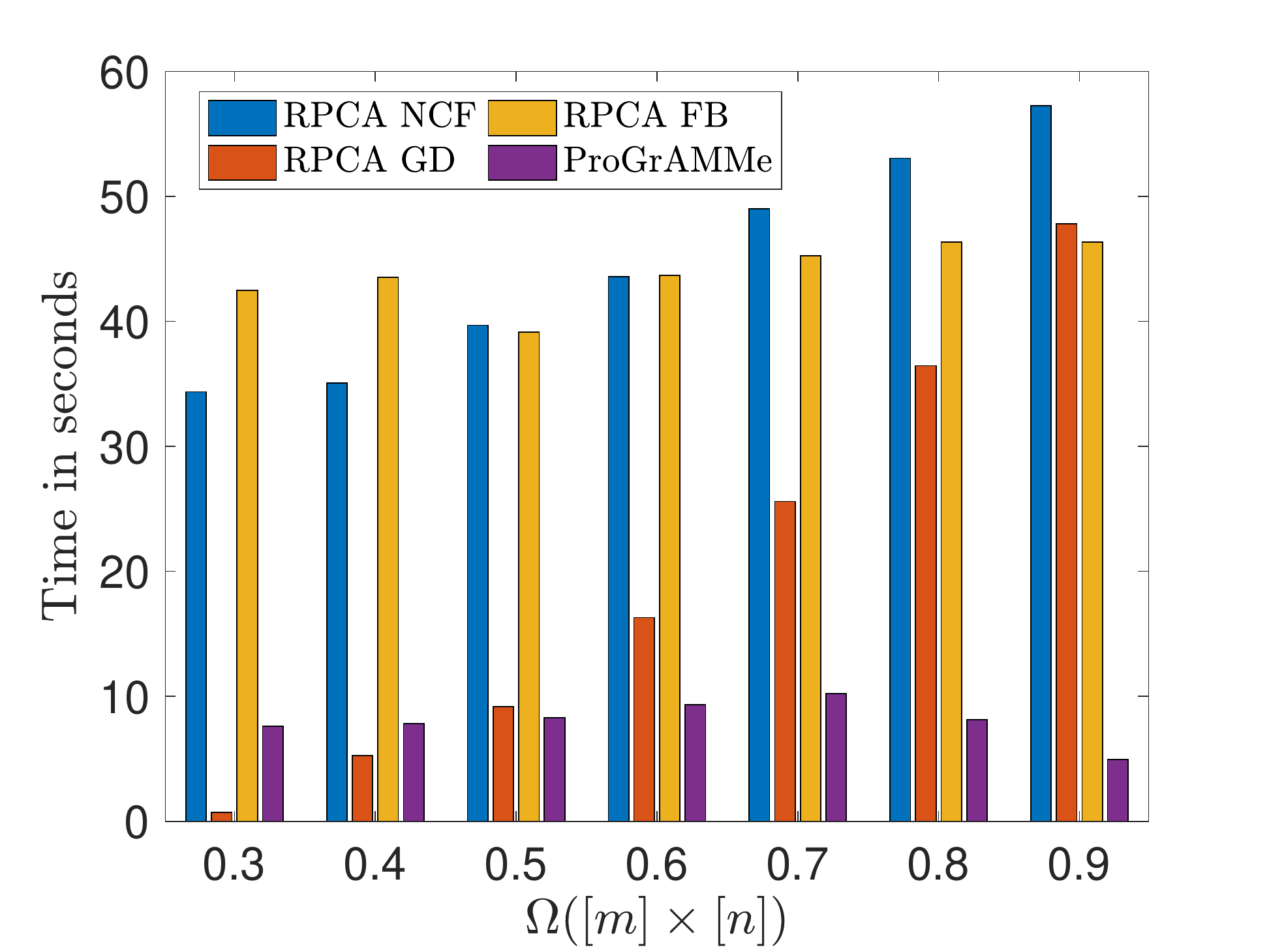} } \\[-2mm]
%%%%%%%%%\
%\vspace{-1ex}
\caption{{Quantitative comparison between different algorithms on Stuttgart {\tt Basic} sequence for different levels of partially observed/missing data case with respect to the $d_{\epsilon}(X,Y)$ metric.~In Figure (h), the bar diagram shows the execution time of different algorithms for different subsample $\Omega$.~ProGrAMMe-$1$ has the least average execution time in all scenarios.}}\label{Fig:missng_data_bg}
\end{figure}

%\begin{figure}[!ht]
%\centering
%\begin{minipage}{0.33\textwidth}
%    \centering
%    \includegraphics[width=\linewidth]{BG_partial/missing_data_1.eps}
%\end{minipage}%
%\begin{minipage}{0.33\textwidth}
%    \centering
%    \includegraphics[width=\linewidth]{BG_partial/missing_data_2.eps}
%\end{minipage}
%\begin{minipage}{0.33\textwidth}
%    \centering
%    \includegraphics[width=\linewidth]{BG_partial/missing_data_3.eps}
%\end{minipage}\\
%\centering
%\begin{minipage}{0.33\textwidth}
%    \centering
%    \includegraphics[width=\linewidth]{BG_partial/missing_data_4.eps}
%\end{minipage}%
%\begin{minipage}{0.33\textwidth}
%    \centering
%    \includegraphics[width=\linewidth]{BG_partial/missing_data_5.eps}
%\end{minipage}
%\begin{minipage}{0.33\textwidth}
%    \centering
%    \includegraphics[width=\linewidth]{BG_partial/missing_data_6.eps}
%\end{minipage}\\
%\centering
%\begin{minipage}{0.38\textwidth}
%    \centering
%    \includegraphics[width=\linewidth]{BG_partial/missing_data_7.eps}
%\end{minipage}%
%\begin{minipage}{0.42\textwidth}
%    \centering
%    \includegraphics[width=\linewidth]{BG_partial/time_missing_data_new.eps}
%\end{minipage}
%%%%%%%%%%\
%\vspace{-1ex}
%\caption{{Quantitative comparison between different algorithms on Stuttgart {\tt Basic} sequence for partially observed/missing data case with respect to the $d_{\epsilon}(X,Y)$ metric.~The last bar diagram shows the execution time of different algorithms for different subsample $\Omega$.~ProGrAMMe-$1$ has the least average execution time in all scenarios.}}\label{Fig:missng_data_bg}
%\end{figure}

\subsection{Background estimation--Additional qualitative and quantitative evaluation}\label{sec:appendix_BG}
We show the qualitative results on 7 other real-world video sequences from the CDNet 2014 and SBI datasets in Figure \ref{Fig:bg_2}.~In almost all sequences,  \program-$\epsilon$ performs consistently well compare to the other state-of-the-art methods.~We do not include iEALM or APG due to their higher execution time. 

In Figure \ref{fig:quant_bg}, we show two robust quantitative measures for the background estimation experiments on Stuttgart {\tt Basic} video: peak signal to noise ratio~(PSNR) and mean structural similarities index measure~(SSIM) \cite{mssim}. PSNR is defined as $10\log_{10}$ of the ratio of the peak signal energy to the mean square error~(MSE) between the processed video signal and the ground truth. Let $F(:,i)-\hat{X}(:,i)$ be the reconstructed vectorized foreground frame and $G(:,i)$ be the corresponding ground truth frame, then PSNR is defined as $10\log_{10}\frac{{\rm M}_I^2}{{\rm MSE}}$, where ${\rm MSE}= \frac{1}{mn}\|F(:,i)-\hat{X}(:,i)-G(:,i)\|^2$ and ${\rm M}_I=255$ is the maximum possible pixel value of the image, as the pixels are represented using 8 bits per sample. For a reconstructed image with 8 bits bit depth, the PSNR are between 30 and 50 dB, where the higher is the better as we minimize the MSE between images with respect the maximum signal value of the image. 

For both measures, we perceive the information how the high-intensity regions of the image are coming through the noise, and pay less attention to the low-intensity regions. We remove the noisy components from the recovered foreground, $F(:,i)-\hat{X}(:,i)$, by using the threshold $\epsilon'$, such that we set the components below $\epsilon'$ in $E$ to 0. In our experiments, we set $\epsilon'=10^{-4}$. To calculate the SSIM of each recovered foreground video frame, we consider an $11\times 11$ Gaussian window with standard deviation~($\sigma$) 1.5 and consider the corresponding ground truth as the reference image. Among the methods tested, \program-$\epsilon$ has the highest average SSIM~(or MSSIM).~To compare PSNR of recovered foreground frames, we use \program-$\epsilon$, GRASTA \cite{grasta}, recursive projected compressive sensing~(ReProCS)\cite{reprocs}, inexact ALM~(iEALM) \cite{LinChenMa}, and RPCA GD~\cite{RPCAgd}. iEALM has the highest average PSNR, 30.05 dB among all the methods, whereas \program-$\epsilon$ has albeit less, an average PSNR 29.45 dB. However, \program-$\epsilon$ needs an average 9.48 seconds to produce the results, compared to the average execution time of iEALM is 183 seconds. 

\subsection{Background estimation from partially observed/missing data---Quantitative evaluation}\label{sec:appendix-missingdataBG}
For background estimation on partially observed data we used the  $\epsilon$-proximity metric---$d_{\epsilon}(X,Y)$ proposed in \cite{duttahanzely} on Stuttgart {\tt Basic} video. The performance of RPCA nonconvex feasibility (RPCA feasibility or RPCA NCF) \cite{duttahanzely} with respect to $d_{\epsilon}(X,Y)$ stays stable for all subsample $\Omega$. The performance of the best pair RPCA (also known as RPCA forward-backward or RPCA FB) \cite{dutta2019best} is stable except for $\Omega=0.3$. The performance of RPCA GD \cite{RPCAgd} keeps downgrading as we decrease the cardinality of the support $\Omega$. Surprisingly, the performance of \program-$1$ gets better for this experiment as we decrease $\Omega$. Furthermore, the average execution time for \program-$1$ is stable for different $\Omega$, and is around 8 seconds. While the next best average execution time 19.67 seconds is recorded for RPCA GD. The average execution time of NCF and best pair are 44 and 43 seconds, respectively.

\subsection*{Acknowledgement}
Aritra Dutta acknowledges being an affiliated researcher at the Pioneer Centre for AI, Denmark. Jingwei Liang acknowledges support from the Shanghai Municipal Science and Technology Major Project (2021SHZDZX0102) and the support from SJTU and Huawei ExploreX Funding (SD6040004/033). 

%%%%%%%%%%%%%%%%%%%%%%%%%%%%%%%%%%%%%%%%%%%%%%%%%%%%%%%%%%%%%%%%%%%%

\begin{small}
\bibliographystyle{plain}
\bibliography{references}
\end{small}

\end{document}